\documentclass[a4paper, reqno, 11pt]{amsart}
\usepackage[utf8]{inputenc}

\usepackage{amsthm}
\usepackage{amssymb}
\usepackage{amsxtra}
\usepackage{exscale}
\usepackage[mathscr]{euscript}
\usepackage{url}
\usepackage[all,ps,tips,tpic]{xy}
\usepackage{graphicx}
\usepackage{color}

\usepackage{mathrsfs}

\RequirePackage{xspace}
\RequirePackage{etoolbox}
\RequirePackage{varwidth}
\RequirePackage{enumitem}
\RequirePackage{tensor}
\RequirePackage{mathtools}
\RequirePackage{longtable}
\RequirePackage{multirow}

\newenvironment{altenumerate}
   {\begin{list}
      {\textup{(\theenumi)} }
      {\usecounter{enumi}
       \setlength{\labelwidth}{0pt}
       \setlength{\labelsep}{2pt}
       \setlength{\leftmargin}{0pt}
       \setlength{\itemsep}{\the\smallskipamount}
       \renewcommand{\theenumi}{\roman{enumi}}
      }}
   {\end{list}}

\newtheorem{lem}{Lemma}[section]
\newtheorem{lemma}[lem]{Lemma}
\newtheorem{definition}[lem]{Definition}

\newtheorem{cor}[lem]{Corollary}
\newtheorem{corollary}[lem]{Corollary}
\newtheorem{thm}[lem]{Theorem}
\newtheorem{theorem}[lem]{Theorem}
\newtheorem{prop}[lem]{Proposition}
\newtheorem{proposition}[lem]{Proposition}

\newtheorem{hypothesis}[lem]{Hypothesis}

\theoremstyle{remark}
\newtheorem{rem}[lem]{Remark}
\newtheorem{remark}[lem]{Remark}
\newtheorem{remarks}[lem]{Remarks}
\newtheorem{example}[lem]{Example}
\newtheorem{question}[lem]{Question}

\DeclareMathOperator{\Hom}{Hom}

\DeclareMathOperator{\coker}{coker}

\DeclareMathOperator{\diag}{diag}
\DeclareMathOperator{\Spa}{Spa}
\DeclareMathOperator{\Spec}{Spec}

\DeclareMathOperator{\Gal}{Gal}
\DeclareMathOperator{\End}{End}

\DeclareMathOperator{\Map}{Map}
\DeclareMathOperator{\Frob}{Frob}
\DeclareMathOperator{\Ann}{Ann}
\DeclareMathOperator{\twoinjlim}{{2-}\!\varinjlim}

\def\ad{\mathrm{ad}}
\def\et{\mathrm{\acute{e}t}}
\def\fet{\mathrm{f\acute{e}t}}

\def\cont{\mathrm{cont}}

\def\cycl{\mathrm{cycl}}

\def\LT{\mathrm{LT}}
\def\Dr{\mathrm{Dr}}
\def\GH{\mathrm{GH}}

\def\HT{\mathrm{HT}}
\def\reg{\mathrm{reg}}
\def\patch{\mathrm{patch}}
\def\comp{\mathrm{comp}}

\newcommand{\cO}{\mathcal{O}}
\newcommand{\cM}{\mathcal{M}}
\newcommand{\N}{\mathbb{N}}
\newcommand{\Z}{\mathbb{Z}}
\newcommand{\F}{\mathbb{F}}
\newcommand{\Q}{\mathbb{Q}}
\newcommand{\R}{\mathbb{R}}
\newcommand{\A}{\mathbb{A}}

\newcommand{\C}{\mathbb{C}}
\newcommand{\G}{\mathbb{G}}

\newcommand{\T}{\mathbb{T}}
\newcommand{\bS}{\mathbb{bS}}
\newcommand{\mm}{\mathfrak{m}}
\newcommand{\frap}{\mathfrak{p}}
\newcommand{\tr}{\operatorname*{tr}}
\newcommand{\GL}{\mathrm{GL}}
\newcommand{\SL}{\mathrm{SL}}
\newcommand{\PGL}{\mathrm{PGL}}

\newcommand{\Res}{\mathrm{Res}}
\newcommand{\Sh}{\mathrm{Sh}}
\newcommand{\proj}{\mathrm{proj}}

\newdir{ >}{{}*!/-5pt/@{>}}

\SelectTips{cm}{10}

\begin{document}
\title{On the $p$-adic cohomology of the Lubin-Tate tower}
\author{Peter Scholze}
\begin{abstract}
We prove a finiteness result for the $p$-adic cohomology of the Lubin-Tate tower. For any $n\geq 1$ and $p$-adic field $F$, this provides a canonical functor from admissible $p$-adic representations of $\GL_n(F)$ towards admissible $p$-adic representations of $\Gal_F\times D^\times$, where $\Gal_F$ is the absolute Galois group of $F$, and $D/F$ is the central division algebra of invariant $1/n$.

Moreover, we verify a local-global-compatibility statement for this correspondence, and compatibility with the patching construction of Caraiani-Emerton-Gee-Geraghty-Paskunas-Shin.
\end{abstract}

\date{\today}
\maketitle
\tableofcontents
\pagebreak

\section{Introduction}

The goal of this paper is to provide further evidence for the existence of a $p$-adic local Langlands correspondence, as was first envisioned by Breuil, \cite{Breuil}, and established for $\GL_2(\Q_p)$ by Colmez, \cite{Colmez}, Paskunas, \cite{Paskunas}, and others. So far, little is known beyond $\GL_2(\Q_p)$, and work of Breuil-Paskunas, \cite{BreuilPaskunas}, shows that already for $\GL_2(F)$, $F\neq \Q_p$, the situation is very difficult. There is recent work of Caraiani-Emerton-Gee-Geraghty-Paskunas-Shin, \cite{CEGGPS}, that construct {\it some} $p$-adic $\GL_n(F)$-representation starting from an $n$-dimensional representation of the absolute Galois group of a $p$-adic field $F$, for general $n$ and $F$. Their construction is based on the patching construction of Taylor-Wiles, and is thus global in nature. Unfortunately, it is not clear that their construction gives a representation independent of the global situation.

In this paper, we work in the opposite direction. Namely, starting from a $p$-adic $\GL_n(F)$-representation $\pi$, we produce a representation $F(\pi)$ of the absolute Galois group $\Gal_F$, for any $n$ and $F$, in a purely local way. Corollary \ref{PatchingMainThm} ensures that (for $n=2$), composing the patching construction with our functor gives back the original Galois representation.

Actually, $F(\pi)$ also carries an admissible $D^\times$-action, where $D/F$ is the central division algebra of invariant $1/n$. Thus, simultaneously, this indicates the existence of a $p$-adic Jacquet-Langlands correspondence relating $p$-adic $\GL_n(F)$ and $D^\times$-representations. Such a correspondence is not known already for $\GL_2(\Q_p)$, and its formalization remains mysterious, as the $D^\times$-representations are necessarily (modulo $p$) of infinite length. However, we do not pursue these questions here.

Let us now describe our results in more detail. Let $n\geq 1$ be an integer and $F/\Q_p$ a finite extension. Let $\cO\subset F$ be the ring of integers, $\varpi\in \cO$ a uniformizer, and let $q$ be the cardinality of the residue field of $F$, which we identify with $\F_q$. Fix an algebraically closed extension $k$ of $\F_q$, e.g. $\overline{\F}_q$. Let $\breve{F}=F\otimes_{W(\F_q)} W(k)$ be the completion of the unramified extension of $F$ with residue field $k$. Let $\breve{\cO}\subset \breve{F}$ be the ring of integers.

In this situation, one has the Lubin-Tate tower $(\cM_{\LT,K})_{K\subset \GL_n(F)}$, which is a tower of smooth rigid-analytic varieties $\cM_{\LT,K}$ over $\breve{F}$ parametrized by compact open subgroups $K$ of $\GL_n(F)$, with finite \'etale transition maps. There is a compatible continuous action of $D^\times$ on all $\cM_{\LT,K}$, as well as an action of $\GL_n(F)$ on the tower, i.e. $g\in \GL_n(F)$ induces an isomorphism between $\cM_{\LT,K}$ and $\cM_{\LT,g^{-1} K g}$. There is the Gross-Hopkins period map, \cite{GrossHopkins},
\[
\pi_\GH: \cM_{\LT,K}\to \mathbb{P}^{n-1}_{\breve{F}}\ ,
\]
compatible for varying $K$, which is an \'etale covering map of rigid-analytic varieties with fibres $\GL_n(F)/K$. It is also $D^\times$-equivariant if the right-hand side is correctly identified with the Brauer-Severi variety for $D/F$ (which splits over $\breve{F}$). Moreover, there is a Weil descent datum on $\cM_{\LT,K}$, under which $\pi_\GH$ is equivariant for the above identification of $\mathbb{P}^{n-1}_{\breve{F}}$ with the Brauer-Severi variety of $D/F$.

It was first observed by Weinstein, cf. \cite{ScholzeWeinstein}, that the inverse limit
\[
\cM_{\LT,\infty} = \varprojlim_{K\subset \GL_n(F)} \cM_{\LT,K}
\]
exists as a perfectoid space. The induced map
\[
\pi_\GH: \cM_{\LT,\infty}\to \mathbb{P}^{n-1}_{\breve{F}}
\]
is in a suitable sense a $\GL_n(F)$-torsor; however, it takes a little bit of effort to make this statement precise, and we do not do so here. However, for any smooth $\GL_n(F)$-representation $\pi$ on an $\F_p$-vector space,\footnote{One can also handle more general base rings, and we do so in the paper.} one can construct a Weil-equivariant sheaf $\mathcal{F}_\pi$ on the \'etale site of the rigid space $\mathbb{P}^{n-1}_{\breve{F}}$. Our main theorem is the following.

\begin{thm} Let $\pi$ be an admissible smooth $\GL_n(F)$-representation on an $\F_p$-vector space. The cohomology group
\[
H^i_\et(\mathbb{P}^{n-1}_C,\mathcal{F}_\pi)
\]
is independent of the choice of an algebraically closed complete extension $C$ of $\breve{F}$, and vanishes for $i>2(n-1)$. For all $i\geq 0$,
\[
H^i_\et(\mathbb{P}^{n-1}_{\C_p},\mathcal{F}_\pi)
\]
is an admissible $D^\times$-representation, and the action of the Weil group $W_F$ extends continuously to an action of the absolute Galois group $\Gal_F$ of $F$.
\end{thm}

The proof of this theorem follows closely the proof of finiteness of $\F_p$-cohomology of proper (smooth) rigid spaces, \cite{ScholzePAdicHodge}. In particular, it depends crucially on properness of $\mathbb{P}^{n-1}$, or more precisely, on properness of the image of $\pi_\GH$. Unfortunately, it turns out that the Lubin-Tate case is essentially (up to products and changing the center) the only example of a Rapoport-Zink space with surjective period map. We refer to the Appendix by M. Rapoport for further discussion of this point. Thus, the methods of this paper do not shed light on other groups.

\begin{rem} Intuitively, $H^\ast_\et(\mathbb{P}^{n-1}_C,\mathcal{F}_\pi)$ is the $\pi$-isotypic component of the cohomology of the Lubin-Tate tower, but the formulation is different for several reasons. First, the (usual or compactly supported) cohomology groups of $\cM_{\LT,0,C}$ or $\cM_{\LT,\infty,C}$ itself are not well-behaved, e.g. not admissible and not invariant under change of $C$, cf. work of Chojecki, \cite{Chojecki}. Using lifts of Artin-Schreier covers one can check that already $H^1_\et(\mathbb{B}_C,\F_p)$ is infinite-dimensional and depends on $C$, where $\mathbb{B}_C$ denotes the closed unit disc over $C$. Second, taking the $\pi$-isotypic component is not an exact operation for $\F_p$-representations.
\end{rem}

For the local-global-compatibility results, we have decided to work only with $\GL_2$, as this leads to many technical simplifications; it is to be expected that many arguments can be adapted to $\GL_n$ if one uses Harris-Taylor type Shimura varieties, \cite{HarrisTaylor}. Fix a totally real field $F$ and a place $\frap$ dividing $p$ such that $F_\frap$ is the $p$-adic field considered previously. Moreover, fix an infinite place $\infty_F$ of $F$. Let $D_0$ be a division algebra over $F$ which is split at $\frap$ and is ramified at all infinite places. Let $G = D_0^\times$ be the algebraic group of units in $D_0$. Let $D$ be the inner form of $G$ which is split at $\infty_F$ and ramified at $\frap$ (and unchanged at all other places), and denote by $D^\times$ the algebraic group of units of $D$. Fix a compact open subgroup $U^\frap\subset G(\A_{F,f}^\frap)\cong D^\times(\A_{F,f}^\frap)$. For each $K\subset \GL_2(F_\frap)\cong G(F_\frap)$, one has the space of algebraic automorphic forms
\[
S(K U^\frap,\F_p) = C^0(G(F)\backslash G(\A_{F,f}) / K U^\frap,\Q_p/\Z_p)\ ,
\]
as well as the cohomology
\[
H^1(\Sh_{K^\prime U^\frap,\C},\Q_p/\Z_p)
\]
of the Shimura curve $\Sh_{K^\prime U^\frap}/F$ for $D/F$, for varying $K^\prime\subset D_\frap^\times = D^\times(F_\frap)$. These $H^0$-, resp. $H^1$-, groups are respectively the middle cohomology groups of the relevant Shimura varieties. Let
\[
\pi = \varinjlim_K S(K U^\frap,\Q_p/\Z_p)
\]
and
\[
\rho = \varinjlim_{K^\prime} H^1(\Sh_{K^\prime U^\frap,\C},\Q_p/\Z_p)\ ,
\]
which are admissible $\GL_2(F_\frap)$-, resp. $D^\times_\frap$-representations over $\Z_p$. Moreover, $\rho$ carries a representation of $\Gal_F$, and thus of $\Gal_{F_\frap}$. The following theorem is an easy consequence of Cerednik's $p$-adic uniformization, cf. \cite{Cerednik}, \cite{DrinfeldUnif}, \cite{RapoportZinkBook}, \cite{BoutotZink}, along with the duality isomorphism between the Lubin-Tate and the Drinfeld tower, cf. \cite{FaltingsTwoTowers}, \cite{FarguesTwoTowers}, \cite{ScholzeWeinstein}.

\begin{thm} There is a canonical $\Gal_{F_\frap}\times D^\times_\frap$-equivariant isomorphism
\[
H^1_\et(\mathbb{P}^1_{\C_p},\mathcal{F}_\pi) \cong \rho\ .
\]
\end{thm}

This is a form of a $p$-adic local-global-compatibility result, and we deduce the following more precise results. Fix an absolutely irreducible (odd) $2$-dimensional representation $\overline{\sigma}$ of $\Gal_F$ over a finite extension $\F_q$ of $\F_p$; this gives rise to a maximal ideal $\mm$ of the abstract Hecke algebra $\T$ (coming from unramified places), and we assume that the localization $\pi_\mm\neq 0$, i.e. $\overline{\sigma}$ is modular. There is a corresponding Hecke algebra $\T(U^\frap)_\mm$ (a complete local noetherian ring with residue field $\F_q$) acting faithfully on $\pi_\mm$. There is a $2$-dimensional Galois representation
\[
\sigma_\mm: \Gal_F\to \GL_2(\T(U^\frap)_\mm)
\]
characterized by the Eichler-Shimura relations expressing the characteristic polynomials of Frobenius elements in terms of Hecke operators. The next result says that one can recover $\sigma_\mm|_{\Gal_{F_\frap}}$ from $\pi_\mm$.

\begin{thm} There is a $\T(U^\frap)_\mm[\Gal_{F_\frap}\times D^\times_\frap]$-equivariant isomorphism
\[
H^1_\et(\mathbb{P}^1_{\C_p},\mathcal{F}_{\pi_\mm}) \cong \sigma_\mm|_{\Gal_{F_\frap}}\otimes_{\T(U^\frap)_\mm} \rho_\mm[\sigma_\mm]\ ,
\]
for some faithful $\T(U^\frap)_\mm[D^\times_\frap]$-module $\rho_\mm[\sigma_\mm]$ carrying the trivial $\Gal_{F_\frap}$-action. If $\overline{\sigma}|_{\Gal_{F_\frap}}$ is irreducible, this determines the $\T(U^\frap)_\mm[\Gal_{F_\frap}]$-representation $\sigma_\mm|_{\Gal_{F_\frap}}$ uniquely.
\end{thm}

Moreover, there is a version for the $\mm$-torsion.

\begin{thm} The $2$-dimensional $\Gal_{F_\frap}$-representation $\overline{\sigma}|_{\Gal_{F_\frap}}$ is determined by the admissible $\GL_2(F_\frap)$-representation $\pi[\mm]$. More precisely, $\overline{\sigma}|_{\Gal_{F_\frap}}$ can be read off from the $\Gal_{F_\frap}$-representation
\[
H^1_\et(\mathbb{P}^1_{\C_p},\mathcal{F}_{\pi[\mm]})\ ,
\]
which is an infinite-dimensional admissible $\Gal_{F_\frap}\times D^\times_\frap$-representation. Any indecomposable $\Gal_{F_\frap}$-subrepresentation of $H^1_\et(\mathbb{P}^1_{\C_p},\mathcal{F}_{\pi[\mm]})$ is of dimension $\leq 2$, and
$\overline{\sigma}|_{\Gal_{F_\frap}}$ is determined in the following way.
\begin{altenumerate}
\item[{\rm Case (i)}] If there is a $2$-dimensional indecomposable $\Gal_{F_\frap}$-representation
\[
\sigma^\prime\subset H^1_\et(\mathbb{P}^1_{\C_p},\mathcal{F}_{\pi[\mm]})\ ,
\]
then $\overline{\sigma}|_{\Gal_{F_\frap}} = \sigma^\prime$.
\item[{\rm Case (ii)}] Otherwise, $H^1_\et(\mathbb{P}^1_{\C_p},\mathcal{F}_{\pi[\mm]})$ is a direct sum of characters of $\Gal_{F_\frap}$, and at most two different characters $\chi_1,\chi_2$ of $\Gal_{F_\frap}$ appear; if only one appears, let $\chi_2=\chi_1$ be the only character appearing. Then $\overline{\sigma}|_{\Gal_{F_\frap}} = \chi_1\oplus \chi_2$.
\end{altenumerate}
\end{thm}

{\bf Acknowledgments.} The results of this paper were found during the conference in honour of Henri Carayol and Jean-Pierre Wintenberger in Strasbourg in January 2014, and the author wishes to thank Arthur-C\'esar le Bras for useful discussions during that conference. The compatibility with patching was proved following questions of Ana Caraiani, whom the author wishes to thank. The results of this paper were first announced in February 2014 (during a snowstorm in Princeton). The author wants to apologize for the long delay in the preparation of this paper. Moreover, he wants to thank Christophe Breuil, Frank Calegari, Przemyslaw Chojecki, Pierre Colmez, Matt Emerton, Toby Gee, David Geraghty, Michael Harris, Eugen Hellmann, Vytas Paskunas, Sug Woo Shin, Richard Taylor and Jack Thorne for helpful discussions, and Judith Ludwig and Michael Rapoport for the careful reading of the manuscript. This work was done while the author was a Clay Research Fellow.

\section{Some equivariant sites}

In the proof of our main result, we need to consider cohomology groups of some objects like $\mathbb{P}^{n-1}/K$ for a compact open subgroup $K\subset D^\times$. There are several possible definitions of these cohomology groups. One might define them in terms of the simplicial adic space $(\mathbb{P}^{n-1}\times EK)/K$ with terms $\mathbb{P}^{n-1}\times K^i$, or in terms of some stacky diamond $(\mathbb{P}^{n-1})^\diamond/K$, using diamonds as in \cite{SWDiamonds}. The technically simplest solution seems to be to directly define a site $(\mathbb{P}^{n-1}/K)_\et$ that gives rise to these cohomology groups.     

In the following, let $X$ be either a locally noetherian analytic adic space, in the sense that $X$ is locally of the form $\Spa(A,A^+)$ for some strongly noetherian Tate ring $A$ and a ring of integral elements $A^+\subset A$, or a perfectoid space.\footnote{Everything works for analytic adic spaces which have a well-behaved \'etale site.} If $X$ is a perfectoid space, all affinoid subsets below are assumed to be of the form $\Spa(A,A^+)$, where $A$ is perfectoid. For simplicity, we will spell out only the case of locally noetherian analytic adic spaces.

\begin{definition} Let $G$ be a locally profinite group. An action of $G$ on $X$ is said to be continuous if $X$ admits a cover by open affinoid $\Spa(A,A^+)\subset X$ stabilized by open subgroups $H\subset G$ such that the action morphism $H\times A\to A$ is continuous.
\end{definition}

\begin{lem} Assume that a locally profinite group $G$ acts continuously on the locally noetherian analytic adic space $X$. For any quasicompact open subset $U\subset X$, the stabilizer $G_U\subset G$ of $U$ in $G$ is open. If $U=\Spa(A,A^+)$ is affinoid, then the action morphism $G_U\times A\to A$ is continuous.
\end{lem}

\begin{proof} First, we check that there is a basis of affinoid open subsets $\Spa(A,A^+)\subset X$ which have an open stabilizer $H$ in $G$, and for which the action morphism on $A$ is continuous. It is enough to check that this property passes to rational subsets. Fix a ring of definition $A_0\subset A$ and a pseudouniformizer $\varpi\in A_0$, i.e. a topologically nilpotent unit of $A$. If $U\subset \Spa(A,A^+)$ is the rational subset defined by
\[
U=\{\forall i=1,\ldots,n: |f_i(x)|\leq |g(x)|\neq 0\}
\]
for some $f_1,\ldots,f_n,g\in A$ such that the ideal $(f_1,\ldots,f_n)$ is all of $A$, then $V = \varpi(A^+f_1+\ldots+A^+f_n)\subset A$ is an open neighborhood of $0$ such that for all $f_i^\prime\in f_i+V$, $g^\prime\in g+V$, the rational subset defined by $f_1^\prime,\ldots,f_n^\prime,g^\prime$ agrees with $U$. From this and the continuity of the action morphism $H\times A\to A$, it follows that the stabilizer $H_U$ of $U$ in $H$ is open.

To check that the action of $H_U$ on $\cO_X(U)$ is continuous, we deal with two cases separately. First, assume that all $f_i=1$. Then $\cO_X(U)$ is the completion of $A[1/g]$ with respect to the topology making $\varpi^m A_0[1/g]$ a basis of open neighborhoods of $0$. The action of $h\in H_U$ sends $g^{-1}$ to $h(g)^{-1} = g^{-1}(1 + a_h g^{-1} + a_h^2 g^{-2} +\ldots)$ in case $h(g) = g-a_h$ for some element $a_h\in \varpi A_0$; this happens in an open subgroup $h\in H^\prime\subset H_U$. Moreover, $a_h$ varies continuously with $h$, which implies that also
\[
h(g)^{-1} = g^{-1}(1 + a_h g^{-1} + a_h^2 g^{-2} +\ldots)\in \cO_X(U)
\]
varies continuously with $h$. Going through the definitions, this implies that the action of $H^\prime$ on $\cO_X(U)$ is continuous, and as $H^\prime\subset H_U$ is open, this implies the same for the action of $H_U$ on $\cO_X(U)$.

Now assume that $g=1$. In that case, $\cO_X(U)$ is the completion of $A$ with respect to the topology making $\varpi^m A_0[f_1,\ldots,f_n]$ a basis of open neighborhoods of $0$. In this case, continuity is immediately verified. In general, as $A$ is Tate, any rational subset is a rational subset of the second form inside a rational subset of the first form, verifying continuity of $H_U\times \cO_X(U)\to \cO_X(U)$.

Thus, any quasicompact open $U\subset X$ is covered by finitely many $U_i\subset U$ whose stabilizer $G_i\subset G$ is open. The intersection $\cap_i G_i\subset G$ is still open and stabilizes $U$, proving the first claim. For the second claim, if $U=\Spa(A,A^+)$, one can choose all $U_i=\Spa(A_i,A_i^+)\subset U$ affinoid such that the action of $G_i$ on $A_i$ is continuous. Then the action of $\cap_i G_i$ on the closed subspace $A\subset \prod_i A_i$ is continuous, giving the result.
\end{proof}

We will also need a result about extending group actions to finite \'etale covers.

\begin{lem} Let $X=\Spa(A,A^+)$ with $A$ a strongly noetherian Tate ring. Let $G$ be a profinite group acting continuously on $X$, and let $B$ be a finite \'etale $A$-algebra, $B^+\subset B$ the integral closure of $A^+$, and $Y=\Spa(B,B^+)$. Assume that there is some closed subgroup $H_0\subset G$ such that the $H_0$-action on $X$ lifts to an $H_0$-action on $Y$, and fix such an action. Then there is an open subgroup $H\subset G$ containing $H_0$ and a continouous action of $H$ on $Y$ compatible with the action of $X$, and the $H_0$-action on $Y$. Given two such continuous actions of open subgroups $H$, $H^\prime$ on $Y$, there is an open subgroup $H^{\prime\prime}\subset H\cap H^\prime$ containing $H_0$ on which they agree.
\end{lem}

Note that in particular, one can apply the lemma in the case that $H_0$ is trivial, or in the case $G=H_0\times G_0$ of two commuting actions.

\begin{proof} Everything can be translated into actions on $A$ resp. $B$. Let $C^0(G,A)$ be the ring of continuous functions $G\to A$ with pointwise addition and multiplication; this is again a complete Tate ring, intuitively corresponding to the space $X\times G$. There is a natural map $m: A\to C^0(G,A)$ sending $f\in A$ to the map $g\mapsto g(f)$; this corresponds to the action map $X\times G\to X$. There is also the diagonal embedding $p: A\to C^0(G,A)$ corresponding to the projection $X\times G\to X$.

One checks that giving a continuous action of $H$ on $B$ is equivalent to giving an isomorphism of finite \'etale $C^0(H,A)$-algebras
\[
B\otimes_{A,m} C^0(H,A)\cong B\otimes_{A,p} C^0(H,A)
\]
satisfying the obvious cocycle condition over $C^0(H\times H,A)$. Now recall the following result of Elkik, \cite{Elkik}, and Gabber-Ramero, \cite[Proposition 5.4.53]{GabberRamero}, cf. also \cite[Lemma 7.5 (i)]{ScholzePerfectoid}.

\begin{thm}\label{ElkikGabberRamero} Let $R_i$ be a filtered inductive system of complete Tate rings with compatible rings of definition $R_{i,0}\subset R_i$. Pick a pseudouniformizer $\varpi\in R_{i,0}$ for some $i$, which we assume is minimal, thus giving compatible pseudouniformizers $\varpi\in R_{i,0}$ for all $i$. Let $R_0$ be the $\varpi$-adic completion of $\varinjlim_i R_{i,0}$, and $R=R_0[\varpi^{-1}]$. Then
\[
R_\fet\cong \twoinjlim_i (R_i)_\fet\ .
\]
\end{thm}

Applying this to the system $R_i = C^0(H_i,A)$ for a basis of open subgroups $H_i\subset G$ containing $H_0$, with $R_{i,0} = C^0(H_i,A_0)$, we get $R=C^0(H_0,A)$ as the completed direct limit. As we are given an isomorphism
\[
B\otimes_{A,m} C^0(H_0,A)\cong B\otimes_{A,p} C^0(H_0,A)
\]
of finite \'etale $C^0(H_0,A)$-algebras satisfying the cocycle condition, the theorem of Elkik-Gabber-Ramero shows that this spreads in an essentially unique way into an isomorphism
\[
B\otimes_{A,m} C^0(H,A)\cong B\otimes_{A,p} C^0(H,A)
\]
for an open subgroup $H\times G$ containing $H_0$. Moreover, the cocycle condition is satisfied for $H$ sufficiently small, by applying the same reasoning for the system of the $C^0(H_i\times H_i,A)$.
\end{proof}

This implies the same result for \'etale maps.

\begin{cor}\label{ResidualAction} Let $X$ be a locally noetherian analytic adic space equipped with a continuous action by a profinite group $G$. Let $Y\to X$ be an \'etale map, and assume that $Y$ is qcqs, and carries a compatible action of a closed subgroup $H_0\subset G$. Then there is an open subgroup $H$ of $G$ containing $H_0$ which acts continuously on $Y$ extending the $H_0$-action, compatibly with the action on $X$, and two such actions agree after shrinking $H$. Any morphism $Y\to Y^\prime$ of qcqs \'etale adic spaces equipped with $H_0$-actions over $X$ is equivariant for the $H$-action if $H$ is small enough.
\end{cor}

\begin{proof} We have already verified this result for finite \'etale maps and open subsets. In general $Y\to X$ has an open cover by finitely many subsets which are compositions of open subsets and finite \'etale maps. Thus, we can get such actions over a quasicompact open cover $\{Y_i\to Y\}$; to glue them to all of $Y$, we need to make them compatible on $Y_i\times_Y Y_j$. As $Y$ is quasiseparated, the fibre products $Y_i\times_Y Y_j$ are quasicompact. Using that any two actions of open subgroups $H$, $H^\prime$ on $Y_i\times_Y Y_j$ agree on an open neighborhood then gives the action on $Y$. Similarly, one checks that this action is equivariant for morphisms.
\end{proof}

Now we can define the equivariant \'etale site.

\begin{definition} Let $X$ be a locally noetherian analytic adic space with a continuous action by a locally profinite group $G$. Let $(X/G)_\et$ be the site whose objects are (locally noetherian analytic) adic spaces $Y$ equipped with a continuous action of $G$, and a $G$-equivariant \'etale morphism $Y\to X$. Morphisms are $G$-equivariant maps over $X$, and a family of morphisms $\{f_i: Y_i\to Y\}$ is a cover if $|Y| = \bigcup_i f_i(|Y_i|)$.

Let $(X/G)_\et^\sim$ denote the associated topos.
\end{definition}

It is directly verified that $(X/G)_\et$ has good properties, e.g. all finite limits exist. If $G$ is profinite, there is also a good notion of quasicompact and quasiseparated objects.

\begin{lem} Let $X$ be a locally noetherian analytic adic space with a continuous action by a profinite group $G$.
\begin{altenumerate}
\item[{\rm (i)}] An object $Y\in (X/G)_\et$ is quasicompact if and only if $|Y|$ is quasicompact.
\item[{\rm (ii)}] An object $Y\in (X/G)_\et$ is quasiseparated if and only if $|Y|$ is quasiseparated.
\item[{\rm (iii)}] A morphism $f: Y\to Y^\prime$ in $(X/G)_\et$ is quasiseparated (resp. quasicompact) if and only if $|f|: |Y|\to |Y^\prime|$ is quasiseparated (resp. quasicompact).
\item[{\rm (iv)}] Consider the set of $U\in (X/G)_\et$ for which $U$ is affinoid; this forms a basis for the topology consisting of qcqs objects which are stable under fibre products.
\item[{\rm (v)}] The site $(X/G)_\et$ is algebraic, in particular locally coherent.
\end{altenumerate}
\end{lem}

\begin{proof} As \'etale maps are open, it follows that if $|Y|$ is quasicompact, then so is $Y\in (X/G)_\et$: If $\{f_i: Y_i\to Y\}$ is a cover, so that $|Y| = \bigcup_i f_i(|Y_i|)$, then finitely many of the open subsets $f_i(|Y_i|)$ already cover, giving a finite subcover in $(X/G)_\et$.

Next, we show that the set of affinoid $U\in (X/G)_\et$ forms a basis for the topology. For any $Y\in (X/G)_\et$, pick an open affinoid subset $V\subset Y$. This is stabilized by some open subgroup $H\subset G$, and then $U=V\times_H G$ is an affinoid space (as it is non-equivariantly isomorphic to $V\times G/H$). One gets a $G$-equivariant map $U\to Y$, and these cover $Y$. Obviously, the set of affinoid $U\in (X/G)_\et$ is stable under fibre products, proving (iv).

Now let $Y\in (X/G)_\et$ be quasicompact. Then we can cover $Y$ by finitely many affinoid $U_i\in (X/G)_\et$. The resulting surjection from a quasicompact space $\bigsqcup_i |U_i|$ to $|Y|$ shows that $|Y|$ is quasicompact, proving (i). All other properties are readily established.
\end{proof}

Moreover, we need the following property.

\begin{prop}\label{ResidualActionSites} Let $X$ be a qcqs locally noetherian analytic adic space with a continuous action by a profinite group $G$. The association mapping $Y\in (X/G)_\et$ to $Y\in X_\et$ defines a morphism of sites $X_\et\to (X/G)_\et$ under which $X_\et^\sim$ is a projective limit of the fibred topos $((X/H)_\et^\sim)_H$, where $(X/H)_\et^\sim$ is considered as a fibred topos over the category of open subgroups $H\subset G$ in an obvious way. More generally, whenever $H_0\subset G$ is a closed subgroup, $(X/H_0)_\et^\sim$ is a projective limit of the fibred topos $(X/H)_\et^\sim$ for $H\supset H_0$ open subgroups of $G$. In particular, for any sheaf $\mathcal{F}\in (X/G)_\et^\sim$,
\[
H^i((X/H_0)_\et,\mathcal{F}) = \varinjlim_{H_0\subset H\subset G} H^i((X/H)_\et,\mathcal{F})\ ,
\]
where we write $\mathcal{F}$ also for its pullback to $(X/H_0)_\et$, resp. $(X/H)_\et$.
\end{prop}

\begin{proof} We can replace $(X/G)_\et$ by the site $(X/G)_\et^{\mathrm{qcqs}}$ of qcqs $Y\in (X/G)_\et$, which gives rise to the same topos. Then, by Lemma \ref{ResidualAction} and the previous identification of qcqs objects, we have an identification of categories
\[
(X/H_0)_\et^{\mathrm{qcqs}}\cong \twoinjlim_{H_0\subset H\subset G} (X/H)_\et^{\mathrm{qcqs}}\ ,
\]
where $H$ runs over open subgroups. Moreover, $(X/H_0)_\et^{\mathrm{qcqs}}$ is equipped with the weakest induced topology. Thus, by SGA 4 VI Th\'eor\`eme 8.2.3, we get the result.
\end{proof}

It is useful to combine this result with the observation that $(X/H)_\et\to (X/G)_\et$ is a slice, if $H\subset G$ is open.

\begin{prop}\label{SliceTopoi} Let $X$ be a locally noetherian analytic adic space with a continuous action by a locally profinite group $G$. Let $H\subset G$ be an open subgroup, and consider $X\times_H G\in (X/G)_\et$. Then the functor $U\mapsto U\times_H G$ induces an equivalence between $(X/H)_\et$ and the slice site $(X/G)_\et / (X\times_H G)$.
\end{prop}

\begin{proof} It is enough to prove that one gets an equivalence of categories $(X/H)_\et\cong (X/G)_\et / (X\times_H G)$, as the notion of covers corresponds. The inverse functor is given by sending a $G$-equivariant map $U\to X\times_H G$ to the fibre over $X = X\times_H H\to X\times_H G$, and the functors are clearly inverse.
\end{proof}

Assume now that $X$ lives over $\Spa(K,\cO_K)$ for some nonarchimedean field $K$; fix a pseudouniformizer $\varpi\in \cO_K$. Let $\cO_X^+/\varpi$ be the sheaf on $X_\et$ which is the sheafication of $U\mapsto \cO_X^+(U)/\varpi$.

\begin{lem}\label{DefOXmodp} Let $X$ be a locally noetherian analytic adic space over $\Spa(K,\cO_K)$ with a continuous action by a locally profinite group $G$ (compatible with the trivial action on $\Spa(K,\cO_K)$). The association $\cO_{X/G}^+/\varpi$ mapping $U\in (X/G)_\et$ to $((\cO_X^+/\varpi)(U))^G$ is a sheaf on $(X/G)_\et$. The pullback of $\cO_{X/G}^+/\varpi$ to $X_\et$ is equal to $\cO_X^+/\varpi$.
\end{lem}

We warn the reader that there is no sheaf $\cO_{X/G}^+$ in general whose reduction modulo $\varpi$ is $\cO_{X/G}^+/\varpi$, so that we are doing some abuse of notation here. The problem is that $\cO_X^+$ may not have enough sections invariant under $G$, but continuity of the action of $G$ implies that there are many sections which are invariant modulo $\varpi$.

\begin{proof} The sheaf property of $\cO_{X/G}^+/\varpi$ follows by taking $G$-invariants in the sheaf property of $\cO_X^+/\varpi$. To check the pullback, one first checks from the definition that the definition of $\cO_{X/G}^+/\varpi$ is compatible with pullback to $(X/H)_\et$ for an open subgroup $H\subset G$, and then one uses Proposition \ref{ResidualActionSites} along with the observation that by continuity of the $G$-action, any section of $(\cO_X^+/\varpi)(U)$ is invariant under some open $H\subset G$ if $U$ is qcqs.
\end{proof}

We will need the following ``conservativity" property.

\begin{lem}\label{Conservativity} Let $X$ be a locally noetherian analytic adic space with a continuous action by a locally profinite group $G$. Then a pointed sheaf $\mathcal{F}$ on $(X/G)_\et$ is trivial if and only if its pullback to $X_\et$ is trivial.
\end{lem}

\begin{proof} Assume that the pullback of $\mathcal{F}$ to $X_\et$ is trivial, and let $s\in \mathcal{F}(U)$, $U\in (X/G)_\et$, be a section. Assume first that $s$ becomes trivial after pullback to $(X/H)_\et$ for some open $H\subset G$. Then $s$ becomes trivial over $U\in (X/H)_\et$, which corresponds to $U\times_H G\in (X/G)_\et$, which is a cover $U\times_H G\to U$ of $U$ in $(X/G)_\et$, so that $s$ is already trivial in $(X/G)_\et$.

Thus, it is enough to check that $s$ becomes trivial after pullback to $(X/H)_\et$ for some open $H\subset G$. In particular, we may assume that $G$ is profinite, and then that $U$ is qcqs. By assumption the pullback of $s$ to $U\in X_\et$ is trivial. On the other hand, by Proposition \ref{ResidualActionSites}, we have
\[
H^0(U_\et,\mathcal{F}) = \varinjlim_{H\subset G} H^0((U/H)_\et,\mathcal{F})\ ,
\]
so $s$ becomes trivial on $U\in (X/H)_\et$ for some open $H\subset G$, finishing the proof.
\end{proof}

Now assume that $X$ is a locally noetherian adic space over $\Spa(\Q_p,\Z_p)$, and that $(X_H)_H\to X$ is a $G$-torsor for some profinite group $G$, in the sense that for all open normal subgroups $H\subset G$, $X_H\to X$ is a finite \'etale $G/H$-torsor, compatibly in $H$. Moreover, assume that there is a perfectoid space $X_\infty\to X$ such that
\[
X_\infty\sim \varprojlim_H X_H
\]
in the sense that there is a covering of $X_\infty$ by affinoid perfectoid $U_\infty = \Spa(R_\infty,R_\infty^+)$ coming as pullback of affinoid $U_H = \Spa(R_H,R_H^+)\subset X_H$ for all sufficiently small $H$, with $R_\infty^+$ the $p$-adic completion of $\varinjlim_H R_H^+$. In that case, one has $|X_\infty|\cong \varprojlim_H |X_H|$. (Cf. \cite[Definition 2.4.1]{ScholzeWeinstein}.)

Note that there is a natural morphism of sites $(X_\infty/G)_\et\to X_\et$, as any \'etale $U\to X_\et$ pulls back to an \'etale $U_\infty\to X_\infty$ equipped with a continuous action of $G$. Assume moreover that a locally profinite group $J$ acts continuously and compatibly on $X$ and all $X_H$, commuting with the $G$-action. Then $J$ acts continuously on $X_\infty$.

\begin{prop}\label{QuotientPerfectoid} The natural morphism $(X_\infty/G\times J)_\et\to (X/J)_\et$ is an equivalence of sites.
\end{prop}

\begin{proof} Let us sketch the argument. It is enough to check that there is an equivalence of categories $(X_\infty/G\times J)_\et\cong (X/J)_\et$, as the notions of covers correspond. For this, we can replace $J$ by an open subgroup, in particular we can assume that $J$ is compact. Also we can argue locally on $X$, and assume that $X$, and thus all $X_H$, $X_\infty$, are qcqs.

Now it is enough to prove $(X_\infty/G\times J)_\et^{\mathrm{qcqs}}\cong (X/J)_\et^{\mathrm{qcqs}}$ by covering general objects by qcqs objects. We claim that there is an equivalence of categories
\[
(X_\infty/J)_\et^{\mathrm{qcqs}} = \twoinjlim_{H\subset G} (X_H/J)_\et^{\mathrm{qcqs}}\ .
\]
As usual, such a statement can be reduced to quasicompact open embeddings and finite \'etale covers individually. For quasicompact open embeddings, it follows from the identification $|X_\infty|\cong \varprojlim_H |X_H|$ of topological spaces. For finite \'etale covers, it follows from the theorem of Elkik-Gabber-Ramero, Theorem \ref{ElkikGabberRamero}, along with the assumption $X_\infty\sim \varprojlim_H X_H$.

Now, if $U\to X_\infty$ is \'etale and qcqs and admits a compatible continuous $G\times J$-action, then $J$-equivariantly $U\to X_\infty$ comes via pullback from some \'etale qcqs $U_H\to X_H$ for $H$ small enough. Then the identification $U=U_H\times_{X_H} X_\infty\to X_\infty$ endows $U$ with a second $H\times J$-action, agreeing on $J$. As in the proof of Proposition \ref{ResidualActionSites}, we have an equivalence of categories
\[
(X_\infty/J)_\et^{\mathrm{qcqs}}\cong \twoinjlim_H (X_\infty/H\times J)_\et^{\mathrm{qcqs}}\ ,
\]
which shows that the two $H\times J$-actions on $U$ are compatible after shrinking $H$. This gives an $H\times J$-equivariant identification $U=U_H\times_{X_H} X_\infty$, and then the $G\times J$-action on $U$ endows $U_H$ with a $G/H\times J$-action. By finite \'etale descent, this descends $U_H$ to $U_0\to X$, a qcqs \'etale $J$-equivariant map. One checks that this gives the desired equivalence of categories.
\end{proof}

\begin{lem}\label{IdentOXmodp} Under the identification of topoi $(X_\infty/G)_\et^\sim\cong X_\et^\sim$, the sheaves $\cO_{X_\infty/G}^+/p$ and $\cO_X^+/p$ correspond.
\end{lem}

\begin{proof} By Lemma \ref{Conservativity}, this can be checked after pullback to $(X_\infty)_\et$. By Lemma \ref{DefOXmodp}, the sheaf $\cO_{X_\infty/G}^+/p$ pulls back to $\cO_{X_\infty}^+/p$. The same can be verified for $\cO_X^+/p$ by using the local structure of $X_\infty\sim \varprojlim_H X_H$ and \cite[Theorem 7.17]{ScholzePerfectoid} to compute the pullback along $(X_\infty)_\et\to X_\et$.
\end{proof}

\section{Finiteness}

Let us use the notation from the introduction, so $n\geq 1$ is an integer and $F/\Q_p$ a finite extension with ring of integers $\cO\subset F$ and $\varpi\in \cO$. Let $q$ be the cardinality of the residue field of $F$, which we identify with $\F_q$. Fix an algebraically closed extension $k$ of $\F_q$, e.g. $\overline{\F}_q$. Let $\breve{F}=F\otimes_{W(\F_q)} W(k)$ be the completion of the unramified extension of $F$ with residue field $k$. Let $\breve{\cO}\subset \breve{F}$ be the ring of integers.

In this situation, one has the Lubin-Tate tower $(\cM_{\LT,K})_{K\subset \GL_n(F)}$, which is a tower of smooth rigid-analytic varieties $\cM_{\LT,K}$ over $\breve{F}$ parametrized by compact open subgroups $K$ of $\GL_n(F)$, with finite \'etale transition maps, cf. \cite{GrossHopkins}. There is a compatible continuous action of $D^\times$ on all $\cM_{\LT,K}$, as well as an action of $\GL_n(F)$ on the tower, i.e. $g\in \GL_n(F)$ induces an isomorphism between $\cM_{\LT,K}$ and $\cM_{\LT,g^{-1} K g}$. There is the Gross-Hopkins period map, \cite{GrossHopkins},
\[
\pi_\GH: \cM_{\LT,K}\to \mathbb{P}^{n-1}_{\breve{F}}\ ,
\]
compatible for varying $K$, which is an \'etale covering map of rigid-analytic varieties with fibres $\GL_n(F)/K$. It is also $D^\times$-equivariant if the right-hand side is correctly identified with the Brauer-Severi variety for $D/F$ (which splits over $\breve{F}$). Moreover, there is a Weil descent datum on $\cM_{\LT,K}$ (cf. \cite[3.48]{RapoportZinkBook}), under which $\pi_\GH$ is equivariant, under this identification of $\mathbb{P}^{n-1}_{\breve{F}}$ with the Brauer-Severi variety for $D/F$.

Moreover, denote by $\cM_{\LT,\infty}$ over $\breve{F}$ (which lives over the completion of the maximal abelian extension of $F$, which is a perfectoid field) the perfectoid space constructed in \cite{ScholzeWeinstein}, so that
\[
\cM_{\LT,\infty}\sim \varprojlim_K \cM_{\LT,K}\ .
\]

Fix an admissible $\F_p$-representation $\pi$ of $\GL_n(F)$. We want to construct a sheaf $\mathcal{F}_\pi$ on $(\mathbb{P}^{n-1}_{\breve{F}}/D^\times)_\et$, which is also equivariant for the Weil descent datum. The idea is to descend the trivial sheaf $\pi$ along the map
\[
\pi_\GH: \cM_{\LT,\infty}\to \mathbb{P}^{n-1}_{\breve{F}}\ ,
\]
which can be considered as a $\GL_n(F)$-torsor.

\begin{prop}\label{DefFPi} The association mapping a $D^\times$-equivariant \'etale map $U\to \mathbb{P}^{n-1}_{\breve{F}}$ to the $\F_p$-vector space
\[
\Map_{\cont,\GL_n(F)\times D^\times}(|U\times_{\mathbb{P}^{n-1}_{\breve{F}}} \cM_{\LT,\infty}|,\pi)
\]
of continuous $\GL_n(F)\times D^\times$-equivariant maps defines a Weil-equivariant sheaf $\mathcal{F}_\pi$ on $(\mathbb{P}^{n-1}_{\breve{F}}/D^\times)_\et$. The association $\pi\mapsto \mathcal{F}_\pi$ is exact, and all geometric fibres of $\mathcal{F}_\pi$ are isomorphic to $\pi$, i.e. for any $\overline{x}=\Spa(C,C^+)\to \mathbb{P}^{n-1}_{\breve{F}}$ with $C/\breve{F}$ complete algebraically closed and $C^+\subset C$ an open bounded valuation subring, the pullback of $\mathcal{F}_\pi$ to $\bar{x}$,
\[
\mathcal{F}_{\pi,\bar{x}} = \varinjlim_{\bar{x}\in U\in (\mathbb{P}^{n-1}_{\breve{F}}/D^\times)_\et} \mathcal{F}_\pi(U)\ ,
\]
is isomorphic to $\pi$; the isomorphism is canonical after fixing a lift of $\bar{x}$ to $\cM_{\LT,\infty}$.
\end{prop}

\begin{proof} As \'etale covers induce (by definition) surjections on topological spaces, and are open, it follows that $\mathcal{F}_\pi$ is a sheaf; Weil equivariance follows from Weil equivariance of all other objects involved. By exactness of pullback and Lemma \ref{Conservativity}, exactness of $\pi\mapsto \mathcal{F}_\pi$ can be checked after pullback to $(\mathbb{P}^{n-1}_{\breve{F}})_\et$. The pullback of $\mathcal{F}_\pi$ to $(\mathbb{P}^{n-1}_{\breve{F}})_\et$ is the sheaf assigning to an \'etale $U\to \mathbb{P}^{n-1}_{\breve{F}}$ the set of continuous $\GL_n(F)$-equivariant maps
\[
|U\times_{\mathbb{P}^{n-1}_{\breve{F}}} \cM_{\LT,\infty}|\to \pi\ ,
\]
as if $U$ is qcqs, any such map is automatically equivariant for some open $H\subset D^\times$; here, we use Proposition \ref{ResidualActionSites} to compute the pullback of $\mathcal{F}_\pi$. To check exactness over $(\mathbb{P}^{n-1}_{\breve{F}})_\et$, we check at geometric points; it is enough to prove that the stalk of $\mathcal{F}_\pi$ on any geometric point is equal to $\pi$. Thus, fix some geometric point $\bar{x}=\Spa(C,C^+)\to \mathbb{P}^{n-1}_{\breve{F}}$, and let $\{U_i\to X\}$ be the cofiltered inverse system of affinoid \'etale neighborhoods of $\bar{x}$; we may assume that they are all connected. Then
\[
\mathcal{F}_{\pi,\bar{x}} = \varinjlim_i \Map_{\cont,\GL_n(F)}(|U_i\times_{\mathbb{P}^{n-1}_{\breve{F}}} \cM_{\LT,\infty}|,\pi)\ .
\]
Observe that as $U_i$ is connected, $\GL_n(F)$ acts transitively on the connected components of $|U_i\times_{\mathbb{P}^{n-1}_{\breve{F}}} \cM_{\LT,\infty}|$. It follows that the map $\mathcal{F}_{\pi,\bar{x}}\to \pi$ given by evaluation at a fixed point of $\bar{x}\times_{\mathbb{P}^{n-1}_{\breve{F}}} \cM_{\LT,\infty}$ is injective. To check surjectivity, note that by smoothness of $\pi$, any element $f\in \pi$ is invariant under some open subgroup $1+\varpi^m M_n(\cO)$ of $\GL_n(F)$. On the other hand, for any $m$, one can choose $U_i$ so that $U_i\to \mathbb{P}^{n-1}_{\breve{F}}$ factors over
\[
U_i\to \cM_{\LT,m}=\cM_{\LT,1+\varpi^m M_n(\cO)}\ .
\]
In that case, there is a $\GL_n(F)$-equivariant continuous surjection
\[
|U_i\times_{\mathbb{P}^{n-1}_{\breve{F}}} \cM_{\LT,\infty}|\to \GL_n(F)/(1+\varpi^m M_n(\cO))\ .
\]
Composing with the action map $\GL_n(F)/(1+\varpi^m M_n(\cO))\to \pi$ given by acting on $f$ then shows surjectivity of $\mathcal{F}_{\pi,\bar{x}}\to \pi$.
\end{proof}

Let $C/\breve{F}$ be an algebraically closed complete extension with ring of integers $\cO_C$. By a subscript $_C$, we denote the base change to $\Spa(C,\cO_C)$. The goal of this section is to prove the following theorem.

\begin{thm}\label{Finiteness} For any $i\geq 0$, the $D^\times$-representation $H^i_\et(\mathbb{P}^{n-1}_C,\mathcal{F}_\pi)$ is admissible, and vanishes for $i>2(n-1)$. If $\pi$ is injective as $\GL_n(\cO)$-representation, then it vanishes for $i>n-1$. Moreover, the natural map
\[
H^i_\et(\mathbb{P}^{n-1}_C,\mathcal{F}_\pi)\otimes_{\F_p} \cO_C/p\to H^i_\et(\mathbb{P}^{n-1}_C,\mathcal{F}_\pi\otimes \cO_X^+/p)
\]
is an almost isomorphism, and $H^i_\et(\mathbb{P}^{n-1}_C,\mathcal{F}_\pi)$ is independent of $C$ (i.e., the natural map for an inclusion $C\hookrightarrow C^\prime$ is an isomorphism).
\end{thm}

The almost isomorphism with $\cO_X^+/p$-cohomology is an analogue of a result of Faltings, \cite[\S 3, Theorem 8]{FaltingsAlmostEtale}, cf. also \cite[Theorem 1.3]{ScholzePAdicHodge}. One may hope that it allows one to understand the $p$-adic Hodge-theoretic properties of the Galois representations appearing in $H^i_\et(\mathbb{P}^{n-1}_C,\mathcal{F}_\pi)$ for Banach space representations $\pi$, following e.g. \cite{ScholzePAdicHodge}.

Our proof of Theorem \ref{Finiteness} follows closely the proof \cite[Theorem 1.3]{ScholzePAdicHodge}. It starts by proving finiteness of the $\cO_X^+/p$-twisted cohomology groups, and there it starts with a local finiteness result. Let us first recall the form of this result in \cite[Lemma 5.6 (ii)]{ScholzePAdicHodge}. For the formulation, we need two definitions.

\begin{definition} Let $V$ be a smooth affinoid adic space over $\Spa(C,\cO_C)$. A map
\[
V\to \mathbb{T}^n := \Spa(C\langle T_1^{\pm 1},\ldots,T_n^{\pm 1}\rangle,\cO_C\langle T_1^{\pm 1},\ldots,T_n^{\pm 1}\rangle)
\]
is a set of \emph{good coordinates} if it can be written as a composite of finite \'etale maps and rational embeddings.
\end{definition}

\begin{definition} Let $V$ be a separated analytic adic space, and $U\subset V$ a subset. Then $U$ is said to be strictly contained in $V$ if for any maximal point $x=\Spa(K,\cO_K)\in U$ and any open bounded valuation subring $K^+\subset K$, there is a map $\Spa(K,K^+)\to V$ extending $\Spa(K,\cO_K)\to U$.
\end{definition}

\begin{lem}[{\cite[Lemma 5.6 (ii)]{ScholzePAdicHodge}}]\label{AlmFinGenBaseCase} Let $V$ be a smooth affinoid adic space over $\Spa(C,\cO_C)$ with good coordinates, and let $U\subset V$ be a strict rational subset. Then the image of
\[
H^i(V_\et,\cO^+/p)\to H^i(U_\et,\cO^+/p)
\]
is almost finitely generated for all $i\geq 0$.
\end{lem}

In order to facilitate applications, let us note that this result is true without the ``good coordinates" assumption.

\begin{cor}\label{AlmFinGenImage} Let $V$ be a separated smooth quasicompact adic space over $\Spa(C,\cO_C)$, and let $U\subset V$ be a strict quasicompact open subset. Then the image of
\[
H^i(V_\et,\cO^+/p)\to H^i(U_\et,\cO^+/p)
\]
is almost finitely generated for all $i\geq 0$.
\end{cor}

\begin{proof} Our strategy is to compute both sides via compatible Cech spectral sequences associated to coverings of $V$ and $U$ by subsets $V^\prime$, $U^\prime$ to which the previous result applies. By \cite[Lemma 5.4]{ScholzePAdicHodge}, for this strategy to work in cohomological degree $i$, we actually need to run via $N$ spectral sequences, where $N\geq i+2$. Thus, fix some $i$ and $N\geq i+2$.

In a first step, assume that $V$ is affinoid. Below, we will construct a finite index set $J$ along with a cover $U=\bigcup_{j\in J} U_j$ and rational subsets $V_j\subset V$, such that for all $j\in J$, $V_j$ is an affinoid with good coordinates, and $U_j\subset V_j$ is a strict rational subset. Given this data, we can find intermediate strict rational subsets $U_j = U_j^{(N)}\subset \ldots\subset U_j^{(1)} = V_j$. For any subset $S\subset J$, let $U_S^{(k)} = \bigcap_{j\in S} U_j^{(k)}$. Then $V_S := U_S^{(1)}\subset V_j$ is a rational subset, so that $V_S$ has good coordinates; moreover, $U_S^{(k)}\subset V_S$ is a rational subset. This means that Lemma \ref{AlmFinGenBaseCase} applies to $U_S^{(k+1)}\subset U_S^{(k)}$ for all $S\neq\emptyset$, $k=1,\ldots,N-1$. Let $U^{(k)} = \bigcup_{j\in J} U_j^{(k)}\subset V$. For each $k=1,\ldots,N$, there is a Cech spectral sequence
\[
\bigoplus_{S\subset J, |S| = m_1+1} H^{m_2}(U_S^{(k)},\cO^+/p)\Rightarrow H^{m_1+m_2}(U^{(k)},\cO^+/p)\ ,
\]
together with maps between these spectral sequences. Applying Lemma \ref{AlmFinGenBaseCase} and \cite[Lemma 5.4]{ScholzePAdicHodge} gives the result in the case that $V$ is affinoid.

To handle the general case, let us again take a finite index set $J$ along with a cover $U=\bigcup_{j\in J} U_j$ and rational subsets $V_j\subset V$, such that for all $j\in J$, $V_j$ is an affinoid, and $U_j\subset V_j$ is a strict rational subset. Given this data, we can find intermediate strict rational subsets $U_j = U_j^{(N)}\subset \ldots\subset U_j^{(1)} = V_j$. For any subset $S\subset J$, let $U_S^{(k)} = \bigcap_{j\in S} U_j^{(k)}$. Then $V_S := U_S^{(1)}\subset V_j$ is affinoid (as $V$ is separated); moreover, $U_S^{(k)}\subset V_S$ is a strict subset. This means that the affinoid case already handled applies to $U_S^{(k+1)}\subset U_S^{(k)}$ for all $S\neq\emptyset$, $k=1,\ldots,N-1$. Let $U^{(k)} = \bigcup_{j\in J} U_j^{(k)}\subset V$. For each $k=1,\ldots,N$, there is a Cech spectral sequence
\[
\bigoplus_{S\subset J, |S| = m_1+1} H^{m_2}(U_S^{(k)},\cO^+/p)\Rightarrow H^{m_1+m_2}(U^{(k)},\cO^+/p)\ ,
\]
together with maps between these spectral sequences. Applying the affinoid case already handled and \cite[Lemma 5.4]{ScholzePAdicHodge} gives the result.

It remains to construct the cover $U=\bigcup_{j\in J} U_j$ and $V_j\subset V$ such that for all $j\in J$, $V_j$ is an affinoid with good coordinates and $U_j\subset V_j$ is a strict rational subset. This is similar, but easier, than \cite[Lemma 5.3]{ScholzePAdicHodge}. Pick a point maximal point $x\in U$, with closure $\overline{\{x\}}\subset V$ in $V$. We claim that there is an affinoid subset $V_x\subset V$ containing $\overline{\{x\}}$. For this, we use a result of Temkin, \cite[Theorem 3.1]{Temkin}. This requires some translation, as he works with Berkovich spaces there. As $V$ is a qcqs adic space, it is equivalent to a qcqs rigid space, or to a compact Hausdorff strictly $C$-analytic Berkovich space. The question whether the germ $V_x$ is good is precisely the question whether there is an affinoid neighborhood $V_x$ of $\overline{\{x\}}$. The criterion of Temkin answers this question in terms of the closure $\overline{\{x\}}$ in the adic space, cf. \cite[Remark 2.6]{Temkin}. As $V$ is separated, this closure embeds into the Riemann-Zariski space of the completed residue field $K(x)$ at $x$, by the valuative criterion for separatedness, cf. \cite[\S 1.3]{Huber}. On the other hand, by the assumption that $U$ is strictly contained in $V$, $\overline{\{x\}}$ surjects onto the Riemann-Zariski space of $K(x)$. This identifies $\overline{\{x\}}$ with the Riemann-Zariski space of $K(x)$, which is affinoid in the sense of \cite[\S 1]{Temkin}. This verifies the existence of $V_x$. By \cite[Lemma 5.2]{ScholzePAdicHodge}, we may assume that $V_x$ has good coordinates. We may then find a strict rational subset $U_x\subset V_x$ contained in $U$, and containing $U\cap \overline{\{x\}}$. The union of all $U_x$ is equal to $U$; by quasicompacity, we can find a finite subcover $U=\bigcup_{j\in J} U_j$, along with $V_j\subset V$ such that $U_j$ is strictly contained in $V_j$. This produces the desired cover.
\end{proof}

To formulate the local finiteness result in the current setup, recall that one has the Gross-Hopkins period map at level $0$
\[
\pi_\GH: \cM_{\LT,0}\to \mathbb{P}^{n-1}_{\breve{F}}\ ,
\]
which admits local sections as a map of adic spaces. Here and in the following, we write $\cM_{\LT,0} = \cM_{\LT,\GL_n(\cO)}$ for the space at level $0$. Pick some affinoid open subset $V\subset \mathbb{P}^{n-1}_{\breve{F}}$ such that $\cM_{\LT,0}\to \mathbb{P}^{n-1}_{\breve{F}}$ admits a section $V\to \cM_{\LT,0}$ on $V$, which we fix. Recall that $\mathcal{F}_\pi$ is built from the $\GL_n(F)$-torsor $\cM_{\LT,\infty}\to \mathbb{P}^{n-1}_{\breve{F}}$, which factors over a $\GL_n(\cO)$-torsor $\cM_{\LT,\infty}\to \cM_{\LT,0}$. Therefore, the pullback of $\mathcal{F}_\pi$ to $\cM_{\LT,0}$ (and thus to $V$) depends only on the $\GL_n(\cO)$-representation $\pi|_{\GL_n(\cO)}$. More precisely, for any $\GL_n(\cO)$-representation $\pi_0$, one can define the sheaf $\mathcal{F}_{\pi_0}$ on $(\cM_{\LT,0} / D^\times)_\et$ by setting
\[
\mathcal{F}_{\pi_0}(U) = \Map_{\cont,\GL_n(\cO)\times D^\times}(|U\times_{\cM_{\LT,0}} \cM_{\LT,\infty}|,\pi_0)
\]
for $U\in (\cM_{\LT,0} / D^\times)_\et$. The obvious analogue of Proposition \ref{DefFPi} holds true in this setup.

Pick a rational subset $U\subset V$ which is strictly contained in $V$.

\begin{lem}\label{LocalFiniteness} For any $m\geq 0$, there is a compact open $K_0\subset D^\times$ stabilizing $V$, $U$ and the section $V\to \cM_{\LT,0}$ such that for all $K\subset K_0$ and any admissible smooth representation $\pi_0$ of $\GL_n(\cO)$, the image of the natural map
\[
H^i((V_C/K)_\et,\mathcal{F}_{\pi_0}\otimes \cO^+/p)\to H^i((U_C/K)_\et,\mathcal{F}_{\pi_0}\otimes \cO^+/p)
\]
is almost finitely generated for all $i=0,\ldots,m$.
\end{lem}

\begin{proof} Let $\pi^\reg = C^0(\GL_n(\cO),\F_p)$ be the regular representation of $\GL_n(\cO)$. There is a resolution
\[
0\to \pi_0\to (\pi^\reg)^{n_1}\to (\pi^\reg)^{n_2}\to \ldots
\]
for some integers $n_i\in \Z$. This follows from the anti-equivalence of admissible smooth $\GL_n(\cO)$-representations and finitely generated $\F_p[[\GL_n(\cO)]]$-modules, and Lazard's theorem that $\F_p[[\GL_n(\cO)]]$ is noetherian, cf. e.g. \cite[Theorem 2.1.2, Equation 2.2.12]{EmertonOrd1}. Using exactness of $\pi_0\mapsto \mathcal{F}_{\pi_0}$, this induces an $E_1$-spectral sequence computing $H^i((V/K)_\et,\mathcal{F}_{\pi_0})$ in terms of $H^i((V/K)_\et,\mathcal{F}_{\pi^\reg})$, and similarly for $U$. Filtering $U\subset V$ by $N$ strict inclusions of rational subsets and using \cite[Lemma 5.4]{ScholzePAdicHodge} reduces the lemma to the case $\pi_0 = \pi^\reg$.\footnote{In this step, we need to take $K$ small enough to also stabilize all intermediate rational subsets. As the number of intermediate rational subsets depends on the cohomological degree $i$, we need to take $K$ dependent on $i$. A more careful argument would certainly avoid this.}

Let $V_\infty\to \cM_{\LT,\infty}$ be the pullback of $V\to \cM_{\LT,0}$; then $K$ still acts on $V_\infty\in (\cM_{\LT,\infty})_\et$. Recall that there is the isomorphism between Lubin-Tate and Drinfeld tower at infinite level
\[
\cM_{\LT,\infty}\cong \cM_{\Dr,\infty}\ ,
\]
cf. \cite{FaltingsTwoTowers}, \cite{FarguesTwoTowers}, which is an isomorphism of perfectoid spaces by \cite[Theorem 7.2.3]{ScholzeWeinstein}. Also recall that the Drinfeld tower is a tower of smooth adic spaces
\[
\cM_{\Dr,K}
\]
over $\Spa(\breve{F},\breve{\cO})$, parametrized by compact open subgroups $K\subset D^\times$. By \cite[Theorem 6.5.4]{ScholzeWeinstein}, one has
\[
\cM_{\Dr,\infty}\sim \varprojlim_K \cM_{\Dr,K}\ .
\]
In particular, by Proposition \ref{QuotientPerfectoid}, one has an equivalence of sites
\[
(\cM_{\Dr,\infty}/K)_\et\cong (\cM_{\Dr,K})_\et\ ,
\]
under which $V_\infty$ with its continuous $K$-action descends to some $V_K\in \cM_{\Dr,K}$. Then $V_K$ is a quasicompact separated smooth adic space over $\Spa(\breve{F},\breve{\cO})$. Applying Proposition \ref{QuotientPerfectoid} to $V_K$ then shows that there is an equivalence of sites
\[
(V_\infty/K)_\et\cong (V_K)_\et\ .
\]

\begin{lem} Let $\alpha: (V_\infty/K)_\et\to (V/K)_\et$ denote the projection. There is a quasi-isomorphism of complexes of sheaves on $(V/K)_\et$
\[
R\alpha_\ast \cO_{V_\infty/K}^+/p\cong \mathcal{F}_{\pi^\reg}\otimes \cO_{V/K}^+/p\ .
\]
\end{lem}

\begin{proof} Let $\alpha_m: (V_m/K)_\et\to (V/K)_\et$ denote the projection from the preimage $V_m\subset \cM_{\LT,m}$, where $\cM_{\LT,m} = \cM_{\LT,1+\varpi^m M_n(\cO)}$. Then $(V_m/K)_\et$ is equivalent to the slice of $(V/K)_\et$ over $V_m\in (V/K)_\et$ with its natural $K$-action. As $V_m$ is finite \'etale over $V$, one has an isomorphism
\[
R\alpha_{m\ast} \cO_{V_m/K}^+/p\cong \mathcal{F}_{C(\GL_n(\cO)/(1+\varpi^mM_n(\cO)),\F_p)}\otimes \cO_{V/K}^+/p\ .
\]
Indeed, this can be checked locally on $(V/K)_\et$, and after pullback to the slice of $(V/K)_\et$ over $V_m$, everything decomposes into a direct sum, as $V_m\times_V V_m\cong V_m\times \GL_n(\cO)/(1+\varpi^m M_n(\cO))$.

Also note that $\cO_{V_\infty/K}^+/p = \alpha^\ast \cO_{V/K}^+/p$, as can be checked after pullback to $(V_\infty)_\et$ by Lemma \ref{Conservativity}, where it follows from $V_\infty\sim \varprojlim_m V_m$. As
\[
\mathcal{F}_{\pi^\reg} = \varinjlim \mathcal{F}_{C(\GL_n(\cO)/(1+\varpi^mM_n(\cO)),\F_p)}\ ,
\]
the statement of the lemma translates into the equality
\[
R\alpha_\ast \alpha^\ast \mathcal{G} = \varinjlim_m R\alpha_{m\ast} \alpha_m^\ast \mathcal{G}
\]
for the sheaf $\mathcal{G} = \cO_{V/K}^+/p$ on $(V/K)_\et$. In fact, this is true for any sheaf $\mathcal{G}$ on $(V/K)_\et$, and follows from SGA 4 VI Corollaire 8.7.5 and the identification of $(V_\infty/K)_\et^\sim$ as a projective limit of the fibred topos
\[
(V_m/K)_\et^\sim\cong (V_\infty/(1+\varpi^mM_n(\cO))\times K)_\et \ ,
\]
cf. Proposition \ref{ResidualActionSites}, Proposition \ref{QuotientPerfectoid}.
\end{proof}

Thus, we can rewrite
\[
H^i((V_C/K)_\et,\mathcal{F}_{\pi^\reg}\otimes \cO_{V_C/K}^+/p) = H^i((V_{\infty,C}/K)_\et,\cO_{V_{\infty,C}/K}^+/p)\ ,
\]
which in turn, by Lemma \ref{IdentOXmodp}, can be rewritten as $H^i((V_{K,C})_\et,\cO_{V_{K,C}}^+/p)$. Similarly, we have
\[
H^i((U_C/K)_\et,\mathcal{F}_{\pi^\reg}\otimes \cO_{U_C/K}^+/p) = H^i((U_{K,C})_\et,\cO_{U_{K,C}}^+/p)\ ,
\]
where $U_K\subset \cM_{\Dr,K}$ is defined like $V_K$, so that $U_K$ is a strict open subset of $V_K$. By Corollary \ref{AlmFinGenImage}, the image of
\[
H^i((V_{K,C})_\et,\cO_{V_{K,C}}^+/p)\to H^i((U_{K,C})_\et,\cO_{U_{K,C}}^+/p)
\]
is almost finitely generated, which is what we wanted to prove.
\end{proof}

\begin{cor}\label{FirstGlobalFiniteness} Fix some $j\geq 0$. Then there is some compact open $K_0\subset D^\times$ such that for all open $K\subset K_0$ and admissible smooth representations $\pi$ of $\GL_n(F)$, the cohomology group
\[
H^i((\mathbb{P}^{n-1}_C/K)_\et,\mathcal{F}_\pi\otimes \cO^+/p)
\]
is almost finitely generated for all $i=0,\ldots,j$.
\end{cor}

\begin{proof} This follows from the local statement by picking sufficiently many open affinoid covers of $\mathbb{P}^{n-1}$ satisfying the hypothesis of Lemma \ref{LocalFiniteness} and using \cite[Lemma 5.4]{ScholzePAdicHodge}, cf. proof of \cite[Lemma 5.8]{ScholzePAdicHodge}.
\end{proof}

\begin{cor}\label{SecondGlobalFiniteness} For any compact open subgroup $K\subset D^\times$ and admissible smooth representation $\pi$ of $\GL_n(F)$, the cohomology group
\[
H^i((\mathbb{P}^{n-1}_C/K)_\et,\mathcal{F}_\pi\otimes \cO^+/p)
\]
is almost finitely generated for all $i\geq 0$.
\end{cor}

\begin{proof} It is enough to prove this for $i\leq j$ for any fixed $j$. Then the previous corollary shows that the statement is true after replacing $K$ by some open normal $K^\prime\subset K$. But $H^i((\mathbb{P}^{n-1}_C/K)_\et,\mathcal{G})$ is computed by a Hochschild-Serre spectral sequence from $H^{i_1}(K/K^\prime,H^{i_2}((\mathbb{P}^{n-1}_C/K^\prime)_\et,\mathcal{G}))$ for any sheaf $\mathcal{G}$, giving the result in general.
\end{proof}

\begin{cor}\label{FinitenessmodK} For any compact open subgroup $K\subset D^\times$ and admissible smooth representation $\pi$ of $\GL_n(F)$, the cohomology group
\[
H^i((\mathbb{P}^{n-1}_C/K)_\et,\mathcal{F}_\pi)
\]
is finite for all $i\geq 0$, and the map
\[
H^i((\mathbb{P}^{n-1}_C/K)_\et,\mathcal{F}_\pi)\otimes \cO_C/p\to H^i((\mathbb{P}^{n-1}_C/K)_\et,\mathcal{F}_\pi\otimes \cO^+/p)
\]
is an almost isomorphism.
\end{cor}

\begin{proof} The proof is the same as that of \cite[Theorem 5.1]{ScholzePAdicHodge} (which in turn is modelled on that of \cite[\S 3, Theorem 8]{FaltingsAlmostEtale}). Note that the argument there is written in terms of the pro-\'etale site which we have not introduced for $\mathbb{P}^{n-1}_C/K$. One can rewrite the argument entirely in terms of the \'etale site as follows. On the pro-\'etale site of $\mathbb{P}^{n-1}_C$, one can look at the sheaf $\cO^{\flat -}$ which is the quotient of $\cO^\flat$ by the subsheaf of topologically nilpotent elements. It is a simple exercise to present $\cO^{\flat -}$ as a colimit of the sheaf $\cO^+/p$ along suitable transition maps. Namely, first define
\[
\cO^{++}/p = \cO^+/p\otimes_{\cO_C} \mm_C
\]
(which can be written as a colimit according to $\mm_C = \bigcup p^{1/n} \cO_C$). There is an isomorphism $x\mapsto x^p - p^{(p-1)/p} x$ from $\cO^{++}/p^{1/p}$ to $\cO^{++}/p$ (this can be checked on the pro-\'etale site of $\mathbb{P}^{n-1}_C$); let
\[
\widetilde{\Phi}^{-1}: \cO^{++}/p\cong \cO^{++}/p^{1/p}\cong p^{(p-1)/p} \cO^{++}/p\subset \cO^{++}/p
\]
be the inverse of this isomorphism composed with multiplication by $p^{(p-1)/p}$. Then
\[
\cO^{\flat -} = \varinjlim_{\widetilde{\Phi}^{-1}} \cO^{++}/p\ .
\]
This implies that $\cO^{\flat -}$ comes via pullback from the \'etale site of $\mathbb{P}^{n-1}_C$, and in fact from $(\mathbb{P}^{n-1}_C/D^\times)_\et$. One can check that $\cO^{\flat -}$ is a sheaf of $\cO_C^\flat$-modules; let $\cO^{\flat -}[(p^\flat)^k]\subset \cO^{\flat-}$ denote the subsheaf of elements killed by $(p^\flat)^k$. Then there are short exact sequences
\[
0\to \cO^{\flat -}[(p^\flat)^{k_1}]\to \cO^{\flat-}[(p^\flat)^{k_1+k_2}]\to \cO^{\flat-}[(p^\flat)^{k_2}]\to 0\ ,
\]
as well as Frobenius isomorphisms
\[
\cO^{\flat-}[(p^\flat)^k]\cong \cO^{\flat-}[(p^\flat)^{pk}]\ ,
\]
and $\cO^{\flat -}[p^\flat]\cong \cO^{++}/p$ is almost isomorphic to $\cO^+/p$. All of these statements can be checked on the pro-\'etale site of $\mathbb{P}^{n-1}_C$. This implies that the cohomology groups
\[
M_k = H^i((\mathbb{P}^{n-1}_C/K)_\et,\mathcal{F}_\pi\otimes \cO^{\flat-}[(p^\flat)^k])
\]
satisfy the assumptions of \cite[Lemma 2.12]{ScholzePAdicHodge}. In particular, there is an almost isomorphism
\[
H^i((\mathbb{P}^{n-1}_C/K)_\et,\mathcal{F}_\pi\otimes \cO^+/p)^a\cong (\cO_C^a/p)^r
\]
for some integer $r\geq 0$, compatible with Frobenius. By tensoring with the maximal ideal, this gives an actual isomorphism
\[
H^i((\mathbb{P}^{n-1}_C/K)_\et,\mathcal{F}_\pi\otimes \cO^{++}/p)\cong (\mm_C/p)^r\ ,
\]
compatible with Frobenius, which in turn induces an isomorphism
\[
H^i((\mathbb{P}^{n-1}_C/K)_\et,\mathcal{F}_\pi\otimes \cO^{\flat-})\cong (C^\flat/\mm_{C^\flat})^r
\]
compatible with Frobenius, by passing through the colimit defining $\cO^{\flat -}$. But there is an Artin-Schreier sequence
\[
0\to \mathcal{F}_\pi\to \mathcal{F}_\pi\otimes \cO^{\flat -}\to \mathcal{F}_\pi\otimes \cO^{\flat -}\to 0\ ;
\]
exactness can be checked over the pro-\'etale site of $\mathbb{P}^{n-1}_C$, where $\mathcal{F}_\pi$ is locally trivial, and the result follows from the Artin-Schreier sequence for $\cO^{\flat -}$. Now the long exact cohomology sequence
\[
\ldots\to H^i((\mathbb{P}^{n-1}_C/K)_\et,\mathcal{F}_\pi)\to H^i((\mathbb{P}^{n-1}_C/K)_\et,\mathcal{F}_\pi\otimes \cO^{\flat-})\to H^i((\mathbb{P}^{n-1}_C/K)_\et,\mathcal{F}_\pi\otimes \cO^{\flat-})\to \ldots
\]
implies that
\[
H^i((\mathbb{P}^{n-1}_C/K)_\et,\mathcal{F}_\pi)\cong \F_p^r\ ,
\]
as $\varphi - 1$ is surjective on $C^\flat/\mm_{C^\flat}$ with kernel $\F_p$.
\end{proof}

\begin{lem}\label{HochschildSerre} For any compact open subgroup $K\subset D^\times$ and any sheaf $\mathcal{G}$ on $(\mathbb{P}^{n-1}_C/K)_\et$, there is a Hochschild-Serre spectral sequence
\[
H^{i_1}_\cont(K,H^{i_2}(\mathbb{P}^{n-1}_C,\mathcal{G}))\Rightarrow H^{i_1+i_2}((\mathbb{P}^{n-1}_C/K)_\et,\mathcal{G})\ .
\]
\end{lem}

Here, all $K$-modules are considered as discrete.

\begin{proof} One gets the spectral sequence as a direct limit over $K^\prime$ of spectral sequences
\[
H^{i_1}(K/K^\prime,H^{i_2}((\mathbb{P}^{n-1}_C/K^\prime)_\et,\mathcal{G}))\Rightarrow H^{i_1+i_2}((\mathbb{P}^{n-1}_C/K)_\et,\mathcal{G})\ ,
\]
which one gets as Cartan-Leray spectral sequences for the covering $\mathbb{P}^{n-1}_C\times_{K^\prime} K\to \mathbb{P}^{n-1}_C$ in $(\mathbb{P}^{n-1}_C/K)_\et$, under the identification of $(\mathbb{P}^{n-1}_C/K^\prime)_\et$ with a slice of $(\mathbb{P}^{n-1}_C/K)_\et$.
\end{proof}

Also recall the following lemma about continuous group cohomology of $p$-adic Lie groups.

\begin{lem}\label{AdmissibleFiniteCohom} Let $G$ be a compact $p$-adic Lie group, and let $\pi$ be an admissible smooth $\F_p$-representation of $G$. Then $H^i_\cont(G,\pi)$ is finite-dimensional for all $i\geq 0$.
\end{lem}

\begin{proof} Under the identification of continuous group cohomology with the derived functor of $G$-invariants in case $G$ is compact, cf. \cite[Proposition 2.2.6]{EmertonOrd2}, this follows from the anti-equivalence of admissible smooth representations of $G$ with finitely generated $\F_p[[G]]$-modules, cf. \cite[2.2.12]{EmertonOrd1}, and Lazard's theorem that $\F_p[[G]]$ is noetherian, cf. e.g. \cite[Theorem 2.1.2]{EmertonOrd1}.
\end{proof}

\begin{cor} For any admissible smooth representation $\pi$ of $\GL_n(F)$, the cohomology group
\[
H^i_\et(\mathbb{P}^{n-1}_C,\mathcal{F}_\pi)
\]
is an admissible $D^\times$-representation invariant under change of $C$. The map
\[
H^i_\et(\mathbb{P}^{n-1}_C,\mathcal{F}_\pi)\otimes_{\F_p} \cO_C/p\to H^i_\et(\mathbb{P}^{n-1}_C,\mathcal{F}_\pi\otimes \cO^+/p)
\]
is an almost isomorphism.
\end{cor}

\begin{proof} The almost isomorphism follows by passing to the direct limit over $K$ in Corollary \ref{FinitenessmodK}, using Lemma \ref{ResidualActionSites}. E.g. by computing the right-hand side using the pro-\'etale site, and a simplicial affinoid perfectoid cover over which $\mathcal{F}_\pi$ is free, one sees that enlarging $C\hookrightarrow C^\prime$, the map
\[
H^i_\et(\mathbb{P}^{n-1}_C,\mathcal{F}_\pi\otimes \cO^+/p)\otimes_{\cO_C/p} \cO_{C^\prime}/p\to H^i_\et(\mathbb{P}^{n-1}_{C^\prime},\mathcal{F}_\pi\otimes \cO^+/p)
\]
is an almost isomorphism. Here, use that if $X=\Spa(R,R^+)$ is an affinoid perfectoid space over $\Spa(C,\cO_C)$ with base-change $X^\prime = \Spa(R^\prime,R^{\prime +})$ to $\Spa(C^\prime,\cO_{C^\prime})$, then $R^{\prime +}/p$ is almost isomorphic to $R^+/p\otimes_{\cO_C/p} \cO_{C^\prime}/p$, cf. proof of \cite[Proposition 6.18]{ScholzePerfectoid}. This implies that $H^i_\et(\mathbb{P}^{n-1}_C,\mathcal{F}_\pi)$ is invariant under change of $C$.

To check that $H^i_\et(\mathbb{P}^{n-1}_C,\mathcal{F}_\pi)$ is an admissible $D^\times$-representation, we argue by induction on $i$; thus, assume the result is known for $i^\prime<i$. We need to show that for any compact open subgroup $K\subset D^\times$, the space
\[
H^i_\et(\mathbb{P}^{n-1}_C,\mathcal{F}_\pi)^K
\]
is finite-dimensional. Consider the Hochschild-Serre spectral sequence
\[
H^{i_1}_\cont(K,H^{i_2}_\et(\mathbb{P}^{n-1}_C,\mathcal{F}_\pi))\Rightarrow H^{i_1+i_2}((\mathbb{P}^{n-1}_C/K)_\et,\mathcal{F}_\pi)
\]
from Lemma \ref{HochschildSerre}. In particular, consider the contributions on the diagonal $i_1+i_2=i$. For $i_2<i$, the group $H^{i_2}_\et(\mathbb{P}^{n-1}_C,\mathcal{F}_\pi)$ is admissible as $K$-representation by induction, which by Lemma \ref{AdmissibleFiniteCohom} implies that
\[
H^{i_1}_\cont(K,H^{i_2}_\et(\mathbb{P}^{n-1}_C,\mathcal{F}_\pi))
\]
is finite-dimensional if $i_2<i$. The only other contribution to the diagonal $i_1+i_2=i$ comes from
\[
H^i_\et(\mathbb{P}^{n-1}_C,\mathcal{F}_\pi)^K\ .
\]
Assume it was infinite-dimensional. In the spectral sequence, it only interacts with terms where $i_2<i$, and only finitely many such. This gives only a finite-dimensional space, so an infinite-dimensional space survives to the $E_\infty$-page, which contributes an infinite-dimensional space to $H^i((\mathbb{P}^{n-1}_C/K)_\et,\mathcal{F}_\pi)$. However, this space is finite by Corollary \ref{FinitenessmodK}. Thus, we see that $H^i_\et(\mathbb{P}^{n-1}_C,\mathcal{F}_\pi)$ is an admissible $D^\times$-representation, as desired.
\end{proof}

To complete the proof of Theorem \ref{Finiteness}, it remains to verify the vanishing statements in large degrees. We claim that
\[
H^i_\et(\mathbb{P}^{n-1}_C,\mathcal{F}_\pi\otimes \cO^+/p)
\]
is almost zero for $i>2(n-1)$ in general, and $i>n-1$ if $\pi$ is injective as $\GL_n(\cO)$-representation. As the cohomological dimension of $|\mathbb{P}^{n-1}_C|$ is $n-1$, it is enough to prove that under the projection $\lambda: (\mathbb{P}^{n-1}_C)_\et\to |\mathbb{P}^{n-1}_C|$,
\[
R^i\lambda_\ast (\mathcal{F}_\pi\otimes \cO^+/p)
\]
is almost for $i>n-1$, and for $i>0$ if $\pi$ is injective as $\GL_n(\cO)$-representation. It is enough to prove this after pullback to $\cM_{\LT,0,C}$, as $\cM_{\LT,0,C}\to \mathbb{P}^{n-1}_C$ admits local sections. After this pullback, $\mathcal{F}_\pi$ can be written as the inductive limit of the $\mathcal{F}_{\pi^H}$ over all open subgroups $H\subset \GL_n(\cO)$, all of which are $\F_p$-local systems. Thus, almost vanishing for $i>n-1$ follows from \cite[Lemma 5.6]{ScholzePAdicHodge}.

Now assume $\pi$ is injective as $\GL_n(\cO)$-representation. Then we may write $\pi$ as $\GL_n(\cO)$-representation as a direct summand of a power of $\pi^\reg$. Thus, we can reduce to the case that $\pi=\pi^\reg$. In that case, we have to compute
\[
R^i f_\ast \cO^+/p
\]
for the projection $f: (\cM_{\LT,\infty,C})_\et\to |\cM_{\LT,0,C}|$. But by \cite[Lemma 2.10.1]{Weinstein} (and its proof), the space $|\cM_{\LT,0,C}|$ is covered by open affinoid $U$ whose preimage $U_\infty\subset \cM_{\LT,\infty,C}$ is affinoid perfectoid, so that by \cite[Proposition 7.13]{ScholzePerfectoid}, the higher cohomology groups $H^i((U_\infty)_\et,\cO^+/p)$ are almost zero for $i>0$.

\section{Admissible representations: General base rings}

In this section, we want to extend the finiteness results from the previous section to admissible representations of $\GL_n(F)$ over more general base rings.

\begin{definition}[{\cite[\S 2]{EmertonOrd1}}] Let $(A,\mm)$ be a complete noetherian local ring with finite residue field of characteristic $p$, and $G$ a $p$-adic analytic group. An $A[G]$-module $V$ is called smooth if for all $v\in V$, there is some open subgroup $H\subset G$ and $i\geq 1$ such that $v$ is $H$-invariant and $\mm^i v = 0$.

A smooth $A[G]$-module $V$ is called admissible if for all $i\geq 1$ and $H\subset G$ open, the $A/\mm^i$-module $V^H[\mm^i]$ is finitely generated (equivalently, of finite length).
\end{definition}

\begin{rem} In case $A=\Z_p$, the representations live on $p$-torsion modules like $\Q_p/\Z_p$. In geometric settings, one gets such representations by considering the cohomology with $\Q_p/\Z_p$-coefficients (which carries essentially the same information as completed cohomology with $\Z_p$-coefficients).
\end{rem}

We recall that the category of admissible $A[G]$-modules is well-behaved.

\begin{thm}[{\cite[Proposition 2.2.13]{EmertonOrd1}}]\label{AdmSerreSub} The category of admissible $A[G]$-modules is abelian, and a Serre subcategory of the category of smooth $A[G]$-modules.
\end{thm}

In this section, we prove the following generalization of Theorem \ref{Finiteness}.

\begin{thm}\label{FinitenessBaseRing} Let $(A,\mm)$ be a complete noetherian local ring with finite residue field of characteristic $p$. Let $V$ be an admissible $A[\GL_n(F)]$-module, and let $\mathcal{F}_V$ be the corresponding sheaf on $(\mathbb{P}^{n-1}_C/D^\times)_\et$. For all $i\geq 0$, the $D^\times$-representation
\[
H^i_\et(\mathbb{P}^{n-1}_C,\mathcal{F}_V)
\]
is admissible, independent of $C$, and vanishes for $i>2(n-1)$. The natural map
\[
H^i_\et(\mathbb{P}^{n-1}_C,\mathcal{F}_V)\otimes_{\Z_p} \cO_C\to H^i_\et(\mathbb{P}^{n-1}_C,\mathcal{F}_V\otimes \cO^+)
\]
is an almost isomorphism.
\end{thm}

\begin{rem} Emerton also introduces the notion of $p$-adically admissible representations in \cite[Definition 2.4.7]{EmertonOrd1}, making it possible to say that completed cohomology (which is a $p$-adically complete $\Z_p$-module) itself is admissible. An obvious variant holds for this notion of admissibility as well.
\end{rem}

\begin{proof} Note first that by Proposition \ref{ResidualActionSites}, we have
\[
H^i_\et(\mathbb{P}^{n-1}_C,\mathcal{F}_V) = \varinjlim_{K\subset D^\times} H^i((\mathbb{P}^{n-1}_C/K)_\et,\mathcal{F}_V)\ ,
\]
where $K\subset D^\times$ acts trivially on $H^i_\et((\mathbb{P}^{n-1}_C/K),\mathcal{F}_V)$. Moreover, each site $(\mathbb{P}^{n-1}_C/K)_\et$ is coherent, so as $V=\varinjlim V[\mm^j]$, we have
\[
H^i((\mathbb{P}^{n-1}_C/K)_\et,\mathcal{F}_V) = \varinjlim H^i((\mathbb{P}^{n-1}_C/K)_\et,\mathcal{F}_{V[\mm^j]})\ ,
\]
where $\mm^j$ annihilates $H^i((\mathbb{P}^{n-1}_C/K)_\et,\mathcal{F}_{V[\mm^j]})$. It follows that $H^i_\et(\mathbb{P}^{n-1}_C,\mathcal{F}_V)$ is smooth. To prove admissibility, we now have Theorem \ref{AdmSerreSub} available.

Assume first that $V$ is killed by $\mm^j$. We induct on the minimal such $j$; for $j=1$, the result is given by Theorem \ref{Finiteness}. Now look at the exact sequence
\[
0\to V[\mm]\to V\to \overline{V}\to 0\ ,
\]
where $\overline{V} = V/V[\mm]$. It induces a long exact sequence
\[
\ldots\to H^i_\et(\mathbb{P}^{n-1}_C,\mathcal{F}_{V[\mm]})\to H^i_\et(\mathbb{P}^{n-1}_C,\mathcal{F}_V)\to H^i_\et(\mathbb{P}^{n-1}_C,\mathcal{F}_{\overline{V}})\to \ldots\ .
\]
The outer two terms are admissible by induction. This implies, by Theorem \ref{AdmSerreSub}, that the middle term is admissible as well. Using the $5$-lemma, one also proves the almost isomorphism by induction.

In general, the almost isomorphism follows by writing
\[
H^i_\et(\mathbb{P}^{n-1}_C,\mathcal{F}_V) = \varinjlim H^i_\et(\mathbb{P}^{n-1}_C,\mathcal{F}_{V[\mm^j]})\ ,
\]
and similarly for the $\cO^+/p$-twisted cohomology groups. For admissibility, we induct on $i$, so assume that for all admissible $A[\GL_n(F)]$-modules $V$, $H^{i^\prime}_\et(\mathbb{P}^{n-1}_C,\mathcal{F}_V)$ is admissible for $i^\prime<i$. Fix some $j$ and $V$, and generators $(f_1,\ldots,f_n) = \mm^j$. There is an exact sequence
\[
0\to V[\mm^j]\to V\buildrel{(f_1,\ldots,f_n)}\over\longrightarrow V^n\ .
\]
Let $\overline{V} = V / V[\mm^j]$ and $W = \coker( V\buildrel{(f_1,\ldots,f_n)}\over\longrightarrow V^n) = V^n / \overline{V}$, both of which are admissible $A[\GL_n(F)]$-modules. Now
\[
H^i_\et(\mathbb{P}^{n-1}_C,\mathcal{F}_V)[\mm^j] = \ker\left(H^i_\et(\mathbb{P}^{n-1}_C,\mathcal{F}_V)\to H^i_\et(\mathbb{P}^{n-1}_C,\mathcal{F}_{V^n})\right)[\mm^j]\ ,
\]
and there is an exact sequence
\[\begin{aligned}
0&\to \ker\left(H^i_\et(\mathbb{P}^{n-1}_C,\mathcal{F}_V)\to H^i_\et(\mathbb{P}^{n-1}_C,\mathcal{F}_{\overline{V}})\right)[\mm^j]\\
&\to \ker\left(H^i_\et(\mathbb{P}^{n-1}_C,\mathcal{F}_V)\to H^i_\et(\mathbb{P}^{n-1}_C,\mathcal{F}_{V^n})\right)[\mm^j]\\
&\to \ker\left(H^i_\et(\mathbb{P}^{n-1}_C,\mathcal{F}_{\overline{V}})\to H^i_\et(\mathbb{P}^{n-1}_C,\mathcal{F}_{V^n})\right)[\mm^j]\ .
\end{aligned}\]
Therefore, it is enough to show that the two outer terms
\[
\ker\left(H^i_\et(\mathbb{P}^{n-1}_C,\mathcal{F}_V)\to H^i_\et(\mathbb{P}^{n-1}_C,\mathcal{F}_{\overline{V}})\right)\ ,\ \ker\left(H^i_\et(\mathbb{P}^{n-1}_C,\mathcal{F}_{\overline{V}})\to H^i_\et(\mathbb{P}^{n-1}_C,\mathcal{F}_{V^n})\right)
\]
are admissible $A[D^\times]$-modules. But the first admits a surjection from the admissible $A[D^\times]$-module
\[
H^i_\et(\mathbb{P}^{n-1}_C,\mathcal{F}_{V[\mm^j]})\ ,
\]
and the second from the $A[D^\times]$-module
\[
H^{i-1}_\et(\mathbb{P}^{n-1}_C,\mathcal{F}_W)\ ,
\]
which is admissible by induction.
\end{proof}

We end this section with two results of general nature. First, we observe that the Weil group action extends to a Galois group action. Let $I_F\subset W_F\subset \Gal_F$ be the inertia, Weil, and Galois group of $F$, respectively.

\begin{prop}\label{GalAction} Let $(A,\mm)$ be a complete noetherian local ring with finite residue field of characteristic $p$. Let $V$ be an admissible $A[\GL_n(F)]$-module, and let $\mathcal{F}_V$ be the corresponding sheaf on $(\mathbb{P}^{n-1}_C/D^\times)_\et$, where $C=\C_p$ (and $k=\bar{\F}_p$). Then the natural $W_F$-action on
\[
H^i_\et(\mathbb{P}^{n-1}_C,\mathcal{F}_V)\ ,
\]
coming from the $I_F$-action on $C=\C_p$ and the Weil descent datum, is continuous and extends (necessarily uniquely) to a continuous action of $\Gal_F$.
\end{prop}

\begin{proof} Writing $V$ as the union of $V[\mm^j]$, we may assume that $\mm^j=0$ for some $j$, so that $A$ is finite. Continuity of the $W_F$-action reduces to continuity of the $I_F$-action, which follows from writing
\[
H^i_\et(\mathbb{P}^{n-1}_{\C_p},\mathcal{F}_V) = \varinjlim_{M/\breve{F}} H^i_\et(\mathbb{P}^{n-1}_M, \mathcal{F}_V)
\]
as a direct limit over finite extensions $M$ of $\breve{F}$ contained in $\C_p$. Now for all compact open subgroups $K\subset D^\times$, the group
\[
H^i_\et(\mathbb{P}^{n-1}_{\C_p},\mathcal{F}_V)^K
\]
is finite. But any continuous action of $W_F$ on a finite set extends continuously to $\Gal_F$. Namely, an open subgroup $I_0\subset I_F$ acts trivially, and it remains to extend the $W_F/I_0$-action to a $\Gal_F/I_0$-action. This follows by observing that some power of any fixed Frobenius element acts trivially, as any element of a finite group is of finite order.
\end{proof}

Moreover, one can always compute $H^0$.

\begin{prop}\label{ComputeH0} Let $(A,\mm)$ be a complete noetherian local ring with finite residue field of characteristic $p$. Let $V$ be an admissible $A[\GL_n(F)]$-module, and let $\mathcal{F}_V$ be the corresponding sheaf on $(\mathbb{P}^{n-1}_C/D^\times)_\et$. Then the natural map
\[
H^0_\et(\mathbb{P}^{n-1}_C,\mathcal{F}_{V^{\SL_n(F)}})\hookrightarrow H^0_\et(\mathbb{P}^{n-1}_C,\mathcal{F}_V)
\]
is an isomorphism. It acquires an action of $\GL_n(F) / \SL_n(F) = F^\times$ (via the determinant map), and the group $W_F\times D^\times$ acts via the map $W_F\times D^\times\to F^\times$ given by the inverse of the product of the Artin reciprocity map (sending geometric Frobenii to uniformizers) and the reduced norm.
\end{prop}

\begin{proof} This follows from the identification of the geometric connected components
\[
\pi_0 \cM_{\LT,\infty,C} = F^\times
\]
and the identification of the $\GL_n(F)\times W_F\times D^\times$-action given by Strauch, \cite{StrauchConnComp}.
\end{proof}

\section{Shimura curves}\label{ShimuraCurves}

For the global situation, we change notation slightly. Let now $F$ be totally real number field with a fixed place $\frap$ above $p$ and fixed infinite place $\infty_F$, and $D_0/F$ a quaternion algebra which is definite at all infinite places, and split at $\frap$. Let $G=D_0^\times$ be the algebraic group of units of $D_0$, and let $D^\times/F$ be the nontrivial inner form of $G$ which is isomorphic to $G$ away from $\frap$ and $\infty_F$. Then, as notation suggests, $D^\times$ is the algebraic group of units of a quaternion algebra $D/F$; it is a division algebra at $\frap$, and split at $\infty_F$. Fix an identification
\[
G(\A_{F,f}^\frap)\cong D(\A_{F,f}^\frap)\ .
\]
Our previous local objects are given by $F_\frap$, $G(F_\frap)\cong \GL_2(F_\frap)$ and $D^\times_\frap = D^\times(F_\frap)$.

Associated with $D^\times/F$ (or rather $\Res_{F/\Q} D^\times$) and the conjugacy class of
\[
h: \bS=\Res_{\C/\R} \G_m\to (\Res_{F/\Q} D^\times)_\R = \prod_{\tau: F\hookrightarrow \R} D^\times\otimes_{F,\tau} \R
\]
which is trivial in all components different from $\infty_F$, and equal to
\[
a+ib\in \bS(\R) = \C^\times\mapsto \left(\begin{array}{cc} a & b\\ -b & a \end{array}\right)\in (D^\times\otimes_{F,\infty_F} \R)(\R)\cong \GL_2(\R)
\]
in the component of $\infty_F$, one has (a tower of) Shimura curves $\Sh_U / F$ parametrized by (sufficiently small) compact open $U\subset D^\times(\A_{F,f})$.

Fix some tame level, i.e. a compact open subgroup $U^\frap\subset G(\A_{F,f}^\frap)$. Then $U^\frap$ is of the form $U^\frap = U^\frap_S U^S$, where $S$ is a finite set of finite places of $F$ containing all places above $p$, $U^\frap_S\subset G(\A_{F,S}^\frap)$ is compact open, and $U^S = \prod_{v\not\in S} \GL_2(\cO_{F_v})\subset \GL_2(\A_{F,f}^S)\cong G(\A_{F,f}^S)$ is a product of hyperspecial maximal compact open subgroups. We consider the Hecke algebra
\[
\T = \T_S = \Z[\GL_2(\A_{F,f}^S)//U^S] = \prod_{v\not\in S} \T_v\ ,
\]
where
\[
\T_v = \Z[\GL_2(F_v)//\GL_2(\cO_{F_v})]\cong \Z[T_v,S_v^{\pm 1}]\ .
\]
Here, as usual, $T_v$ is the Hecke operator corresponding to the double coset
\[
\GL_2(\cO_{F_v})\left(\begin{array}{cc} \varpi_v & 0\\0 & 1\end{array}\right) \GL_2(\cO_{F_v})\ ,
\]
and $S_v$ is the one corresponding to
\[
\GL_2(\cO_{F_v})\left(\begin{array}{cc} \varpi_v & 0\\0 & \varpi_v\end{array}\right) \GL_2(\cO_{F_v})\ ,
\]
where $\varpi_v$ is a local uniformizer at $v$. Moreover, let us fix an absolutely irreducible representation
\[
\overline{\sigma}: \Gal_{F,S}\to \GL_2(\F_q)\ ,
\]
where $\Gal_{F,S}$ is the Galois group of the maximal extension of $F$ unramified outside $S$ and $q$ is a power of $p$. This gives rise to a maximal ideal $\mm = \mm_{\overline{\sigma}}$ of $\T$, given as the kernel of the map $\T\to \F_q$ sending $T_v$ to $\tr(\overline{\sigma}(\Frob_v))$ and $S_v$ to $q_v \det(\overline{\sigma}(\Frob_v))$ for $v\not\in S$, where $\Frob_v\in \Gal_{F,S}$ is a Frobenius element, and $q_v$ is the cardinality of the residue field at $v$.

The Hecke algebra $\T$ acts on $H^i(\Sh_{K U^\frap,\C},\Z_p)$ for all compact open $K\subset D^\times_\frap$. Observe that, as $\overline{\sigma}$ is absolutely irreducible, the localization
\[
H^i(\Sh_{K U^\frap,\C},\Z_p)_{\mm} = 0
\]
at $\mm$ vanishes for $i\neq 1$. Indeed, in degree $0$, the action of $D^\times(\A_{F,f})$ factors through the determinant (i.e., reduced norm) $\det: D^\times(\A_{F,f})\to \A_{F,f}^\times$, so that in particular the associated Galois representations are reducible. By Poincar\'e duality (with $\F_p$-coefficients), the same applies to $i=2$, leaving only $i=1$. Also note that this implies that
\[
H^1(\Sh_{K U^\frap,\C},\Z_p)_{\mm}
\]
is torsion-free. To avoid trivialities, we assume that it is nonzero, i.e. $\overline{\sigma}$ is modular.

Let $\T(K U^\frap)$ be the image of $\T$ in $\End(H^1(\Sh_{K U^\frap,\C},\Z))$, and let $\T(K U^\frap)_{\mm}$ be its $\mm$-adic completion. Then $\T(K U^\frap)_{\mm}$ acts faithfully on
\[
H^1(\Sh_{K U^\frap,\C},\Z_p)_{\mm}\ .
\]
There is an associated Galois representation.

\begin{thm}[\cite{CarayolHilbert}, \cite{CarayolGalRepr}]\label{ThmExGalRepr} There is a unique (up to conjugation) continuous $2$-dimensional Galois representation
\[
\sigma = \sigma_\mm: \Gal_{F,S}\to \GL_2(\T(K U^\frap)_\mm)
\]
unramified outside $S$, such that for all $v\not\in S$,
\[
\tr(\sigma(\Frob_v)) = T_v\ ,\ \det(\sigma(\Frob_v)) = q_v S_v\ .
\]
\end{thm}

\begin{proof} From \cite{CarayolHilbert}, one gets the existence of Galois representations for the $\Q_p$-cohomology of $\Sh_{K U^\frap}$, which in particular (as $\T(K U^\frap)$ is reduced) gives a representation
\[
\Gal_{F,S}\to \GL_2(\T(K U^\frap)_\mm[1/p])\ .
\]
On the other hand, all characteristic polynomials of Frobenii take values in $\T(K U^\frap)_\mm$, so one gets a determinant with values in $\T(K U^\frap)_\mm$, cf. \cite[Example 2.32]{ChenevierDet}. But $\overline{\sigma}$ is absolutely irreducible by assumption, so there is a representation $\sigma$ as desired by \cite[Theorem 2.22]{ChenevierDet}.\footnote{Carayol in \cite{CarayolGalRepr} gives a different argument for this gluing. For $p\neq 2$, one can use the theory of pseudorepresentations in place of determinants.}
\end{proof}

Note that in particular, $\overline{\sigma} = \sigma\mod \mm$. In fact, one sees $\sigma$ in $H^1(\Sh_{K U^\frap,\overline{F}},\Z_p)_\mm$: We want to prove that
\[
H^1(\Sh_{K U^\frap,\overline{F}},\Z_p)_\mm = \sigma\otimes_{\T(K U^\frap)_\mm} \rho
\]
for some $\T(K U^\frap)_\mm$-module $\rho$ on which $\Gal_F$ acts trivially. It turns out that there are some useful general lemmas about such situations.

\begin{definition}\label{DefTypic} Let $(R,\mm_R)$ be a noetherian local ring, $G$ some group, and
\[
\sigma_R: G\to \GL_n(R)
\]
an $n$-dimensional representation such that $\overline{\sigma}_R = \sigma_R\mod \mm_R$ is absolutely irreducible. Let $M$ be an $R[G]$-module. Then $M$ is said to be $\sigma_R$-typic if one can write $M$ as a tensor product
\[
M = \sigma_R\otimes_R M_0\ ,
\]
where $M_0$ is an $R$-module, and $G$ acts only through its action on $\sigma_R$.
\end{definition}

\begin{prop}\label{TypicEquiv} In the situation of Definition \ref{DefTypic}, if $M$ is $\sigma_R$-typic, then
\[
M_0 = \Hom_{R[G]} (\sigma_R,M)\ .
\]
The functors $M_0\mapsto \sigma_R\otimes_R M_0$, $M\mapsto \Hom_{R[G]}(\sigma_R,M)$ induce an equivalence of categories between the category of $\sigma_R$-typic $R[G]$-modules and the category of $R$-modules.
\end{prop}

\begin{proof} It is enough to prove that for any $R$-module $M_0$, the natural map
\[
M_0\to \Hom_{R[G]}(\sigma_R,\sigma_R\otimes_R M_0)
\]
is an isomorphism. As both sides commute with filtered colimits, we may assume that $M_0$ is finitely generated. Filtering by modules generated by one element, one can reduce to the case that $M_0 = R/I$ for some ideal $I\subset R$. Replacing $R$ by $R/I$, we can assume that $M_0=R$. In that case, we have to prove that
\[
\End_{R[G]}(\sigma_R) = R\ .
\]
But this follows from the assumption that $\overline{\sigma}_R$ is absolutely irreducible.
\end{proof}

\begin{prop}\label{TypicSubmodule} In the situation of Definition \ref{DefTypic}, assume that $M$ is $\sigma_R$-typic, and $N\subset M$ is an $R[G]$-submodule. Then $N$ is $\sigma_R$-typic.
\end{prop}

\begin{proof} Write $M=\sigma_R\otimes_R M_0$ as usual. We may assume that $M_0$ is finitely generated, by writing $M_0$ as a filtered colimit of finitely generated modules $M_{0,i}$ and $N$ as the filtered colimit of $N\cap (\sigma_R\otimes_R M_{0,i})$ (noting that the category of $\sigma_R$-typic modules is closed under filtered colimits).

We may further replace $M_0$ by the image of
\[
\sigma_R^\vee\otimes_R N\hookrightarrow \sigma_R^\vee\otimes_R \sigma_R\otimes_R M_0\to M_0\ .
\]
After this reduction, we claim that $N=M$. If $N\neq M$, then by Nakayama $N\to M/\mm_R M$ is not surjective. The image of $N$ in $M/\mm_R M = \overline{\sigma}_R\otimes_{R/\mm_R} M_0/\mm_R$ is equal to $\overline{\sigma}_R\otimes_{R/\mm_R} \overline{N}$ for some $R/\mm_R$-vector space $\overline{N}$, as $\overline{\sigma}_R$ is absolutely irreducible, and $N$ is $G$-stable. Let $M_0^\prime$ be the kernel of the composite map $M_0\to M_0/\mm_R\to (M_0/\mm_R)/\overline{N}$, which is a proper submodule of $M_0$. On the other hand, the image of $\sigma_R^\vee\otimes_R N\to M_0$ is contained in $M_0^\prime$, as $N$ is contained in the $\sigma_R$-typic module $\sigma_R\otimes_R M_0^\prime$. This is a contradiction, finishing the proof.
\end{proof}

For later use, we record the following lemma.

\begin{lem}\label{TypicDeterminesRepr} Let $M$ be an $R[G]$-module which is faithful as $R$-module (i.e., the map $R\to \End(M)$ is injective). Assume that $M$ is $\sigma_R$-typic and $\sigma^\prime_R$-typic, for two representations $\sigma_R$, $\sigma^\prime_R$ as above. Moreover, assume that $R$ is henselian. Then $\sigma_R\cong \sigma^\prime_R$.
\end{lem}

\begin{proof} By checking over $R/\mm_R$, one sees that $\overline{\sigma}_R\cong \overline{\sigma^\prime}_R$; in particular, $\sigma_R$ and $\sigma^\prime_R$ are of the same dimension. By \cite[Theorem 2.22]{ChenevierDet}, it is enough to prove that the determinants associated with $\sigma_R$ and $\sigma^\prime_R$ agree, i.e. for all $g\in G$, the characteristic polynomials of $\sigma(g)$ and $\sigma^\prime(g)$ agree. For this, it is enough to find ideals $I_i\subset R$ with empty intersection, such that the determinants agree modulo $I_i$ for all $i$. Write $M=\sigma_R\otimes_R M_0$ for some $R$-module $M_0$, which is necessarily faithful. For each element $m\in M_0$, one has the annihilator $I_m = \Ann(m)\subset R$. By faithfulness, the intersection of all $I_m$ is trivial. Thus, we may work modulo $I_m$. Note that $\sigma_R\otimes_R R/I_m\hookrightarrow M$ sending $x\otimes 1$ to $x\otimes m$ is an $R[G]$-submodule. By Proposition \ref{TypicSubmodule}, $\sigma_R\otimes_R R/I_m$ is still $\sigma^\prime_R$-typic, so
\[
\sigma_R\otimes_R R/I_m\cong (\sigma^\prime_R\otimes_R R/I_m)\otimes_{R/I_m} A
\]
for some $R/I_m$-module $A$. The isomorphism implies that $A^{\dim \sigma_R}\cong (R/I_m)^{\dim \sigma_R}$, which implies that $A$ is finite projective of rank $1$ as $R/I_m$-module, i.e. a line bundle. As $R/I_m$ is local, it follows that $A\cong R/I_m$ is free. Thus, $\sigma_R$ and $\sigma^\prime_R$ are isomorphic after reduction to $R/I_m$, which finishes the proof.
\end{proof}

Now one has the following theorem, due to Carayol: In \cite{CarayolHilbert}, he gives a description of the $\Q_p$-cohomology, and in \cite{CarayolGalRepr}, he explains how to get an integral statement.

\begin{thm}\label{H1Typic} The $\T(K U^\frap)_\mm[\Gal_{F,S}]$-module $H^1(\Sh_{K U^\frap,\overline{F}},\Z_p)_\mm$ is $\sigma$-typic.
\end{thm}

\begin{proof} By Proposition \ref{TypicSubmodule}, it is enough to prove that
\[
H^1(\Sh_{K U^\frap,\overline{F}},\Z_p)_\mm[1/p]
\]
is $\sigma$-typic. But this follows from the description of the $\Q_p$-cohomology of $\Sh_{K U^\frap,\overline{F}}$ by Carayol, \cite{CarayolHilbert}.
\end{proof}

At this point, we can also pass to completed cohomology. Let
\[
\widehat{H}^1(U^\frap,\Z_p) = \varprojlim_m \varinjlim_K H^1(\Sh_{K U^\frap,\overline{F}},\Z/p^m\Z)\ ,
\]
and
\[
\widehat{H}^1(U^\frap,\Z_p)_\mm = \varprojlim_m \varinjlim_K H^1(\Sh_{K U^\frap,\overline{F}},\Z/p^m\Z)_\mm\ ,
\]
Then the inverse limit
\[
\T(U^\frap)_\mm := \varprojlim_K \T(K U^\frap)_\mm
\]
acts faithfully and continuously on $\widehat{H}^1(U^\frap,\Z_p)_\mm$.

\begin{prop} There is a unique (up to conjugation) continuous $2$-dimensional Galois representation
\[
\sigma = \sigma_\mm: \Gal_{F,S}\to \GL_2(\T(U^\frap)_\mm)
\]
unramified outside $S$, such that for all $v\not\in S$,
\[
\tr(\sigma(\Frob_v)) = T_v\ ,\ \det(\sigma(\Frob_v)) = q_v S_v\ .
\]
The ring $\T(U^\frap)_\mm$ is a complete noetherian local ring with finite residue field.
\end{prop}

\begin{proof} One gets a $2$-dimensional determinant with values in $\T(U^\frap)_\mm$ by \cite[Example 2.32]{ChenevierDet}. This gives rise to a representation as $\overline{\sigma}$ is absolutely irreducible, by \cite[Theorem 2.22 (i)]{ChenevierDet}.

For the final assertion, note that the existence of the Galois representation $\sigma$ gives a map from the Galois deformation ring $R_{\overline{\sigma},S}$ to $\T(U^\frap)_\mm$. This map is necessarily surjective, as $T_v$ and $S_v$ can be recovered from the image of Frobenius elements $\Frob_v$. As $R_{\overline{\sigma},S}$ is a complete noetherian local ring with finite residue field, so is $\T(U^\frap)_\mm$.
\end{proof}

\begin{prop}\label{CompH1Typic} The $\T(U^\frap)_\mm[\Gal_{F,S}]$-module
\[
\widehat{H}^1(\Sh_{U^\frap,\overline{F}},\Z_p)_\mm
\]
is $\sigma$-typic.
\end{prop}

\begin{proof} This follows from Theorem \ref{H1Typic} (noting that the $\sigma$'s are compatible), and the observation that all operations in the definition of
\[
\widehat{H}^1(\Sh_{U^\frap,\overline{F}},\Z_p)_\mm = \varprojlim_m \varinjlim_K H^1(\Sh_{K U^\frap,\overline{F}},\Z/p^m\Z)_\mm
\]
preserve $\sigma$-typic modules.
\end{proof}

\section{Local-global compatibility}

In this section, we prove a local-global compatibility result for the functor constructed above. This turns out to be mostly a straightforward consequence of $p$-adic uniformization, originally due to Cerednik, \cite{Cerednik}, and in moduli-theoretic terms to Drinfeld, \cite{DrinfeldUnif}, and generalized by Rapoport-Zink, \cite{RapoportZinkBook}, and Varshavsky, \cite{Varshavsky}.

We continue to consider the Shimura curves $\Sh_U$ associated to a division algebra $D$ over a totally real field $F$ as in the previous section.

\begin{definition} Let $\rho_{U^\frap}^i$ be the admissible $\Z_p$-representation of $\Gal_{F,S}\times D^\times_\frap$ given by
\[
\rho_{U^\frap}^i = H^i(U^\frap,\Q_p/\Z_p) = \varinjlim_K H^i(\Sh_{K U^\frap,\overline{F}},\Q_p/\Z_p)\ .
\]
Let $\pi_{U^\frap}$ be the admissible $\Z_p$-representation of $\GL_2(F_\frap)=G(F_\frap)$ given by the space of continuous functions
\[
\pi_{U^\frap} = C^0(G(F)\backslash G(\A_{F,f}) / U^\frap, \Q_p/\Z_p)\ .
\]
\end{definition}

We note that one would usually consider the space
\[
\pi_{U^\frap}^\comp = C^0(G(F)\backslash G(\A_{F,f}) / U^\frap, \Z_p)\ ,
\]
but this carries the same information as $\pi_{U^\frap}$: One can write $\pi_{U^\frap} = \pi_{U^\frap}^\comp\otimes_{\Z_p} \Q_p/\Z_p$ and conversely
\[
\pi_{U^\frap}^\comp = T_p \pi_{U^\frap} = \lim_p \pi_{U^\frap}[p^n]\ .
\]
A similar discussion applies to $\rho_{U^\frap}^i$ (at least if everything is interpreted in the derived sense, or if everything is concentrated in only one cohomological degree). We will be mostly interested in $\rho_{U^\frap} := \rho_{U^\frap}^1$.

\begin{thm}\label{LocalGlobal} There is a natural isomorphism of admissible $\Gal_{F_\frap}\times D^\times_\frap$-representations over $\Z_p$,
\[
H^i_\et(\mathbb{P}^1_{\C_p},\mathcal{F}_{\pi_{U^\frap}})\cong \rho_{U^\frap}^i\ .
\]
\end{thm}

In fact, the statement is true on the derived level. The key tool to proving Theorem \ref{LocalGlobal} is the $p$-adic uniformization theorem.

\begin{thm}[Cerednik] Let $U=K U^\frap\subset D^\times(\A_{F,f}) = D^\times_\frap\times G(\A_{F,f}^\frap)$. There is an isomorphism of adic spaces over $\C_p$,
\[
(\Sh_U\otimes_F \C_p)^\ad\cong G(F)\backslash [\cM_{\Dr,K,\C_p}\times G(\A_{F,f}^\frap)/U^\frap]\ ,
\]
compatible for varying $U$, and with the Weil descent datum to $F$.
\end{thm}

A proof relying on Rapoport-Zink's book has been given by Boutot-Zink, \cite{BoutotZink}. Let
\[
\Sh_{U^\frap,\C_p} = G(F)\backslash [\cM_{\Dr,\infty,\C_p}\times G(\A_{F,f}^\frap)/U^\frap]\ ,
\]
which is a perfectoid space over $\C_p$ (equipped with a Weil descent datum to $F$), such that
\[
\Sh_{U^\frap,\C_p} \sim \varprojlim_K (\Sh_{K U^\frap}\otimes_F \C_p)^\ad\ .
\]
These properties follow from the similar properties of $\cM_{\Dr,\infty,\C_p}$, cf. \cite[Theorem 6.5.4]{ScholzeWeinstein}.

In particular, we find that
\begin{equation}\label{CohomInfLevel}
H^i(\Sh_{U^\frap,\C_p},\Q_p/\Z_p) = \varinjlim_K H^i(\Sh_{K U^\frap,\C_p},\Q_p/\Z_p) = H^i(U^\frap,\Q_p/\Z_p)
\end{equation}
as $W_{F_\frap}\times D^\times_\frap$-representations; here $W_{F_\frap}\subset \Gal_{F_\frap}\subset \Gal_F$ denotes the (local) Weil group of $F_\frap$.

On the other hand, by \cite[Proposition 7.1.1]{ScholzeWeinstein}, there is the Hodge-Tate period map
\[
\pi_\HT: \cM_{\Dr,\infty,\C_p}\to \mathbb{P}^1_{\C_p}\ ,
\]
compatible with Weil descent data, where the right-hand side is the Brauer-Severi variety for $D/F$. Under the duality isomorphism
\[
\cM_{\Dr,\infty,\C_p}\cong \cM_{\LT,\infty,\C_p}\ ,
\]
this is identified with the Grothendieck-Messing period map, cf. \cite[Theorem 7.2.3]{ScholzeWeinstein}. The $\GL_2(F_\frap)$-equivariance of the Hodge-Tate period map ensures that it gives a map
\[
\pi_\HT^\Sh: \Sh_{U^\frap,\C_p} = G(F)\backslash [\cM_{\Dr,\infty,\C_p}\times G(\A_{F,f}^\frap)/U^\frap]\to \mathbb{P}^1_{\C_p}\ .
\]
Note that $\pi_\HT$ is $W_{F_\frap}\times D^\times_\frap$-equivariant.

\begin{rem} Here, we construct a global Hodge-Tate period map directly from the local Hodge-Tate period map. As the Shimura curves under consideration are not of Hodge type, one cannot formally use the construction of a global Hodge-Tate period map in \cite{ScholzeTorsion} to get one in this setup. In cases of overlap, it is to be expected that these period maps are compatible, but we do not discuss this here.
\end{rem}

\begin{prop}\label{PiHTPush} There is a $W_{F_\frap}\times D^\times_\frap$-equivariant isomorphism of sheaves on the \'etale site of (the adic space) $\mathbb{P}^1_{\C_p}$,
\[
R\pi_{\HT\et\ast}^\Sh \Q_p/\Z_p\cong \mathcal{F}_{\pi_{U^\frap}}\ .
\]
\end{prop}

\begin{proof} First, we check that the higher direct images vanish. It is enough to check this at stalks, so let $\bar{x}=\Spa(C,C^+)\to \mathbb{P}^1_{\C_p}$ be any geometric point, i.e. $C/\breve{F}$ is complete algebraically closed and $C^+\subset C$ is an open and bounded valuation subring. We may assume that $C$ is the completion of the algebraic closure of the residue field of $\mathbb{P}^1_{\C_p}$ at the image of $\bar{x}$. Let $\bar{x}\to U_i\to \mathbb{P}^1_{\C_p}$ be a cofinal system of \'etale neighborhoods of $\bar{x}$; then $\bar{x}\sim \varprojlim_i U_i$. Write
\[
U_i^\Sh\to \Sh_{U^\frap,\C_p}
\]
for the pullback of $U_i$, so that $U_i^\Sh$ is a perfectoid space \'etale over $\Sh_{U^\frap,\C_p}$. One can form the inverse limit $U_{\bar{x}}^\Sh = \varprojlim U_i^\Sh$ in the category of perfectoid spaces (over $\C_p$). Now
\[
(R^j \pi_{\HT\et\ast}^\Sh \Q_p/\Z_p)_{\bar{x}} = \varinjlim_i H^j_\et(U_i^\Sh,\Q_p/\Z_p) = H^j_\et(U_{\bar{x}}^\Sh,\Q_p/\Z_p)\ .
\]
On the other hand, the fibre $U_{\bar{x}}^\Sh$ is given by profinitely many copies of $\bar{x}$,
\[
U_{\bar{x}}^\Sh = \Spa(C^0(G(F)\backslash \GL_2(F_\frap)\times G(\A_{F,f}^\frap) / U^\frap,C),C^0(G(F)\backslash \GL_2(F_\frap)\times G(\A_{F,f}^\frap) / U^\frap,C^+))\ .
\]
This implies that
\[
H^j_\et(U_{\bar{x}}^\Sh,\Q_p/\Z_p)
\]
vanishes for $j>0$, and equals $C^0(G(F)\backslash \GL_2(F_\frap)\times G(\A_{F,f}^\frap) / U^\frap,\Q_p/\Z_p)$ in degree $0$, e.g. by writing $U_{\bar{x}}^\Sh$ as an inverse limit of finitely many copies of $\bar{x}$ and using \cite[Corollary 7.18]{ScholzePerfectoid}.

It remains to identify $\pi_{\HT\et\ast}^\Sh \Q_p/\Z_p$. The previous computation already showed that the fibres are isomorphic to $\pi_{U^\frap}$. Let $U\to \mathbb{P}^1_{\C_p}$ be any \'etale map. We have to construct a map
\[\begin{aligned}
&H^0(U\times_{\mathbb{P}^1_{\C_p}} G(F)\backslash [\cM_{\Dr,\infty,\C_p}\times G(\A_{F,f}^\frap)/U^\frap],\Q_p/\Z_p)\\
\to &\Map_{\cont,\GL_2(F_\frap)}(|U\times_{\mathbb{P}^1_{\C_p}} \cM_{\Dr,\infty,\C_p}|,C^0(G(F)\backslash \GL_2(F_\frap)\times G(\A_{F,f}^\frap) / U^\frap,\Q_p/\Z_p))\ .
\end{aligned}\]
But the left hand side is the same as
\[
C^0(|U\times_{\mathbb{P}^1_{\C_p}} G(F)\backslash [\cM_{\Dr,\infty,\C_p}\times G(\A_{F,f}^\frap)/U^\frap]|,\Q_p/\Z_p)
\]
and it remains to observe that there is a natural $\GL_2(F_\frap)$-equivariant map
\[\begin{aligned}
(U&\times_{\mathbb{P}^1_{\C_p}} \cM_{\Dr,\infty,\C_p})\times (G(F)\backslash \GL_2(F_\frap)\times G(\A_{F,f}^\frap) / U^\frap)\\
&\to U\times_{\mathbb{P}^1_{\C_p}} G(F)\backslash [\cM_{\Dr,\infty,\C_p}\times G(\A_{F,f}^\frap)/U^\frap]\ .
\end{aligned}\]
\end{proof}

From Proposition \ref{PiHTPush}, we see that
\[
H^i(\Sh_{U^\frap,\C_p},\Q_p/\Z_p) = H^i_\et(\mathbb{P}^1_{\C_p},\mathcal{F}_{\pi_{U^\frap}})\ .
\]
Together with \eqref{CohomInfLevel}, this gives Theorem \ref{LocalGlobal} (noting that $W_{F_\frap}$-equivariance implies $\Gal_{F_\frap}$-equivariance by continuity).

\section{Consequences}\label{Consequences}

In this section, we continue the setup of Section \ref{ShimuraCurves}. Again, we fix an absolutely irreducible (odd) $2$-dimensional representation $\overline{\sigma}$ of $\Gal_F$ over a finite extension $\F_q$ of $\F_p$. We assume that the associated maximal ideal $\mm$ of the abstract Hecke algebra $\T$ satisfies
\[
\pi_{U^\frap,\mm}\neq 0
\]
for some $U^\frap$. We fix a finite set $S$ of finite places containing all places above $p$, such that $\overline{\sigma}$ is unramified outside $S$, and $U^\frap = U_S^\frap\times U^S$, where $U^S=\prod_{v\not\in S}\GL_2(\cO_{F_v})\subset \GL_2(\A_{F,f}^S)\cong G(\A_{F,f}^S)$. We want to see that if $\pi_{U^\frap,\mm}\neq 0$, then also $\rho_{U^\frap,\mm}\neq 0$. Note that an automorphic representation of $G$ transfers to $D^\times$ if and only if it is discrete series at $\frap$, by the Jacquet-Langlands correspondence. We will construct some cuspidal types which will allow us to construct congruences to representations which are discrete series (even cuspidal) at $\frap$, and thus transfer all torsion classes from $G$ to $D^\times$.

\begin{prop}\label{BasicTypes} Let $m\geq 0$ be an integer. Consider the compact open subgroup
\[
U_m=\left\{\left(\begin{array}{cc}1+\varpi^{m+1}\cO_{F_\frap} & \varpi^m \cO_{F_\frap} \\ \varpi^{m+1} \cO_{F_\frap} & 1+\varpi^{m+1} \cO_{F_\frap}\end{array}\right)\right\}\subset \GL_2(F_\frap)\ .
\]
There is a homomorphism
\[
\alpha_m: U_m\to \cO_{F_\frap} / \varpi^m \cO_{F_\frap}: \left(\begin{array}{cc}1+\varpi^{m+1}a & \varpi^m b \\ \varpi^{m+1} c & 1+\varpi^{m+1} d\end{array}\right)\mapsto b+c\ .
\]
For each nontrivial character $\psi: \cO_{F_\frap} / \varpi^m \cO_{F_\frap}\to \C^\times$, if $\pi$ is an irreducible smooth representation of $\GL_2(F_\frap)$ such that $\pi|_{U_m}$ contains the character $\psi\circ \alpha_m$, then $\pi$ is cuspidal.
\end{prop}

\begin{proof} Assume that $\psi$ is trivial on $\varpi^k \cO_{F_\frap}$ with $k$ minimal, so $0<k\leq m$. Let $\psi^\prime: \cO_{F_\frap} / \varpi \cO_{F_\frap}\to \C^\times$ be the restriction of $\psi$ to $\varpi^{k-1} \cO_{F_\frap} / \varpi^k \cO_{F_\frap}$. Then $\pi|_{U_{m+k-1}}$ contains the character $\psi^\prime\circ \alpha_{m+k-1}$. But this corresponds to a ramified simple stratum in the sense of \cite[Definition 13.1]{BushnellHenniart}, and so any representation $\pi$ containing the character $\psi^\prime\circ \alpha_{m+k-1}$ of $U_{m+k-1}$ is cuspidal. Namely, $\ell(\pi)>0$ by \cite[12.9 Theorem]{BushnellHenniart} and $\pi$ cannot contain an essentially scalar stratum by \cite[13.2 Proposition (1), 11.1 Proposition 1]{BushnellHenniart}, thus is cuspidal by \cite[14.5 Theorem, 13.3 Theorem]{BushnellHenniart}.
\end{proof}

Let $e$ be the ramification index of $[F_\frap:\Q_p]$, and fix a surjection $\beta_m: \cO_{F_\frap} / \varpi^{me}\to \Z/p^m\Z$. In particular, we get the following corollary.

\begin{cor}\label{ConstructionTypes} Let $A_m=\Z_p[T]/\left((T^{p^m}-1)/(T-1)\right)$. Define a character $\psi_m: U_{me}\to A_m^\times$ by composing
\[
\beta_m\circ \alpha_{me}: U_{me}\to \cO_{F_\frap} / \varpi^{me}\to \Z/p^m\Z
\]
with the map sending $1\in \Z/p^m\Z$ to $T\in A_m^\times$. Any automorphic representation of $G$ appearing in
\[
C^0(G(F)\backslash G(\A_{F,f}) / U_{me}\times U^\frap,\psi_m)[1/p]
\]
is cuspidal at $\frap$.$\hfill \Box$
\end{cor}

\begin{cor} Let $\T(U^\frap)_\mm$ be defined as in Section \ref{ShimuraCurves}, so that it acts faithfully on $H^1(U^\frap,\Q_p/\Z_p)_\mm$. The natural action of $\T$ on
\[
\pi_{U^\frap,\mm} = C^0(G(F)\backslash G(\A_{F,f}) / U^\frap,\Q_p/\Z_p)_\mm
\]
extends to a continuous action of $\T(U^\frap)_\mm$.
\end{cor}

\begin{rem} One can deduce from this Corollary the existence of Galois representations for Hilbert modular forms which are nowhere discrete series, assuming only its existence for forms which are discrete series at $\frap$. Thus, this provides an alternative argument for Taylor's construction of these Galois representations, \cite{TaylorHilbertGalRepr}, and it seems reasonable to expect that one could do a similar argument in the compact unitary case, providing an alternative to the construction of Galois representations of Shin, \cite{ShinGalRepr}, and Chenevier-Harris, \cite{ChenevierHarris}, by reducing directly to the representations constructed by Harris-Taylor, \cite{HarrisTaylor}.
\end{rem}

\begin{proof} It is enough to check this for each group
\[
C^0(G(F)\backslash[G(\A_{F,f}) / K^\prime U^\frap],\Z/p^m\Z)_\mm\ .
\]
We may assume that $K^\prime = U_{me}$ is of the form in Corollary \ref{ConstructionTypes}, as these groups are cofinal. (As we use only one $m$, we may have to increase simulateneously $m$ in the coefficients $\Z/p^m\Z$ for this.) In that case, $\Z/p^m\Z\cong A_m/(T-1)$, and $\psi_m\mod (T-1)$ is trivial. Thus, there is a $\T$-equivariant surjection
\[
C^0(G(F)\backslash[G(\A_{F,f}) / U_{me} U^\frap],\psi_m)_\mm\to C^0(G(F)\backslash[G(\A_{F,f}) / U_{me} U^\frap],\Z/p^m\Z)_\mm\ .
\]
We see that it suffices to show that the action of $\T$ on
\[
M=C^0(G(F)\backslash[G(\A_{F,f}) / U_{me} U^\frap],\psi_m)_\mm
\]
extends to a continuous action of $\T(U^\frap)_\mm$. But $M$ is $p$-torsion free, so it suffices to check in characteristic $0$. There, the result follows by observing that by Corollary \ref{ConstructionTypes}, all automorphic representations of $G$ appearing in $M[1/p]$ are cuspidal at $\frap$, and thus transfer to $D^\times$, where they show up in the cohomology group
\[
H^1(\Sh_{K_m U^\frap,\C},\Z_p)_\mm
\]
for $K_m$ sufficiently small. As $\T(U^\frap)_\mm\twoheadrightarrow \T(K_m U^\frap)_\mm$ acts by definition continuously on $H^1(\Sh_{K_m U^\frap,\C},\Z_p)_\mm$, the result follows.
\end{proof}

Recall that there is a $2$-dimensional Galois representation
\[
\sigma: \Gal_{F,S}\to \GL_2(\T(U^\frap)_\mm)\ ,
\]
and that by Proposition \ref{CompH1Typic}, $\rho_{U^\frap,\mm}$ is $\sigma$-typic, so
\[
\rho_{U^\frap,\mm} = \sigma\otimes_{\T(U^\frap)_\mm} \rho_{U^\frap}[\sigma]
\]
for some $\T(U^\frap)_\mm[D^\times_\frap]$-module $\rho_{U^\frap}[\sigma]$.

Summing up, we have the following result.

\begin{cor} There is a canonical $\T(U^\frap)_\mm[\Gal_{F_\frap}\times D^\times_\frap]$-equivariant isomorphism
\[
H^1_\et(\mathbb{P}^1_{\C_p},\mathcal{F}_{\pi_{U^\frap,\mm}}) \cong \sigma|_{\Gal_{F_\frap}}\otimes_{\T(U^\frap)_\mm} \rho_{U^\frap}[\sigma]\ .
\]
The $\T(U^\frap)_\mm$-module $\rho_{U^\frap}[\sigma]$ is faithful.$\hfill \Box$
\end{cor}

In particular this implies that the localization $\pi_{U^\frap,\mm}$ determines the representation
\[
\sigma|_{\Gal_{F_\frap}}: \Gal_{F_\frap}\to \GL_2(\T(U^\frap)_\mm)\ ,
\]
at least if $\overline{\sigma}|_{\Gal_{F_\frap}}$ is absolutely irreducible.

\begin{thm} Assume that $\overline{\sigma}|_{\Gal_{F_\frap}}$ is absolutely irreducible.\footnote{By Theorem \ref{MainThmModP} below, this can be determined in terms of $\pi_{U^\frap}[\mm]$.} Then
\[
\sigma|_{\Gal_{F_\frap}}: \Gal_{F_\frap}\to \GL_2(\T(U^\frap)_\mm)
\]
is determined by $\pi_{U^\frap,\mm}$. More precisely, the $\T(U^\frap)_\mm[\Gal_{F_\frap}]$-module
\[
H^1_\et(\mathbb{P}^1_{\C_p},\mathcal{F}_{\pi_{U^\frap,\mm}})
\]
is $\sigma|_{\Gal_{F_\frap}}$-typic, and faithful as $\T(U^\frap)_\mm$-module; this determines $\sigma|_{\Gal_{F_\frap}}$ by Lemma \ref{TypicDeterminesRepr} above.$\hfill \Box$
\end{thm}

We want to pass from information about the localization $\pi_{U^\frap}$ at $\mm$ to the $\mm$-torsion $\pi_{U^\frap}[\mm]$. For this, observe the following.

\begin{prop} For any ideal $I\subset \T(U^\frap)_\mm$, the natural map
\[
H^1_\et(\mathbb{P}^1_{\C_p},\mathcal{F}_{\pi_{U^\frap,\mm}[I]})\to H^1_\et(\mathbb{P}^1_{\C_p},\mathcal{F}_{\pi_{U^\frap,\mm}})[I]
\]
is injective, and the action of $(\cO_D^\times)_1$ on the cokernel is trivial, where $(\cO_D^\times)_1\subset \cO_D^\times$ denotes the subgroup of elements of reduced norm $1$.
\end{prop}

\begin{proof} The group
\[
H^0(\mathbb{P}^1_{\C_p},\mathcal{F}_{\pi_{U^\frap,\mm}}) = H^0(\Sh_{U^\frap,\C_p},\Q_p/\Z_p)_\mm
\]
is trivial, as $\overline{\sigma}$ is absolutely irreducible (as global $\Gal_F$-representation). Now note that if $I=(f_1,\ldots,f_m)$ is a sequence of generators, then they give an embedding
\[
\pi_{U^\frap,\mm} / \pi_{U^\frap,\mm}[I]\hookrightarrow \prod_{i=1}^m \pi_{U^\frap,\mm}\ ;
\]
let $\overline{\pi}$ be its cokernel. The displayed injection implies that
\[
H^0(\mathbb{P}^1_{\C_p},\mathcal{F}_{\pi_{U^\frap,\mm}/\pi_{U^\frap,\mm}[I]}) = 0\ ,
\]
from which one gets injectivity of the map in the proposition.

To see that $(\cO_D^\times)_1$ acts trivially on the cokernel, note that the cokernel injects into the kernel of
\[
H^1_\et(\mathbb{P}^1_{\C_p},\mathcal{F}_{\pi_{U^\frap,\mm}/\pi_{U^\frap,\mm}[I]})\to H^1_\et(\mathbb{P}^1_{\C_p},\mathcal{F}_{\prod_{i=1}^m \pi_{U^\frap,\mm}})\ .
\]
But this kernel admits a surjection from $H^0(\mathbb{P}^1_{\C_p},\mathcal{F}_{\overline{\pi}})$. On such groups, $(\cO_D^\times)_1$ acts trivially by Proposition \ref{ComputeH0}.
\end{proof}

\begin{thm}\label{MainThmModP} The $2$-dimensional $\Gal_{F_\frap}$-representation $\overline{\sigma}|_{\Gal_{F_\frap}}$ is determined (up to isomorphism) by the admissible $\GL_2(F_\frap)$-representation
\[
\pi_{U^\frap}[\mm] = C^0( G(F)\backslash G(\A_{F,f}) / U^\frap , \F_q)[\mm]\ .
\]
More precisely, $\overline{\sigma}|_{\Gal_{F_\frap}}$ can be read off from the $\Gal_{F_\frap}$-representation
\[
H^1_\et(\mathbb{P}^1_{\C_p},\mathcal{F}_{\pi_{U^\frap}[\mm]})\ ,
\]
which is an infinite-dimensional admissible $\Gal_{F_\frap}\times D^\times_\frap$-representation. Any indecomposable $\Gal_{F_\frap}$-subrepresentation of $H^1_\et(\mathbb{P}^1_{\C_p},\mathcal{F}_{\pi_{U^\frap}[\mm]})$ is of dimension $\leq 2$, and
$\overline{\sigma}|_{\Gal_{F_\frap}}$ is determined in the following way.
\begin{altenumerate}
\item[{\rm Case (i)}] If there is a $2$-dimensional indecomposable $\Gal_{F_\frap}$-representation
\[
\sigma^\prime\subset H^1_\et(\mathbb{P}^1_{\C_p},\mathcal{F}_{\pi_{U^\frap}[\mm]})\ ,
\]
then $\overline{\sigma}|_{\Gal_{F_\frap}} = \sigma^\prime$.
\item[{\rm Case (ii)}] Otherwise, $H^1_\et(\mathbb{P}^1_{\C_p},\mathcal{F}_{\pi_{U^\frap}[\mm]})$ is a direct sum of characters of $\Gal_{F_\frap}$, and at most two different characters $\chi_1,\chi_2$ of $\Gal_{F_\frap}$ appear; if only one appears, let $\chi_2=\chi_1$ be the only character appearing. Then $\overline{\sigma}|_{\Gal_{F_\frap}} = \chi_1\oplus \chi_2$.
\end{altenumerate}
\end{thm}

\begin{proof} Recall that
\[
H^1_\et(\mathbb{P}^1_{\C_p},\mathcal{F}_{\pi_{U^\frap}[\mm]})\subset \overline{\sigma}|_{\Gal_{F_\frap}}\otimes \rho_{U^\frap}[\mm]\ ,
\]
with $(\cO_D^\times)_1$ acting trivially on the cokernel. Thus, the cokernel is an admissible representation of $D^\times / (\cO_D^\times)_1 = F^\times$; an argument identifying the central character of $\rho_{U^\frap}[\mm]$ in terms of the determinant of $\overline{\sigma}$ then shows that the cokernel is finite-dimensional. Thus, to prove that $H^1_\et(\mathbb{P}^1_{\C_p},\mathcal{F}_{\pi_{U^\frap}[\mm]})$ is infinite-dimensional, it is enough to prove that $\rho_{U^\frap}[\mm]$ is infinite-dimensional.

Assume it was finite-dimensional. Pick a minimal prime ideal $\tilde{\mm}\subset \T(K U^\frap)_\mm$, corresponding to some cuspidal automorphic representation $\pi$ contributing to $H^1(\Sh_{K U^\frap},\Z_p)_\mm$. Let $L/\Q_p$ be the finite extension which is the residue field of $\T(K U^\frap)_\mm$ at $\tilde{\mm}$, and let $\varpi_L\in \cO_L\subset L$ be a uniformizer and its ring of integers. Let $\tilde{\mm}_L\subset \cO_L\otimes_{\Z_q} \T(K U^\frap)_\mm$ be the kernel of the induced multiplication map to $\cO_L$. Then $H^1(\Sh_{K U^\frap,\C},L/\cO_L)[\tilde{\mm}_L]$ is a divisible torsion $\cO_L$-module whose $\varpi_L$-torsion is contained in $\rho_{U^\frap}[\mm]\otimes_{\F_q} \cO_L/\varpi_L$. It follows that the $\pi$-part of $H^1(\Sh_{K U^\frap,\C},L)$ is of bounded dimension over $L$, independently of $K$ and $\pi$. This implies that the $D^\times_\frap$-representation appearing in $\pi$ is of bounded dimension. On the other hand, by using suitable cuspidal types one can make this $D^\times_\frap$-representation arbitrarily ramified, which makes its dimension arbitrarily big, e.g. by \cite[54.4 Proposition]{BushnellHenniart}.

Any indecomposable $\Gal_{F_\frap}$-subrepresentation of
\[
H^1(\mathbb{P}^1_{\C_p},\mathcal{F}_{\pi_{U^\frap}[\mm]})\subset \overline{\sigma}|_{\Gal_{F_\frap}}\otimes \rho_{U^\frap}[\mm]
\]
is isomorphic to a subrepresentation of $\overline{\sigma}|_{\Gal_{F_\frap}}$, and in particular of dimension $\leq 2$. If $\overline{\sigma}|_{\Gal_{F_\frap}}$ is indecomposable, then it occurs as a subrepresentation of $H^1_\et(\mathbb{P}^1_{\C_p},\mathcal{F}_{\pi_{U^\frap}[\mm]})$, as the cokernel of the displayed inclusion is finite-dimensional. This deals with Case (i).

If $\overline{\sigma}|_{\Gal_{F_\frap}} =\chi_1\oplus \chi_2$ is decomposable, then the displayed inclusion shows that
\[
H^1_\et(\mathbb{P}^1_{\C_p},\mathcal{F}_{\pi_{U^\frap}[\mm]})
\]
is a direct sum of characters $\chi_1$ and $\chi_2$. Moreover, both of them appear (if they are distinct), by finiteness of the cokernel. This deals with Case (ii).
\end{proof}

\section{Patching: The key geometric input}\label{PatchingPrep}

In this section, we will prove a refinement of Theorem \ref{Finiteness} that will allow us to prove compatibility with patching. The argument is closely related to the notion of ultraproducts, but we will take a very algebraic approach.

Fix an infinite set $\{\pi_i\}_{i\in I}$ of admissible smooth $\F_p$-representations of $\GL_n(F)$. Assume that for all compact open $H\subset \GL_n(\cO)$, the dimension of $\pi_i^H$ (for varying $i$) is bounded.

Let $\Pi$ be the subset of smooth vectors in $\prod_{i\in I} \pi_i$, i.e.
\[
\Pi = \bigcup_H \prod_{i\in I} \pi_i^H\ .
\]
This is a representation of $\GL_n(F)$ on an $R=\prod_{i\in I} \F_p$-module. Before going on, it is helpful to recall some properties of $R$.

\begin{lem} The inclusion
\[
I\hookrightarrow \Spec R\ ,
\]
sending $i\in I$ to the the kernel of the projection $R\to \F_p$ to the $i$-th coordinate, identifies $\Spec R$ with the Stone-Cech compactification of $I$. For each $x\in \Spec R$, the local ring $R_x$ is $\F_p$. There is an identification
\[
R = C^0(\Spec R,\F_p)
\]
of $R$ with continuous maps $\Spec R\to \F_p$. The ring $R$ is coherent.
\end{lem}

\begin{proof} The identification of $\Spec R$ with the Stone-Cech compactification of $I$ is standard.\footnote{Contrary to some other statements of the lemma, it does not use that $\F_p$ is finite.} For each maximal ideal $\mathfrak{m}\subset R$, the corresponding ultrafilter $\mathfrak{F}_{\mathfrak{m}}$ on $I$ is given by those subsets $I^\prime\subset I$ such that the idempotent element $e_{I^\prime}$ which is $1$ at $i\in I^\prime$ and $0$ otherwise, is not in $\mathfrak{m}$.

First, we check that all local rings are isomorphic to $\F_p$. Take a prime ideal $\mathfrak{p}\subset R$. For any $x\in R$, the equation
\[
\prod_{a=0}^{p-1} (x-a)=0
\]
holds true (by checking in each factor). Modulo $\mathfrak{p}$, it follows that $x=a$ for some $a\in \F_p$, as desired. It follows that all points of $\Spec R$ are closed, and thus that $\Spec R$ is profinite. It follows that the structure sheaf on $\Spec R$ is the constant sheaf $\F_p$, which implies the equality $R=C^0(\Spec R, \F_p)$. As $\Spec R$ is profinite, this can be written as a filtered colimit of finite products of $\F_p$; this is a filtered colimit of noetherian algebras along flat transition maps, showing that $R$ is coherent.
\end{proof}

Fix a point $x\in \Spec R\setminus I$; this corresponds to a nonprincipal ultrafilter $\mathfrak{F}$ on $I$ under the identification with the Stone-Cech compactification. It follows that
\[
\pi^\patch := \Pi\otimes_R R_x
\]
is an admissible smooth $R_x = \F_p$-representation of $\GL_n(F)$, which we will call the patched representation. Here, the word ``patched" is used in the sense of Taylor-Wiles patching, where one builds a new object $X^\patch$ from an infinite set $\{X_i\}_{i\in I}$ of objects such that each ``finite piece" of $X$ looks like a corresponding ``finite piece" of $X_i$ for infinitely many $i$. In our setup, we have for instance the following simple observation.

\begin{lem} For each compact open normal subgroup $H\subset \GL_n(\cO)$, there are infinitely many $i\in I$ (more precisely, for all $i\in I^\prime$ with $I^\prime\in \mathfrak{F}$) such that
\[
(\pi^\patch)^H\cong \pi_i^H
\]
as $\GL_n(\cO)/H$-representations.
\end{lem}

\begin{proof} There are only finitely many isomorphism classes of $\GL_n(\cO)/H$-representations of bounded dimension; recall that the dimension of $\pi_i^H$ was assumed to be bounded. As $\mathfrak{F}$ is an ultrafilter, it follows that for $i\in I^\prime$ with $I^\prime\in \mathfrak{F}$, all $\pi_i^H\cong \pi_0$ are isomorphic. But then
\[
\Pi^H\otimes_R \prod_{i\in I^\prime} \F_p\cong \pi_0\otimes_{\F_p} \prod_{i\in I^\prime} \F_p\ ,
\]
and thus also $(\pi^\patch)^H\cong \pi_0$.
\end{proof}

In the patching construction of \cite{CEGGPS}, one chooses some representation of $\GL_n(\cO)/H$ which occurs infinitely often, and then chooses them compatibly for all $H$. After that, one wants to extend the resulting $\GL_n(\cO)$-representation to all of $\GL_n(F)$ by allowing extra Hecke operators. This is possible only if the previous choices were made carefully; in our setup, everything works automatically. We leave it to the reader to verify that the representation constructed in \cite{CEGGPS} can be obtained as $\pi^\patch$ for a suitably chosen $x\in \Spec R\setminus I$; this amounts to going through their construction, and with every choice made one has to shrink the filter accordingly.

As before, one can attach to $\Pi$ a sheaf $\mathcal{F}_\Pi$ of $R$-modules on $(\mathbb{P}^{n-1}_{\breve{F}}/D^\times)_\et$ by sending a $D^\times$-equivariant \'etale $U\to \mathbb{P}^{n-1}_{\breve{F}}$ to the set of $D^\times\times \GL_n(F)$-equivariant continuous maps
\[
|U\times_{\mathbb{P}^{n-1}_{\breve{F}}} \cM_{\LT,\infty}|\to \Pi\ .
\]

The result of this section is the following.

\begin{thm}\label{PatchingThm} Assume that all $\pi_i$ are injective as $H$-representations for some compact open subgroup $H\subset \GL_n(F)$ (independent of $i$).\footnote{It would be enough to assume that they have perfect resolutions by injective $H$-representations which are of ``bounded complexity" in a suitable sense. As in our application, they will actually be injective, we restrict to this simpler setup.} For all $j\geq 0$ and compact open $K\subset D^\times$, the cohomology group
\[
H^j((\mathbb{P}^{n-1}_C/K)_\et,\mathcal{F}_\Pi)
\]
is a finitely presented $R$-module.
\end{thm}

To explain the meaning of this result, we need the following classification of finitely presented $R$-modules.

\begin{lem}\label{FinitelyPresentedModules} Let $V_i$, $i\in I$, be a sequence of $\F_p$-vector spaces of bounded dimension. Then
\[
M=\prod_{i\in I} V_i
\]
is a finitely presented $R$-module. Conversely, if $M$ is a finitely presented $R$-module with specialization $V_i$ over $\F_p$ at $i\in I\subset \Spec R$, then the $V_i$ are of bounded dimension, and the natural map
\[
M\to \prod_{i\in I} V_i
\]
is an isomorphism.
\end{lem}

\begin{proof} For the first part, we may find a finite decomposition $I=\bigsqcup_{d=0}^D I_d$ such that $V_i\cong \F_p^d$ for $i\in I_d$. Then $M_d = \prod_{i\in I_d} \F_p^d\cong R_d^d$ is a finite free $R_d = \prod_{i\in I_d} \F_p$-module, and
\[
M = \prod_{d=0}^D M_d
\]
is a finitely presented $R=\prod_{d=0}^D R_d$-module.

Assume now that $M$ is finitely presented. Then we may find a presentation
\[
R^N\to R^D\to M\to 0\ .
\]
The map $R^N\to R^D$ is given by a matrix $A=(A_i)_{i\in I}\in M_{N\times D}(R) = \prod_{i\in I} M_{N\times D}(\F_p)$. There are only finitely many possibilities for each $A_i$, so after a finite decomposition of $I$ (corresponding to a clopen decomposition of $\Spec R$), we may assume that $A$ is constant. In that case, $M\cong R^d$ is constant, so that the claim is clear.
\end{proof}

In particular, the following follows directly from Theorem \ref{PatchingThm}, using that taking cohomology commutes with localization on $R$.

\begin{cor}\label{PatchingCor} Assume that all $\pi_i$ are injective as $H$-representations for some compact open subgroup $H\subset \GL_n(F)$. For all $j\geq 0$ and compact open $K\subset D^\times$, one has
\[
H^j((\mathbb{P}^{n-1}_C/K)_\et,\mathcal{F}_\Pi) = \prod_{i\in I} H^j((\mathbb{P}^{n-1}_C/K)_\et,\mathcal{F}_{\pi_i})\ ,
\]
and thus
\[
H^j((\mathbb{P}^{n-1}_C/K)_\et,\mathcal{F}_{\pi^\patch}) = \left(\prod_{i\in I} H^j((\mathbb{P}^{n-1}_C/K)_\et,\mathcal{F}_{\pi_i})\right)\otimes_R R_x\ .
\]
\end{cor}

Intuitively, the last statement says that patching commutes with the functor $\pi\mapsto H^j((\mathbb{P}^{n-1}_C/K)_\et,\mathcal{F}_\pi)$.

To prove Theorem \ref{PatchingThm}, we follow the proof of Theorem \ref{Finiteness}. In doing so, we need to establish some properties of $R\otimes_{\F_p} \cO_C/p$ first.

\begin{lem} The ring $R\otimes_{\F_p} \cO_C/p$ is coherent.
\end{lem}

\begin{proof} Recall that $R=C^0(\Spec R,\F_p)$ can be written as $R=\varinjlim R_j$ where each $R_j$ is a finite product of $\F_p$'s. Then also $R\otimes_{\F_p} \cO_C/p = \varinjlim R_j\otimes_{\F_p} \cO_C/p$ can be written as a filtered colimit of coherent algebras along flat transition maps, and thus is coherent itself. Here, we use that $\cO_C/p$ is coherent, namely any finitely generated ideal $J\subset \cO_C/p$ is in fact principal, $J=\cO_C/p\cdot x$, and those are finitely presented, $J\cong \cO_C/(p/x)$.
\end{proof}

\begin{cor}\label{AlmostCoherence} The ring $R\otimes_{\F_p} \cO_C/p$ is almost coherent in the sense that the category of almost finitely presented $R\otimes_{\F_p} \cO_C/p$ is abelian, and closed under kernels, cokernels and extensions.
\end{cor}

\begin{proof} By the previous lemma, the category of finitely presented $R\otimes_{\F_p} \cO_C/p$ has these properties. The corollary follows by approximating almost finitely presented modules (and maps between them) by finitely presented modules.
\end{proof}

Now, we first prove the analogue of the local finiteness result. As in the previous section, choose some affinoid $V\subset \mathbb{P}^{n-1}_{\breve{F}}$ which lifts to $\cM_{\LT,0}$, and fix such a lift; moreover fix some strict quasicompact open subset $U\subset V$.

\begin{lem}\label{LocalFinitenessPatching} Assume that all $\pi_i$ are injective as $H$-representations for some compact open subgroup $H\subset \GL_n(F)$. For any $m\geq 0$, there is a compact open $K_0\subset D^\times$ stabilizing $V$, $U$ and the section $V\to \cM_{\LT,0,C}$ such that for all $K\subset K_0$, the image of the natural map
\[
H^j((V/K)_\et,\mathcal{F}_\Pi\otimes \cO^+/p)\to H^j((U/K)_\et,\mathcal{F}_\Pi\otimes \cO^+/p)
\]
is an almost finitely presented $R\otimes_{\F_p} \cO_C/p$-module for all $j=0,\ldots,m$.
\end{lem}

\begin{proof} As before, this statement depends only on $\Pi$ as a $\GL_n(\cO)$-representation. We assumed that all $\pi_i$ are injective as $H$-representations for some open subgroup $H\subset \GL_n(\cO)$; fix such an $H$ which is pro-$p$ and normal in $\GL_n(\cO)$. As $\F_p[[H]]$ is local, it follows that $\pi_i$ is isomorphic to $\dim \pi_i^H$ many copies of the regular representation $\pi^\reg_H$ of $H$. As $\dim \pi_i^H$ is bounded, it follows that $\Pi|_H$ can be written as a direct summand of $(\pi^\reg_H)^n\otimes_{\F_p} R$ for some $n$.

Let $V_H, U_H\subset \cM_{\LT,H}$ be the preimages of $V,U\subset \cM_{\LT,0}$. There is a Hochschild-Serre spectral sequence
\[
H^{j_1}(\GL_n(\cO)/H,H^{j_2}((V_H/K)_\et,\mathcal{F}_\Pi\otimes \cO^+/p))\Rightarrow H^{j_1+j_2}((V/K)_\et,\mathcal{F}_\Pi\otimes \cO^+/p)\ ,
\]
and similarly for $U$. Filtering the inclusion $U\subset V$ by sufficiently many strict rational subsets and using the obvious analogue of \cite[Lemma 5.4]{ScholzePAdicHodge} which holds for almost finitely presented $R\otimes_{\F_p}\cO_C/p$ by using Corollary \ref{AlmostCoherence}, this reduces us to proving that the image of
\[
H^j((V_H/K)_\et,\mathcal{F}_\Pi\otimes \cO^+/p)\to H^j((U_H/K)_\et,\mathcal{F}_\Pi\otimes \cO^+/p)
\]
is an almost finitely presented $R\otimes_{\F_p} \cO/p$-module. As $\Pi|_H$ is a direct summand of $(\pi^\reg_H)^n\otimes_{\F_p} R$, this image is a direct summand of $n$ copies of the base extension $\F_p\to R$ of the image of
\[
H^j((V_H/K)_\et,\mathcal{F}_{\pi^\reg_H}\otimes \cO^+/p)\to H^j((U_H/K)_\et,\mathcal{F}_{\pi^\reg_H}\otimes \cO^+/p)\ .
\]
But the latter is almost finitely generated over $\cO_C/p$ by Lemma \ref{LocalFiniteness}, thus almost finitely presented as $\cO_C/p$ is almost noetherian, cf. \cite[Proposition 2.6]{ScholzePAdicHodge}.
\end{proof}

\begin{cor}\label{FirstGlobalFinitenessPatch} Assume that all $\pi_i$ are injective as $H$-representations for some compact open subgroup $H\subset \GL_n(F)$. For any $j\geq 0$ and any compact open $K\subset D^\times$, the $R\otimes_{\F_p} \cO_C/p$-module
\[
H^j((\mathbb{P}^{n-1}_C/K)_\et,\mathcal{F}_\Pi\otimes \cO^+/p)
\]
is almost finitely presented.
\end{cor}

\begin{proof} By a Hochschild-Serre spectral sequence (cf. proof of Corollary \ref{SecondGlobalFiniteness}), we may assume that $K$ is sufficiently small (depending on $j$). In that case, the same argument as for Corollary \ref{FirstGlobalFiniteness} applies, noting as before that the analogue of \cite[Lemma 5.4]{ScholzePAdicHodge} holds for almost finitely presented $R\otimes_{\F_p}\cO_C/p$-modules.
\end{proof}

Now we can finish the proof of Theorem \ref{PatchingThm}. Note that one has an almost isomorphism
\[
H^j((\mathbb{P}^{n-1}_C/K)_\et,\mathcal{F}_\Pi)\otimes \cO_C/p\to H^j((\mathbb{P}^{n-1}_C/K)_\et,\mathcal{F}_\Pi\otimes \cO^+/p)\ .
\]
Indeed, after each localization $R\to R_y= \F_p$ for $y\in \Spec R$, this follows by applying Theorem \ref{Finiteness} to the admissible smooth $R_y=\F_p$-representation $\Pi\otimes_R R_y$. Globally, the result follows from the following simple observation applied to the kernel and cokernel of the displayed map.

\begin{lem} Let $M$ be an $R\otimes_{\F_p} \cO_C/p$-module. Assume that for all $y\in \Spec R$, $M\otimes_R R_y$ is almost zero. Then $M$ is almost zero.
\end{lem}

\begin{proof} Take any $m\in M$ and nilpotent $\epsilon\in \cO_C/p$. Then $\epsilon m=0\in M\otimes_R R_y$ for all $y\in \Spec R$, as $M\otimes_R R_y$ is almost zero. But then $\epsilon m=0\in M$, showing that $m$ is almost zero.
\end{proof}

Finally, Theorem \ref{PatchingThm} follows from Corollary \ref{FirstGlobalFinitenessPatch} and the following lemma.

\begin{lem} Let $M$ be an $R$-module such that $M\otimes_{\F_p} \cO_C/p$ is an almost finitely presented $R\otimes_{\F_p} \cO_C/p$-module. Then $M$ is finitely presented.
\end{lem}

\begin{proof} Any almost finitely presented $R\otimes_{\F_p} \cO_C/p$-module $N$ has elementary divisors in the sense of \cite[Proposition 2.11]{ScholzePAdicHodge} at each $y\in \Spec R$. By approximating $N$ with finitely presented modules, one checks that these assemble into a continuous map
\[
\gamma_N: \Spec R\to \ell^\infty_\geq(\N)_0
\]
using notation employed there. Applying this to $N=M\otimes_{\F_p} \cO_C/p$ shows that the function sending $y\in \Spec R$ to the dimension of $M_y$ is locally constant. Thus, after passing to a clopen decomposition of $\Spec R$, we can assume that for all $y\in \Spec R$, $M_y$ is of dimension $d$. We claim that in this case, $M$ is locally free of rank $d$. Pick any $y\in \Spec R$, and choose a map $R^d\to M$ that becomes an isomorphism after localization at $y$, and let $M^\prime$ be the cokernel. This induces a map
\[
(R\otimes_{\F_p} \cO_C/p)^d\to M\otimes_{\F_p} \cO_C/p
\]
that becomes an isomorphism after localization at $y$. Its cokernel $M^\prime\otimes_{\F_p} \cO_C/p$ is then an almost finitely presented $R\otimes_{\F_p} \cO_C/p$-module whose localization at $y$ is almost zero. Repeating for $M^\prime$ what we know about $M$ then shows that after replacing $\Spec R$ by an open neighborhood of $y$, we have $M^\prime=0$. This means that $R^d\to M$ is surjective. Checking at all local rings implies that $R^d\to M$ is an isomorphism, as all localizations of $M$ are of rank $d$. Thus, $M\cong R^d$ is free, as desired.
\end{proof}

\begin{rem} The results of this section imply similar results over a finite base ring $A$ over $\Z/p^n\Z$ if all $\pi_i^H$ are free $A$-modules.
\end{rem}

\section{Patching}

In this section, we do the analogue of the patching construction from \cite{CEGGPS}, in the simplest possible situation. Our setup here is more restrictive than it should be (in particular, it forces $[F_\frap:\Q_p]$ to be even), but we hope that the simplicity of the discussion gives some justification.

We assume that $\frap$ is the only place above $p$ in $F$, and that $G$ is split at all finite places.\footnote{The latter hypothesis is only imposed to be able to use the references below without further justification.} Let
\[
\overline{\sigma}: \Gal_F\to \GL_2(\F_q)
\]
be absolutely irreducible, and unramified outside $\frap$. Let $U^\frap = \prod_{v\neq \frap} \GL_2(\cO_{F_v})\subset \GL_2(\A_{F,f}^\frap)\cong G(\A_{F,f}^\frap)$. Let $\mm\subset \T$ be the maximal ideal corresponding to $\overline{\sigma}$. Moreover, we fix a character $\psi: \Gal_{F,\frap}\to \cO_L^\times$ unramified outside $\frap$, for some finite extension $L$ of $\Q_p$ with residue field $\F_q$ and uniformizer $\varpi_L$, such that $\det\overline{\sigma} = \psi \chi_\cycl$ mod $\varpi_L$. We assume that
\[
\pi_{U^\frap,\mm} = C^0(G(F)\backslash[G(\A_{F,f})/U^\frap],\Q_p/\Z_p)_\mm\neq 0\ .
\]
Comparing central characters and determinants of associated Galois representations, we see that this implies that also
\[
C^0(G(F)\backslash[G(\A_{F,f})/U^\frap],L/\cO_L)[\psi]_\mm\neq 0\ .
\]

There are framed and unframed Galois deformation rings $R_{\overline{\sigma}}^{\square,\psi}$ and $R_{\overline{\sigma}}^\psi$, parametrizing (framed) deformations of $\overline{\sigma}$ unramified outside $\frap$ and with determinant $\psi\chi_\cycl$, and a local framed Galois deformation ring $R_\frap^{\square,\psi}$ parametrizing framed deformations of $\overline{\sigma}|_{\Gal_F}$ with determinant $\psi\chi_\cycl$. There is a natural map $R_\frap^{\square,\psi}\to R^{\square,\psi}_{\overline{\sigma}}$.

To apply the Taylor-Wiles patching technique \cite{TaylorWiles}, adapted to the case of totally real fields in \cite{TaylorMeromorphic}, we impose some usual hypothesis, cf. \cite[\S 2.2]{KisinFontaineMazur}.

\begin{hypothesis} In this section, assume that the following conditions are satisfied.
\begin{enumerate}
\item[{\rm (i)}] The prime $p\geq 5$.
\item[{\rm (ii)}] The representation $\overline{\sigma}|_{\Gal_{F(\zeta_p)}}$ is absolutely irreducible.
\item[{\rm (iii)}] If $p=5$ and $\overline{\sigma}$ has projective image $\PGL_2(\F_5)$, then the kernel of $\proj \overline{\sigma}$ does not fix $F(\zeta_5)$.
\end{enumerate}
\end{hypothesis}

Under these hypothesis, we have the existence of Taylor-Wiles primes.

\begin{prop}[{\cite[Proposition 2.2.4]{KisinFontaineMazur}}]\label{TWPrimes} The integer $g=\dim_{\F_q} H^1(\Gal_{F,\frap},\ad^0 \overline{\sigma}(1)) - [F:\Q]$ is nonnegative. For each positive integer $n$, there exists a finite set $Q_n$ of $g+[F:\Q]$ primes of $F$ such that $q_v\equiv 1\mod p^n$ for all $v\in Q_n$ and $\overline{\Frob_v}$ has distinct eigenvalues, and with the following property. The framed deformation ring $R_{\overline{\sigma},Q_n}^{\square,\psi}$ parametrizing framed deformations of $\overline{\sigma}$ unramified outside $\frap$ and $Q_n$ and with determinant $\psi\chi_\cycl$ is topologically generated by $g$ elements over $R_\frap^{\square,\psi}$.
\end{prop}

In the following, we fix such a set $Q_n$ for each $n\geq 1$, as well as a non-principal ultrafilter $\mathfrak{F}$ on $\{n\geq 1\}$. This choice accounts for all choices needed to make the patching construction, and in a precise sense it amounts to the choice of $g+[F:\Q]$ ``infinite primes" $v$ of $F$ such that $q_v\equiv 1\mod p^\infty$.

We continue to follow the discussion in \cite[\S 2.2.5]{KisinFontaineMazur}. For each $n\geq 1$, let $U_{Q_n}(1)\subset U_{Q_n}(0)\subset G(\A_{F,f}^\frap)\cong \GL_2(\A_{F,f}^\frap)$ be the compact open subgroups given by
\[
U_{Q_n}(1) = \prod_{v\not\in Q_n} \GL_2(\cO_{F_v})\times \prod_{v\in Q_n} U_v(1)\subset U_{Q_n}(0) = \prod_{v\not\in Q_n} \GL_2(\cO_{F_v})\times \prod_{v\in Q_n} U_v(0)\ ,
\]
where
\[\begin{aligned}
U_v(1) &= \{\left(\begin{array}{cc} a & b \\ c & d\end{array}\right)\mid c\equiv 0\mod v, a/d\mapsto 1\in \Delta_v\}\\
\subset U_v(0) &= \{\left(\begin{array}{cc} a & b \\ c & d\end{array}\right)\mid c\equiv 0\mod v\}\subset \GL_2(\cO_{F_v})\ ,
\end{aligned}\]
where $\Delta_v\cong \Z/p^n\Z$ is the unique quotient of order $p^n$ of the units $k_v^\times$ of the residue field $k_v$ at $v$. Thus, $U_{Q_n}(1)\subset U_{Q_n}(0)$ is a normal subgroup with quotient $\Delta_{Q_n} := U_{Q_n}(0)/U_{Q_n}(1)\cong (\Z/p^n\Z)^{g+[F:\Q]}$.

If necessary, we replace once $\F_q$ by $\F_{q^2}$ in the following step. Doing so, we can fix a root $\alpha_v$ of the polynomial $X^2 - T_v X + q_v S_v$ in $\F_q$ for all $v\in Q_n$. For each sufficiently small compact open subgroup $K\subset \GL_2(F)$, let
\[
S_\psi(K U_{Q_n}(i),\cO_L) = C^0(G(F)\backslash G(\A_{F,f}) / K U_{Q_n}(i),\cO_L)[\psi]
\]
for $i=0,1$ be the space of functions with central character $\psi$. On these spaces, there is an action by the Hecke algebra $\T(U_{Q_n}(i))$ generated by the usual elements $T_v$ and $S_v$ for $v\not\in Q_n$, $v\neq \frap$, as well as operators $U_v$ for $v\in Q_n$ given by the action of the $U_v(i)$-double coset of $\diag(\pi_v,1)$. Let $\mm_{Q_n}(i)\subset \T(U_{Q_n}(i))$ denote the (maximal) ideal generated by $\mm\cap \T(U_{Q_n}(i))$ and $U_v - \alpha_v$ for $v\in Q_n$. By \cite[Lemma 2.1.7]{KisinFontaineMazur}, the natural map
\[
S_\psi(K,\cO_L)_\mm\to S_\psi(K U_{Q_n}(0),\cO_L)_{\mm_{Q_n}(0)}
\]
is an isomorphism. Moreover, \cite[Lemma 2.1.4]{KisinFontaineMazur} implies that
\[
S_\psi(K U_{Q_n}(1),\cO_L)_{\mm_{Q_n}(1)}
\]
is a finite free $\cO_L[\Delta_{Q_n}]$-module with
\[
S_\psi(K U_{Q_n}(1),\cO_L)_{\mm_{Q_n}(1)}\otimes_{\cO_L[\Delta_{Q_n}]} \cO_L\cong S_\psi(K U_{Q_n}(0),\cO_L)_{\mm_{Q_n}(0)}\ .
\]

By the existence of Galois representations, there is an action of the unframed deformation ring $R_{\overline{\sigma},Q_n}^\psi$ on $S_\psi(K U_{Q_n}(1),\cO_L)_{\mm_{Q_n}(1)}$. Moreover, using local deformation rings at places $v\in Q_n$, there is a map $\cO_L[[y_1,\ldots,y_{g+[F:\Q]}]]\to R_{\overline{\sigma},Q_n}^\psi$ such that the action of $\cO_L[[y_1,\ldots,y_{g+[F:\Q]}]]$ on $S_\psi(K U_{Q_n}(1),\cO_L)_{\mm_{Q_n}(1)}$ comes from the $\Delta_{Q_n}$-action via the fixed surjection
\[
\cO_L[[y_1,\ldots,y_{g+[F:\Q]}]]\to \cO_L[(\Z/p^n\Z)^{g+[F:\Q]}]\cong \cO_L[\Delta_{Q_n}]\ .
\]
The map $R_{\overline{\sigma},Q_n}^\psi\to R_{\overline{\sigma},Q_n}^{\square,\psi}$ is formally smooth of dimension $3$, so we can fix
\[
y_{g+[F:\Q]+1},\ldots,y_{g+[F:\Q]+3}
\]
such that
\[
R_{\overline{\sigma},Q_n}^{\square,\psi}\cong R_{\overline{\sigma},Q_n}^\psi[[y_{g+[F:\Q]+1},\ldots,y_{g+[F:\Q]+3}]]\ .
\]
Finally, we fix surjections
\[
R_{\frap}^{\square,\psi}[[x_1,\ldots,x_g]]\to R_{\overline{\sigma},Q_n}^{\square,\psi}
\]
and a lifting
\[
\cO_L[[y_i]]\to R_\frap^{\square,\psi}[[x_1,\ldots,x_g]]\ ,
\]
where we abbreviate $\cO_L[[y_i]] = \cO_L[[y_1,\ldots,y_{g+[F:\Q]+3}]]$ here and in the following.

Set
\[
S_n(K) = R_{\overline{\sigma},Q_n}^{\square,\psi}\otimes_{R_{\overline{\sigma},Q_n}^\psi} S_\psi(K U_{Q_n}(1),\cO_L)_{\mm_{Q_n}(1)}\ ,
\]
which becomes a $R_\frap^{\square,\psi}[[x_1,\ldots,x_g]]$-module via the chosen surjection
\[
R_{\frap}^{\square,\psi}[[x_1,\ldots,x_g]]\to R_{\overline{\sigma},Q_n}^{\square,\psi}\ .
\]

Finally, we can do the patching. Fix an open ideal $I\subset \cO_L[[y_i]]$. Let
\[
\pi_n(I) = \varinjlim_K S_n(K)\otimes_{\cO_L[[y_i]]} \cO_L[[y_i]]/I\ .
\]
Then, for all sufficiently large $n$ so that $I$ contains the kernel of
\[
\cO_L[[y_i]]\to \cO_L[\Delta_{Q_n}][[y_{g+[F:\Q]+1},\ldots,y_{g+[F:\Q]+3}]]\ ,
\]
$\pi_n(I)$ is an admissible $\GL_2(F_\frap)$-representation over the finite ring $\cO_L[[y_i]]/I$ such that $\pi_n(I)^K$ is finite free for all sufficiently small compact open subgroups $K\subset \GL_2(F)$. Moreover,
\[
\pi_n(I)\otimes_{\cO_L[[y_i]]/I} \cO_L/\varpi_L = C^0(G(F)\backslash G(\A_{\F,f}^\frap) / \prod_{v\neq \frap} \GL_n(\cO_{F_v}),\cO_L/\varpi_L)
\]
is independent of $n$, so that in particular the ranks of $\pi_n(I)^K$ are bounded uniformly in $n$. Thus, we may take an ultraproduct as in Section \ref{PatchingPrep}:

For any $I$, we have the map
\[
\prod_{n\geq 1} \cO_L[[y_i]]/I\to \cO_L[[y_i]]/I\ ,
\]
which is the localization at the maximal ideal of the product corresponding to the fixed non-principal ultrafilter $\mathfrak{F}$. Define
\[
\pi_\infty(I) = \varinjlim_K (\prod_{n\geq 1} \pi_n(I)^K)\otimes_{\prod_{n\geq 1} \cO_L[[y_i]]/I} \cO_L[[y_i]]/I\ .
\]
Then $\pi(I)$ is an admissible $\GL_2(F_\frap)$-representation over $\cO_L[[y_i]]/I$ such that
\[
\pi_\infty(I)^K = (\prod_{n\geq 1} \pi_n(I)^K)\otimes_{\prod_{n\geq 1} \cO_L[[y_i]]/I} \cO_L[[y_i]]/I\ .
\]
is finite free. Finally, we can pass to the inverse limit
\[
\pi^\comp_\infty = \varprojlim_I \pi_\infty(I)
\]
to get what one may call a $(\varpi_L,y_1,\ldots,y_{g+[F:\Q]+3})$-adically admissible $\cO_L[[y_i]][\GL_2(F_\frap)]$-representation. Using the freeness properties of the situation, one may also pass to
\[
\pi_\infty = \pi_\infty^\comp\otimes_{\cO_L[[y_i]]} \omega\ ,
\]
where $\omega$ is the injective hull of $\cO_L/\varpi_L$ as $\cO_L[[y_i]]$-module. This is an admissible $\GL_2(F_\frap)$-representation over $\cO_L[[y_i]]$.

Note that $R_\frap^{\square,\psi}[[x_1,\ldots,x_g]]$ acts on all objects considered, in particular on $\pi_\infty^\comp$ and $\pi_\infty$. Using the exact same arguments (and the same ultrafilter $\mathfrak{F}$), one also produces a patched admissible $D^\times_\frap$-representation $\rho_\infty$ over $\cO_L[[y_i]]$ from the cohomology groups
\[
H^1(\Sh_{K U_{Q_n}(1)},\cO_L)[\psi]_{\mm_{Q_n}(1)}\ .
\]
As all these groups carry continuous $\Gal_{F_\frap}$-actions, so does $\rho_\infty$. Actually, $\rho_\infty$ is also an $R_\frap^{\square,\psi}[[x_1,\ldots,x_g]]$-module, and if $\overline{\sigma}|_{\Gal_{F_\frap}}$ is absolutely irreducible, then $\rho_\infty$ is $\sigma$-typic, where $\sigma$ denotes the universal (framed) deformation of $\overline{\sigma}|_{\Gal_{F_\frap}}$.

\begin{cor}\label{PatchingMainThm} There is a canonical $\Gal_{F_\frap}\times D^\times_\frap$-equivariant isomorphism
\[
H^1_\et(\mathbb{P}^1_{\C_p},\mathcal{F}_{\pi_\infty})\cong \rho_\infty
\]
of $R_\frap^{\square,\psi}[[x_1,\ldots,x_g]]$-modules.
\end{cor}

\begin{proof} By passing to a colimit afterwards, it is enough to prove that for an open ideal $I\subset \cO_L[[y_i]]$,
\[
H^1_\et(\mathbb{P}^1_{\C_p},\mathcal{F}_{\pi_\infty(I)})\cong \rho_\infty(I)\ .
\]
By Theorem \ref{LocalGlobal}, we know that for each big enough $n$,
\[
H^1_\et(\mathbb{P}^1_{\C_p},\mathcal{F}_{\pi_n(I)})\cong \rho_n(I)\ ,
\]
and the relevant $H^0$ vanishes. In particular, we get
\[
H^1((\mathbb{P}^1_{\C_p}/K)_\et,\mathcal{F}_{\pi_n(I)})\cong \rho_n(I)^K
\]
for any compact open subgroup $K\subset D^\times_\frap$. Let $\Pi = \bigcup_{H\subset \GL_2(F)} \prod_n \pi_n(I)^H$, the product running over sufficiently big $n$. We have the natural map
\[
H^1((\mathbb{P}^1_{\C_p}/K)_\et,\mathcal{F}_\Pi)\to \prod_n \rho_n(I)^K\ .
\]
This map is an isomorphism by Corollary \ref{PatchingCor} (or its version for finite base rings). Base extension along the fixed map
\[
\prod_n \cO_L[[y_i]]/I\to \cO_L[[y_i]]/I
\]
corresponding to $\mathfrak{F}$ shows that
\[
H^1((\mathbb{P}^1_{\C_p}/K)_\et,\mathcal{F}_{\pi_\infty(I)})\cong \rho_\infty(I)^K\ .
\]
Finally, passage to the direct limit over $K$ gives the result.
\end{proof}

\addtocontents{toc}{\vspace{0.5in}Appendix: Accessible and weakly accessible period domains, by M.~Rapoport\hfill 40\vspace{-0.68in}}

\bibliographystyle{abbrv}
\bibliography{pAdicCohomLT}

\newcommand{\sA}{\ensuremath{\mathscr{A}}\xspace}
\newcommand{\sB}{\ensuremath{\mathscr{B}}\xspace}
\newcommand{\sC}{\ensuremath{\mathscr{C}}\xspace}
\newcommand{\sD}{\ensuremath{\mathscr{D}}\xspace}
\newcommand{\sE}{\ensuremath{\mathscr{E}}\xspace}
\newcommand{\sF}{\ensuremath{\mathscr{F}}\xspace}
\newcommand{\sG}{\ensuremath{\mathscr{G}}\xspace}
\newcommand{\sH}{\ensuremath{\mathscr{H}}\xspace}
\newcommand{\sI}{\ensuremath{\mathscr{I}}\xspace}
\newcommand{\sJ}{\ensuremath{\mathscr{J}}\xspace}
\newcommand{\sK}{\ensuremath{\mathscr{K}}\xspace}
\newcommand{\sL}{\ensuremath{\mathscr{L}}\xspace}
\newcommand{\sM}{\ensuremath{\mathscr{M}}\xspace}
\newcommand{\sN}{\ensuremath{\mathscr{N}}\xspace}
\newcommand{\sO}{\ensuremath{\mathscr{O}}\xspace}
\newcommand{\sP}{\ensuremath{\mathscr{P}}\xspace}
\newcommand{\sQ}{\ensuremath{\mathscr{Q}}\xspace}
\newcommand{\sR}{\ensuremath{\mathscr{R}}\xspace}
\newcommand{\sS}{\ensuremath{\mathscr{S}}\xspace}
\newcommand{\sT}{\ensuremath{\mathscr{T}}\xspace}
\newcommand{\sU}{\ensuremath{\mathscr{U}}\xspace}
\newcommand{\sV}{\ensuremath{\mathscr{V}}\xspace}
\newcommand{\sW}{\ensuremath{\mathscr{W}}\xspace}
\newcommand{\sX}{\ensuremath{\mathscr{X}}\xspace}
\newcommand{\sY}{\ensuremath{\mathscr{Y}}\xspace}
\newcommand{\sZ}{\ensuremath{\mathscr{Z}}\xspace}

\newcommand{\fka}{\ensuremath{\mathfrak{a}}\xspace}
\newcommand{\fkb}{\ensuremath{\mathfrak{b}}\xspace}
\newcommand{\fkc}{\ensuremath{\mathfrak{c}}\xspace}
\newcommand{\fkd}{\ensuremath{\mathfrak{d}}\xspace}
\newcommand{\fke}{\ensuremath{\mathfrak{e}}\xspace}
\newcommand{\fkf}{\ensuremath{\mathfrak{f}}\xspace}
\newcommand{\fkg}{\ensuremath{\mathfrak{g}}\xspace}
\newcommand{\fkh}{\ensuremath{\mathfrak{h}}\xspace}
\newcommand{\fki}{\ensuremath{\mathfrak{i}}\xspace}
\newcommand{\fkj}{\ensuremath{\mathfrak{j}}\xspace}
\newcommand{\fkk}{\ensuremath{\mathfrak{k}}\xspace}
\newcommand{\fkl}{\ensuremath{\mathfrak{l}}\xspace}
\newcommand{\fkm}{\ensuremath{\mathfrak{m}}\xspace}
\newcommand{\fkn}{\ensuremath{\mathfrak{n}}\xspace}
\newcommand{\fko}{\ensuremath{\mathfrak{o}}\xspace}
\newcommand{\fkp}{\ensuremath{\mathfrak{p}}\xspace}
\newcommand{\fkq}{\ensuremath{\mathfrak{q}}\xspace}
\newcommand{\fkr}{\ensuremath{\mathfrak{r}}\xspace}
\newcommand{\fks}{\ensuremath{\mathfrak{s}}\xspace}
\newcommand{\fkt}{\ensuremath{\mathfrak{t}}\xspace}
\newcommand{\fku}{\ensuremath{\mathfrak{u}}\xspace}
\newcommand{\fkv}{\ensuremath{\mathfrak{v}}\xspace}
\newcommand{\fkw}{\ensuremath{\mathfrak{w}}\xspace}
\newcommand{\fkx}{\ensuremath{\mathfrak{x}}\xspace}
\newcommand{\fky}{\ensuremath{\mathfrak{y}}\xspace}
\newcommand{\fkz}{\ensuremath{\mathfrak{z}}\xspace}

\newcommand{\fkA}{\ensuremath{\mathfrak{A}}\xspace}
\newcommand{\fkB}{\ensuremath{\mathfrak{B}}\xspace}
\newcommand{\fkC}{\ensuremath{\mathfrak{C}}\xspace}
\newcommand{\fkD}{\ensuremath{\mathfrak{D}}\xspace}
\newcommand{\fkE}{\ensuremath{\mathfrak{E}}\xspace}
\newcommand{\fkF}{\ensuremath{\mathfrak{F}}\xspace}
\newcommand{\fkG}{\ensuremath{\mathfrak{G}}\xspace}
\newcommand{\fkH}{\ensuremath{\mathfrak{H}}\xspace}
\newcommand{\fkI}{\ensuremath{\mathfrak{I}}\xspace}
\newcommand{\fkJ}{\ensuremath{\mathfrak{J}}\xspace}
\newcommand{\fkK}{\ensuremath{\mathfrak{K}}\xspace}
\newcommand{\fkL}{\ensuremath{\mathfrak{L}}\xspace}
\newcommand{\fkM}{\ensuremath{\mathfrak{M}}\xspace}
\newcommand{\fkN}{\ensuremath{\mathfrak{N}}\xspace}
\newcommand{\fkO}{\ensuremath{\mathfrak{O}}\xspace}
\newcommand{\fkP}{\ensuremath{\mathfrak{P}}\xspace}
\newcommand{\fkQ}{\ensuremath{\mathfrak{Q}}\xspace}
\newcommand{\fkR}{\ensuremath{\mathfrak{R}}\xspace}
\newcommand{\fkS}{\ensuremath{\mathfrak{S}}\xspace}
\newcommand{\fkT}{\ensuremath{\mathfrak{T}}\xspace}
\newcommand{\fkU}{\ensuremath{\mathfrak{U}}\xspace}
\newcommand{\fkV}{\ensuremath{\mathfrak{V}}\xspace}
\newcommand{\fkW}{\ensuremath{\mathfrak{W}}\xspace}
\newcommand{\fkX}{\ensuremath{\mathfrak{X}}\xspace}
\newcommand{\fkY}{\ensuremath{\mathfrak{Y}}\xspace}
\newcommand{\fkZ}{\ensuremath{\mathfrak{Z}}\xspace}

\newcommand{\heart}{{\heartsuit}}

\newcommand{\club}{{\clubsuit}}
\newcommand{\nat}{{\natural}}

\newcommand{\shp}{{\sharp}}
\newcommand{\diam}{{\Diamond}}

\newcommand{\spade}{{\spadesuit}}

\newcommand{\bA}{\mathbf A}
\newcommand{\bE}{\mathbf E}
\newcommand{\bG}{\mathbf G}
\newcommand{\bK}{\mathbf K}
\newcommand{\bM}{\mathbf M}
\newcommand{\bQ}{\mathbf Q}

\newcommand{\BA}{\ensuremath{\mathbb {A}}\xspace}
\newcommand{\BB}{\ensuremath{\mathbb {B}}\xspace}
\newcommand{\BC}{\ensuremath{\mathbb {C}}\xspace}
\newcommand{\BD}{\ensuremath{\mathbb {D}}\xspace}
\newcommand{\BE}{\ensuremath{\mathbb {E}}\xspace}
\newcommand{\BF}{\ensuremath{\mathbb {F}}\xspace}
\newcommand{{\BG}}{\ensuremath{\mathbb {G}}\xspace}
\newcommand{\BH}{\ensuremath{\mathbb {H}}\xspace}
\newcommand{\BI}{\ensuremath{\mathbb {I}}\xspace}
\newcommand{\BJ}{\ensuremath{\mathbb {J}}\xspace}
\newcommand{{\BK}}{\ensuremath{\mathbb {K}}\xspace}
\newcommand{\BL}{\ensuremath{\mathbb {L}}\xspace}
\newcommand{\BM}{\ensuremath{\mathbb {M}}\xspace}
\newcommand{\BN}{\ensuremath{\mathbb {N}}\xspace}
\newcommand{\BO}{\ensuremath{\mathbb {O}}\xspace}
\newcommand{\BP}{\ensuremath{\mathbb {P}}\xspace}
\newcommand{\BQ}{\ensuremath{\mathbb {Q}}\xspace}
\newcommand{\BR}{\ensuremath{\mathbb {R}}\xspace}
\newcommand{\BS}{\ensuremath{\mathbb {S}}\xspace}
\newcommand{\BT}{\ensuremath{\mathbb {T}}\xspace}
\newcommand{\BU}{\ensuremath{\mathbb {U}}\xspace}
\newcommand{\BV}{\ensuremath{\mathbb {V}}\xspace}
\newcommand{\BW}{\ensuremath{\mathbb {W}}\xspace}
\newcommand{\BX}{\ensuremath{\mathbb {X}}\xspace}
\newcommand{\BY}{\ensuremath{\mathbb {Y}}\xspace}
\newcommand{\BZ}{\ensuremath{\mathbb {Z}}\xspace}

\newcommand{\CA}{\ensuremath{\mathcal {A}}\xspace}
\newcommand{\CB}{\ensuremath{\mathcal {B}}\xspace}
\newcommand{\CC}{\ensuremath{\mathcal {C}}\xspace}
\newcommand{\CD}{\ensuremath{\mathcal {D}}\xspace}
\newcommand{\CE}{\ensuremath{\mathcal {E}}\xspace}
\newcommand{\CF}{\ensuremath{\mathcal {F}}\xspace}
\newcommand{\CG}{\ensuremath{\mathcal {G}}\xspace}
\newcommand{\CH}{\ensuremath{\mathcal {H}}\xspace}
\newcommand{\CI}{\ensuremath{\mathcal {I}}\xspace}
\newcommand{\CJ}{\ensuremath{\mathcal {J}}\xspace}
\newcommand{\CK}{\ensuremath{\mathcal {K}}\xspace}
\newcommand{\CL}{\ensuremath{\mathcal {L}}\xspace}
\newcommand{\CM}{\ensuremath{\mathcal {M}}\xspace}
\newcommand{\CN}{\ensuremath{\mathcal {N}}\xspace}
\newcommand{\CO}{\ensuremath{\mathcal {O}}\xspace}
\newcommand{\CP}{\ensuremath{\mathcal {P}}\xspace}
\newcommand{\CQ}{\ensuremath{\mathcal {Q}}\xspace}
\newcommand{\CR}{\ensuremath{\mathcal {R}}\xspace}
\newcommand{\CS}{\ensuremath{\mathcal {S}}\xspace}
\newcommand{\CT}{\ensuremath{\mathcal {T}}\xspace}
\newcommand{\CU}{\ensuremath{\mathcal {U}}\xspace}
\newcommand{\CV}{\ensuremath{\mathcal {V}}\xspace}
\newcommand{\CW}{\ensuremath{\mathcal {W}}\xspace}
\newcommand{\CX}{\ensuremath{\mathcal {X}}\xspace}
\newcommand{\CY}{\ensuremath{\mathcal {Y}}\xspace}
\newcommand{\CZ}{\ensuremath{\mathcal {Z}}\xspace}

\newcommand{\RA}{\ensuremath{\mathrm {A}}\xspace}
\newcommand{\RB}{\ensuremath{\mathrm {B}}\xspace}
\newcommand{\RC}{\ensuremath{\mathrm {C}}\xspace}
\newcommand{\RD}{\ensuremath{\mathrm {D}}\xspace}
\newcommand{\RE}{\ensuremath{\mathrm {E}}\xspace}
\newcommand{\RF}{\ensuremath{\mathrm {F}}\xspace}
\newcommand{\RG}{\ensuremath{\mathrm {G}}\xspace}
\newcommand{\RH}{\ensuremath{\mathrm {H}}\xspace}
\newcommand{\RI}{\ensuremath{\mathrm {I}}\xspace}
\newcommand{\RJ}{\ensuremath{\mathrm {J}}\xspace}
\newcommand{\RK}{\ensuremath{\mathrm {K}}\xspace}
\newcommand{\RL}{\ensuremath{\mathrm {L}}\xspace}
\newcommand{\RM}{\ensuremath{\mathrm {M}}\xspace}
\newcommand{\RN}{\ensuremath{\mathrm {N}}\xspace}
\newcommand{\RO}{\ensuremath{\mathrm {O}}\xspace}
\newcommand{\RP}{\ensuremath{\mathrm {P}}\xspace}
\newcommand{\RQ}{\ensuremath{\mathrm {Q}}\xspace}
\newcommand{\RR}{\ensuremath{\mathrm {R}}\xspace}
\newcommand{\RS}{\ensuremath{\mathrm {S}}\xspace}
\newcommand{\RT}{\ensuremath{\mathrm {T}}\xspace}
\newcommand{\RU}{\ensuremath{\mathrm {U}}\xspace}
\newcommand{\RV}{\ensuremath{\mathrm {V}}\xspace}
\newcommand{\RW}{\ensuremath{\mathrm {W}}\xspace}
\newcommand{\RX}{\ensuremath{\mathrm {X}}\xspace}
\newcommand{\RY}{\ensuremath{\mathrm {Y}}\xspace}
\newcommand{\RZ}{\ensuremath{\mathrm {Z}}\xspace}

\newcommand{\ab}{{\mathrm{ab}}}
\newcommand{\Ad}{{\mathrm{Ad}}}
\newcommand{\alb}{{\mathrm{alb}}}

\newcommand{\Br}{{\mathrm{Br}}}

\newcommand{\cay}{\ensuremath{\operatorname{\fkc_\xi}}\xspace}
\newcommand{\Ch}{{\mathrm{Ch}}}
\newcommand{\cod}{{\mathrm{cod}}}
\newcommand{\cl}{{\mathrm{cl}}}
\newcommand{\Cl}{{\mathrm{Cl}}}
\newcommand{\cm}{{\mathrm {cm}}}
\newcommand{\corr}{\mathrm{corr}}

\newcommand{\del}{\operatorname{\partial Orb}}
\newcommand{\disc}{{\mathrm{disc}}}
\newcommand{\Div}{{\mathrm{Div}}}
\renewcommand{\div}{{\mathrm{div}}}
\newcommand{\DR}{\mathrm{DR}}

\newcommand{\wt}{\widetilde}
\newcommand{\wh}{\widehat}
\newcommand{\pp}{\frac{\partial\ov\partial}{\pi i}}
\newcommand{\pair}[1]{\langle {#1} \rangle}
\newcommand{\wpair}[1]{\left\{{#1}\right\}}
\newcommand{\intn}[1]{\left( {#1} \right)}
\newcommand{\norm}[1]{\|{#1}\|}
\newcommand{\sfrac}[2]{\left( \frac {#1}{#2}\right)}
\newcommand{\ds}{\displaystyle}
\newcommand{\ov}{\overline}
\newcommand{\incl}{\hookrightarrow}
\newcommand{\lra}{\longrightarrow}
\newcommand{\imp}{\Longrightarrow}
\newcommand{\lto}{\longmapsto}
\newcommand{\bs}{\backslash}

\newcommand{\uF}{\underline{F}}
\newcommand{\ep}{\varepsilon}

%%% some additional macros

\newcommand{\nass}{\noalign{\smallskip}}
\newcommand{\htt}{h}
\newcommand{\cutter}{\medskip\medskip \hrule \medskip\medskip}

% Equation  \AMSname
% Theorem   \theoremname

\setcounter{footnote}{0}
\setcounter{equation}{0}

\pagestyle{plain}
\appendix
\section{\large{Accessible and weakly accessible period domains}\\
\vspace{0.3in}
\small{By Michael Rapoport}}
\vspace{0.5in}

\subsection{Introduction}

The goal of this appendix is to investigate in which situations the period maps from RZ spaces towards partial flag varieties are surjective. This question can be posed in two variants: One can either ask that the map is surjective on classical points, or surjective on {\it all} (adic, or equivalently, Berkovich) points. These questions can be translated into the question whether the weakly admissible, resp. admissible, locus inside the partial flag variety is the whole partial flag variety. We answer both of these questions below. It turns out that asking surjectivity for all points is significantly more restrictive, and occurs essentially only in the Lubin-Tate case.

Most of the material presented in this appendix was explained to the author by P.~Scholze. Moreover, we thank S.~Orlik for helpful conversations.

\subsection{Recollections on period domains}
Let $(G, b, \{\mu\})$ be a {\it PD-triple}\footnote{ In \cite[Ex. 9.1.22]{DOR}, to $(G, b)$ is associated an {\it augmented affine group scheme} $\BG$ over the category of $F$-isocrystals, and in \cite[Def. 9.5.1]{DOR} one considers the {\it PD-pair} associated to $(\BG, \{\mu\})$, rather than the triple $(G, b, \{\mu\})$.} over the $p$-adic field $F$. This means that $G$ is a reductive algebraic group over $F$, that $b\in G(\breve F)$, and that $\{\mu\}$ is a conjugacy class of cocharacters of $G$. We will assume throughout that $\{\mu\}$ is minuscule. Two PD-triples $(G, b, \{\mu\})$ and $(G', b', \{\mu'\})$ are called equivalent if there is an isomorphism $G\simeq G'$ which takes $ \{\mu\}$ into $ \{\mu'\}$ and $b$ into a $\sigma$-conjugate of $b'$. All concepts below depend only on the equivalence class of PD-triples. 
Let $E=E(G, \{\mu\})$ be the corresponding reflex field. We denote by $\CF(G, \{\mu\})$ the corresponding partial flag variety defined over $E$, and by $\breve \CF(G, \{\mu\})$ its base change to $\breve E$. We denote by $ \CF(G, \{\mu\})^{\rm wa}$ the {\it period domain} associated to the PD-triple $(G, b, \{\mu\})$, i.e., the {\it weakly admissible} subset of $\breve \CF(G, \{\mu\})$, which we consider as an open adic subset. It is defined by the weak admissibility condition of Fontaine on the Lie algebra of $G$ ({\it semi-stability}, cf. \cite[Def. 9.2.14]{DOR}) and the triviality of the {\it degree} in $\pi_1(G)_{\Gamma, \BQ}$.

\begin{definition}
A PD-triple $(G, b, \{\mu\})$ is {\it weakly accessible} if $\CF(G, b, \{\mu\})^{\rm wa}=\breve \CF(G, \{\mu\})$, i.e., the period domain associated to $(G, b, \{\mu\})$  is the whole partial flag variety. 
\end{definition}

\subsection{The admissible set}
Let $X_F$ be the Fargues-Fontaine curve relative to $F$ (and some fixed algebraically closed perfectoid field of characteristic $p$). By Fargues, \cite{FarguesGBun}, there is a bijection 
\begin{equation}\label{BversusGb}
B(G)\to \big\{\text{$G$-bundles on $X_F$}\big\}/\simeq , \quad b\mapsto \CE_b.
\end{equation}
Restricted to basic elements, this yields even an equivalence of groupoids,
$$
G(\breve F)_{\rm basic}\to \big\{\text{semi-stable $G$-bundles on $X_F$}\big\} .
$$
Here the LHS becomes a groupoid via the action by $\sigma$-conjugacy of $G(\breve F)$. Also, a $G$-bundle $\CE$ is called {\it semi-stable} if for all $\rho\in {\rm Rep}_G$ mapping the center of $G$ into the center of $\GL_n$, the vector bundle $\rho_*(\CE)$ on $X_F$ is semi-stable in the sense of Mumford (recall that $\deg$ and ${\rm rank}$ are well-defined for vector bundles on $X_F$). It is enough to check this for $\rho$ the adjoint representation of $G$.
\begin{definition} Fix a PD-triple $(G, b, \{\mu\})$ over $F$. Let $C$ be an algebraically  closed non-archimedean field extension of $\breve F$, and use the tilt $C^\flat$ of $C$ to build $X_F$; denote by $\infty\in X_F(C)$ the corresponding distinguished point of $X_F$.

To any point $x\in \CF(G, \{\mu\})(C)$, there is associated a $G$-bundle $\CE_{b, x}$ on $X_F$ which is called the {\it modification of $\CE_b$ at $\infty$ along $x$}.

\end{definition}
\begin{remark} 
 If $\CE$ is a vector bundle of rank $n$ on $X_F$, and $\{\mu\}$ is a minuscule cocharacter class of $\GL_n$, then it is clear how to define the modification $\CE_x$ for $x\in \CF(\GL_n, \{\mu\})(C)$. On the other hand, for non-minuscule $\{\mu\}$, or general $G$ (and then even for minuscule cocharacters), it is nontrivial  to define the modification $\CE_{b, x}$. Indeed, the definition involves the $B_{\rm dR}$-Grassmannian ${\rm Gr}_{G}^{B^{+}_{\rm dR}}$. One uses the {\it Bialynicki-Birula morphism}, valid for any $\{\mu\}$,
$$
{\rm Gr}_{G, \{\mu\}}^{B^{+}_{\rm dR}}\to \breve\CF(G, \{\mu\}) ,
$$
which is an isomorphism if $\{\mu\}$ is minuscule. We refer to \cite{CaraianiScholze} for a precise discussion of this point. We note however that on points defined over a finite extension of $\breve{F}$, the Bialynicki-Birula morphism is a bijection (for all $\{\mu\}$).
\end{remark}

\begin{definition}
 A point $x\in \CF(G, \{\mu\})(C)$ is called {\it admissible} with respect to $b$ if the  associated $G$-bundle $\CE_{b, x}$ is semi-stable. Equivalently, the image of $\CE_{b, x}$ under the map in Corollary \ref{modbasic} is the unique basic class $[b^*]$ with $\kappa([b^*])=\kappa([b])-\mu^\natural$. 
\end{definition}
\begin{remarks}\label{rem-a}
(i) An admissible point $x\in \CF(G, \{\mu\})(C)$ is automatically weakly admissible. If $x$ is defined over a finite extension of $\breve F$, the converse is true. For points defined over finite extensions of $\breve{F}$, these assertions can be reduced to the case of $\GL_n$ by using the adjoint representation, for which see \cite{ColmezFontaine}. Now the admissible locus is an open subset of $\CF(G,\{mu\})$ (cf. below) which on classical points agrees with the weakly admissible locus. As the weakly admissible locus is maximal among open subsets with given classical points, it follows that the admissible locus is contained in the weakly admissible locus.

(ii) Assume  that $(G, \{\mu\})\subset (\GL_n, \{\mu_{(1^{(r)}, 0^{(n-r)})}\})$, i.e., the PD-triple $(G, b, \{\mu\})$ is of Hodge type. Then Faltings and Hartl have defined the notion of admissibility of a point in $\CF(G, \{\mu\})(C)$, cf. \cite[ch. XI, \S4]{DOR} (Faltings' definition uses base change to $B_{\rm cris}(C)$; Hartl's definition uses the Robba ring $\tilde B_{\rm rig}^\dagger(C)$; Hartl has shown that these definitions coincide, comp. \cite[Thm. 11.4.11]{DOR}). The  definition of admissibility above specializes in this case to their definition. 
\end{remarks}

\begin{definition} Fix a PD-triple $(G, b, \{\mu\})$ over $F$. The {\it admissible locus} $\CF(G, b, \{\mu\})^{\rm a}$ is the unique open adic subset of $\breve\CF(G, \{\mu\})$ whose $C$-valued points are the admissible points of $\CF(G, \{\mu\})(C)$, for any algebraically  closed non-archimedean field extension of $\breve F$. 
\end{definition}
It follows from \cite{KedlayaLiu} that the admissible set is indeed an open adic subset of $\breve\CF(G, \{\mu\})$, again using the adjoint representation of $G$ to reduce to the case $G=\GL_n$.
\begin{remarks}\label{LTDr}
Whereas we have a fairly accurate picture of what the {\it weakly admissible locus} looks like (and one of the main attractions of the corresponding theory is to determine explicitly this locus in specific cases, cf. \cite[Ch. I]{RZ}), the {\it admissible locus} seems quite amorphous, and is explicitly known in only very few cases. Here are two examples.

(i) Let $(G, b, \{\mu\})=(\GL_n, b, \{\mu_{(1^{(1)}, 0^{(n-1)})}\})$, where $[b]$ is the unique basic element of $B(G, \{\mu\})$. This case is called the {\it Lubin-Tate case}. In this case, all points of $\breve\CF(G, \{\mu\})$ are admissible. This follows by Gross/Hopkins \cite{GH} from  Theorem \ref{imofper} below. Another, more direct, proof is due to Hartl, comp. \cite[Prop. 11.4.14]{DOR}. The same holds  for $(\GL_n, b, \{\mu_{(1^{(n-1)}, 0^{(1)})}\})$, where again $[b]$ is the unique basic element of $B(G, \{\mu\})$. 

(ii) Let $(G, b, \{\mu\})=(D_{\frac{1}{n}}, b, \{\mu_{(1^{(1)}, 0^{(n-1)})}\})$, where $[b]$ is the unique basic element of $B(G, \{\mu\})$. This case is called the {\it Drinfeld case}. In this case, all weakly admissible points of $\breve\CF(G, \{\mu\})$ are admissible.  They form the Drinfeld halfspace inside $\BP^{n-1}$. This follows by Faltings' theorem  \cite[ch. 5]{RZ} from  Theorem \ref{imofper} below, but has also been shown by Hartl, comp. \cite[Prop. 11.4.14]{DOR}. The same holds  for $(D_{-\frac{1}{n}}, b, \{\mu_{(1^{(n-1)}, 0^{(1)})}\})$, where again $[b]$ is the unique basic element of $B(G, \{\mu\})$. 
\end{remarks}

\begin{definition}
A PD-triple $(G, b, \{\mu\})$ is {\it accessible} if $\CF(G, b, \{\mu\})^{\rm a}=\breve\CF(G, b, \{\mu\})$, i.e., the admissible set associated to $(G, b, \{\mu\})$  is the whole partial flag variety. 
\end{definition}

From Remarks \ref{rem-a}, (i)  it follows that an accessible PD-triple is  weakly accessible. 

\begin{proposition}\label{modtriv}
 Associating to a $G$-bundle its isomorphism class, we obtain from \eqref{BversusGb} a bijection
\begin{equation*}
\big\{\text{ iso-classes of $G$-bundles of the form $\CE_{1, x}\mid x\in \CF(G, \{\mu^{-1}\})$}\big\} \to B(G, \{\mu\}) .  
\end{equation*} 
\end{proposition}
\begin{proof}
 Let  $b\in G(\breve F)$. If $[b]$ lies in the image of the map, it follows from the construction of $\CE_{1, x}$ that $\kappa([b])=\mu^\natural$ in $\pi_1(G)_\Gamma$. Now $b$ represents an element of the image of the map if and only if $\CE_b$ is of the form $\CE_{1, x}$; equivalently, if and only if $\CE_{b, x^*}$ is the trivial $G$-bundle for some $x^*\in \CF(G, \{\mu\})$. In other words, this holds if and only if there exists $x^*$ such that $\CE_{b, x^*}$ is a semi-stable $G$-bundle. Hence this is equivalent to $\CF(G, b, \{\mu\})^{\rm a}\neq \emptyset$. This in turn is equivalent to the condition that $\CF(G, b, \{\mu\})^{\rm wa}\neq\emptyset$, as these are two open sets with the same classical points.  By \cite[Thm. 9.5.10]{DOR} this is equivalent to $[b]\in A(G, \{\mu\})$. Since we saw already the equality $\kappa(b)=\mu^{ \natural}$, this  is equivalent to $[b]\in B(G, \{\mu\})$.
\end{proof}

\begin{corollary}\label{modbasic} Let $b\in G(\breve F)$ be basic. Then there is a bijection 
\begin{equation*}
\big\{\text{ iso-classes of $G$-bundles of the form $\CE_{b, x}\mid x\in \CF(G, \{\mu^{-1}\})$}\big\} \to B(J_b, \{\mu\}+\nu_b) .  
\end{equation*} 
Here $\nu_b$ is the central cocharacter associated to the basic element $b$.
\end{corollary}
\begin{proof}
This follows by translation with $b$ from the previous proposition, cf.  \cite[4.18]{Ko2}. Alternatively, one can apply the functor $\mathcal{H}om(\CE_b, \, )$ to the assertion of the corollary, to reduce to the previous proposition. 
\end{proof}

\subsection{Weakly accessible PD-triples}
Our first aim is to determine all weakly accessible PD-Pairs. The following lemma reduces this problem to the core cases. We always make the assumption that the period domain associated to any PD-triple considered below is non-empty. 
\begin{lemma}\label{redlemma}
\begin{altenumerate}
\item $(G, b, \{\mu\})$ is weakly accessible if and only if $(G_\ad, b_\ad, \{\mu_\ad\})$ is weakly accessible.
\item $\big(G_1\times G_2, (b_1, b_2), \{(\mu_1, \mu_2\}\big)$ is weakly accessible if and only if $(G_1, b_1, \{\mu_1\})$  and $(G_2, b_2, \{\mu_2\})$ are both weakly accessible.
\item If $\{\mu\}$ is central, then $(G, b, \{\mu\})$ is weakly accessible. 
\end{altenumerate}
\end{lemma}
\begin{proof} (i) Let $\pi: \breve \CF(G, \{\mu\}) \to\breve \CF(G_\ad, \{\mu_\ad\}) $ denote the natural morphism. Then the assertion follows from 
$$
 \CF(G, \{\mu\})^{\rm wa}=\pi^{-1}\big(\CF(G_\ad, \{\mu_\ad\})^{\rm wa}\big)
$$
(recall that we are assuming both period domains to be non-empty).

Finally, (ii) and (iii) are obvious. 
\end{proof}

After the previous reduction steps, the following proposition gives the complete classification of all weakly accessible PD-triples. 
\begin{proposition}\label{waisall}
Let $(G, b, \{\mu\})$ be a PD-triple defining a non-empty period domain, where $G$ is $F$-simple adjoint and $\{\mu\}$ is non-trivial. Then the PD-triple  $(G, b, \{\mu\})$ is weakly accessible  if and only if  the $F$-group $J_b$ is anisotropic, in which case $[b]$ is basic. 
\end{proposition}
\begin{proof} We note that, $G$ being of adjoint type, weak admissibility is equivalent to semi-stability in the sense of \cite{DOR}, i.e, $ \CF(G, b, \{\mu\})^{\rm wa}= \CF(G, b, \{\mu\})^{\rm ss}$, cf. \cite[top of p.272]{DOR}. We also note that the last sentence follows because if  $J$ is anisotropic, then $b$ is basic. Indeed, if $b$ is not basic, then the slope vector $\nu_b$ is a non-trivial cocharacter of $J$ defined over $F$, cf. \cite[after (3.4.1)]{Ko2}. 

Assume that there exists a point $x\in  \CF(G, \{\mu\})\setminus  \CF(G, b, \{\mu\})^{\rm ss}$. Then, applying \cite[Thm. 9.7.3]{DOR}, we obtain a 1-PS $\lambda$ of $J_b$ defined over $F$ which violates the Hilbert-Mumford inequality. In particular, $\lambda$ is non-trivial, and $J_b$ is not anisotropic.

Conversely, assume that $ \CF(G, b, \{\mu\})^{\rm ss}= \CF(G, \{\mu\})$. We claim that then $J_b$ is anisotropic. To prove this, we may change $b$ within its $\sigma$-conjugacy class $[b]$, since this leaves the isomorphism class of $J_b$ unchanged. We argue by contradiction. So, let us assume that $T$ is a maximal torus of $J_b$ such that $X_*(T)^\Gamma\neq (0)$. Here $\Gamma=\Gal (\bar F/F)$. Then $T\otimes_F\breve F$ is also a maximal torus of $G\otimes_F\breve F$. By assumption, for any $\mu\in X_*(T)$ defining an element $x\in\CF(G, \{\mu\})$, the pair $(b, \CF_x)$ is semi-stable. To apply the Hilbert-Mumford inequality, we fix an invariant inner product $(\, ,\, )$ on $G$, cf. \cite[Def. 6.2.1]{DOR}. Hence by the Hilbert-Mumford inequality \cite[Thm. 9.7.3]{DOR}, we obtain
$$
( \lambda, \mu-\nu_b)\geq 0 ,\quad \forall \lambda\in X_*(T)^\Gamma ,
$$
where $\nu_b\in X_*(T)_\BQ$ denotes the slope vector of $b$. Indeed, the LHS is equal to $\mu^\CL(x, \lambda)$, by \cite[Lemma 11.1.3]{DOR} (in loc.~cit., the situation over a finite field is considered; but the lemma holds in the present situation {\it mutatis mutandum}). Replacing $\lambda$ by its negative, we see that 
$(\lambda, \mu-\nu_b )=0$. Hence $(\lambda, \mu )$ is independent of $\mu\in X_*(T)$ in its geometric conjugacy class. It follows that for any $w, w'$ in the geometric Weyl group $W$ of $T$ in $G$,
\begin{equation}\label{lambdaid}
(\lambda, w\mu-w'\mu) =0 .
\end{equation}
We wish to show that this implies that $\lambda=0$, which would yield the desired contradiction. We write $G=\Res_{F'/F}(G')$, where $G'$ is an absolutely simple adjoint group over the extension field $F'$ of $F$. Let $F'_0$ be the maximal unramified subextension of $F'/F$. Then
\begin{equation}\label{extprod}
G(\breve F)=\prod_{i\in \BZ/f\BZ} G'(\breve F') ,
\end{equation}
where $\BZ/f\BZ$ denotes the Galois group of $F'_0/F$, and where $\breve F$, resp. $\breve F'$, denotes the completion of the maximal unramified extension of $F$, resp. $F'$. Furthermore, it is easy to see that any $b\in G(\breve F)$ is $\sigma$-conjugate to an element in the product on the RHS of \eqref{extprod} of the form $(b'_0, 1,\ldots, 1)$, and that then
\begin{equation*}
J_b=\Res_{F'/F} J'_{b'_0} .
\end{equation*}
Correspondingly, $T=\Res_{F'/F}(T')$, where $T'$ is a maximal torus of $J'_{b'_0}$ defined over $F'$. Hence  
\begin{equation}\label{prodtor}
X_*(T)_\BQ=\prod_{\tau\in\Hom_F(F', \bar F)} X_*(T')_\BQ ,
\end{equation}
with its action by $\Gamma$ induced by the action of $\Gamma'=\Gal(\bar F/F')$ on $X_*(T')_\BQ$. Since $0\neq \lambda\in X_*(T)^\Gamma$, all components $\lambda_\tau$  of $\lambda$ in the product decomposition \eqref{prodtor} are non-zero, and are determined by any one of them. Now $T'\otimes_{F'} \breve F'$ is a maximal torus of $G'\otimes_{F'} \breve F'$ and, since $G'$ is absolutely simple,  its geometric Weyl group  $W'$ acts irreducibly on $X_*(T')_\BQ$,  cf. \cite[ Cor. of Prop. 5 in VI, \S 1.2]{Bourb}. Furthermore, the geometric Weyl group of $T$ is the product of copies of $W'$ over the same index set as in \eqref{prodtor}. Hence the identity \eqref{lambdaid} implies that any time the component $\mu_\tau$ of $\mu$ is non-trivial, the component $\lambda_\tau$ is zero. Hence the assumption $\lambda\in X_*(T)^\Gamma$ implies $\lambda=0$, since the assumption  $\mu\neq 0$ implies that $\mu_\tau\neq 0$ for some $\tau$. This yields the  desired contradiction. 
\end{proof}
\begin{corollary} In Proposition \ref{waisall}, assume that   $G$ is absolutely simple adjoint and that $\{\mu\}$ is non-trivial. Then $(G, b, \{\mu\})$ satisfies the condition of Proposition  \ref{waisall} if and only if $G$ is the algebraic group associated to a simple central algebra $D$ of some rank $n^2$ over $F$,  $[b]$ is basic, and the difference between the  Hasse invariant of $D$ in $\BZ/n\BZ\simeq \pi_1(G)_\Gamma$ and the class  $\kappa([b])$ lies in $(\BZ/n)^\times$. \qed
\end{corollary}
\begin{remark} Note that the class $\{\mu\}$ does not intervene in Proposition  \ref{waisall}. It does, however, enter in the condition that the period domain $ \CF(G, b, \{\mu\})^{\rm wa}$ be non-empty. Indeed, this condition is equivalent to the condition that   $[b]\in A(G, \{\mu\})$, cf. \cite[Thm. 9.5.10]{DOR}, i.e., that $[b]$ be acceptable with respect to $\{\mu\}$ in the sense of \cite{RV}.

\end{remark}
\subsection{Accessible  PD-triples}
Here the classification is much more narrow.
\begin{proposition}
A PD-triple $(G, b, \{\mu\})$ is accessible if and only if $b$ is basic, and the pair $(J_b, \{\mu\})$ is uniform in the sense of \cite[\S 6]{Ko2}, i.e. $B(J_b,\{\mu\})$ contains precisely one element.
\end{proposition}

\begin{proof}
The accessibility of $(G, b, \{\mu\})$ implies its weak accessibility, cf. Remark \ref{rem-a}, (i); hence $b$ is basic by Proposition \ref{waisall}. The assumption that $(G, b, \{\mu\})$ is accessible is equivalent to saying that any modification $\CE_{b, x}$ for $x\in \CF(G, \{\mu\})$ is semi-stable. Hence, by Corollary \ref{modbasic}, the set $B( J_b, \{\mu^{-1}\}+\nu_b)$ contains only one element, i.e., $( J_b, \{\mu^{-1}\}+\nu_b)$ is uniform.  The assertion follows since $( J_b, \{\mu^{-1}\}+\nu_b)$ is uniform if and only if $( J_b, \{\mu^{-1}\})$ is uniform, if and only if $(J_b,\{\mu\})$ is uniform.
\end{proof}

Kottwitz \cite[\S6]{Ko2} has given a complete classification of uniform pairs $(G, \{\mu\})$. Applying his result, we obtain the following corollary.
\begin{corollary} Let $(G, b, \{\mu\})$ be a PD-triple. 
Assume that   $G$ is absolutely simple adjoint, that $\{\mu\}$ is non-trivial, and that  $[b]\in B(G, \{\mu\})$. Then $(G, b, \{\mu\})$ is accessible if and only if $G\simeq\PGL_n$, and $\{\mu\}$ corresponds to $(1, 0,\ldots, 0)$ or $(1, 1,\ldots,1, 0)$.  \qed
\end{corollary}

\subsection{An application to the crystalline period map}
Let $(G, b, \{\mu\})$ be a {\it local Shimura datum} over $F$, cf.  \cite{RV}, i.e., a PD-triple such that $\{\mu\}$ is  {\it minuscule}  and such that $[b]\in B(G, \{\mu\})$. Conjecturally, there is an associated local Shimura variety, i.e., a tower of rigid-analytic spaces over $\breve E$, with members enumerated by the open compact subgroups of $G(\BQ_p)$, 
\begin{equation}
\{\BM_K\}_K=\{\BM(G, b, \{\mu\})_K\}_K ,
\end{equation}
on which $G(\BQ_p)$ acts as Hecke correspondences. The tower comes with a compatible system of morphisms
\begin{equation}
\phi_K\colon \BM_K\to \breve \CF(G, \{\mu\}) .
\end{equation}
The morphism $\phi_K$ is called the {\it crystalline period morphism} at level $K$ of the local Shimura variety attached to $(G, b, \{\mu\})$.  
\begin{theorem}\label{imofper}
Assume that the local Shimura variety associated to $(G, b, \{\mu\})$ comes from an RZ-space of type EL or PEL, in which case the local Shimura variety exists. Then the  image of the crystalline period morphisms coincides with the admissible locus $\CF(G, b, \{\mu\})^{\rm a}$. 
\end{theorem}

\begin{proof}
See \cite{Hartl} (which uses \cite{Falt}) and \cite{SW}. 
\end{proof}

\begin{example} (i) In the Lubin-Tate case (see Remarks \ref{LTDr}, (i)), Gross and Hopkins \cite{GH} have shown that the image of the crystalline period morphism is the whole projective space $\CF(G, \{\mu\})$. 

(ii) In the Drinfeld case  (see Remarks \ref{LTDr}, (ii)), the image of the crystalline period map is the Drinfeld half-space, cf. \cite[ch. 5]{RZ}. 

\end{example}

\begin{corollary}
Assume that the local Shimura variety associated to $(G, b, \{\mu\})$ comes from an RZ-space of type EL or PEL, in which case the local Shimura variety exists. Also, assume that $G$ is absolutely simple. Then the crystalline period morphisms are surjective if and only if the local Shimura variety is of Lubin-Tate type.\qed 
\end{corollary}

\subsection{Open questions}

Here we list some open questions. 
\begin{question}
When is $\CF(G, b, \{\mu\})^{\rm a}=\CF(G, b, \{\mu\})^{\rm wa}$?
\end{question}
This question was answered by Hartl in the case when $G=\GL_n$. Besides the Lubin-Tate case and the Drinfeld case, there is one essentially new case related to $\GL_4$. B. Gross asks whether the PD-triples formed by an adjoint orthogonal group $G$, its natural minuscule coweight $\{\mu\}$ (the one attached to a Shimura variety for $\mathrm{SO}(n-2, 2)$) and the unique basic element in $B(G, \{\mu\})$ give further examples.

\smallskip

For the next question, recall that for any standard parabolic $P^*$ in the quasi-split form $G^*$ of $G$, there is a subset $B(G)_{P^*}$ defined in terms of the Newton map on $B(G)$. If  $P^*=G^*$, then $B(G)_{G^*}=B(G)_{\rm basic}$. We call the inverse image of $B(G)_{P^*}$ under the map in Corollary \ref{modbasic} the HN-stratum  $\CF(G, b, \{\mu\})_{P^*}$ attached to $P^*$. Hence for $P^*=G^*$ the corresponding HN-stratum is the admissible set. 

\begin{question}
For which $P^*$ is the HN-stratum non-empty? Does the decomposition into disjoint sets $\CF(G, b, \{\mu\})_{P^*}$ of $\breve \CF(G, \{\mu\})$ have the stratification property? Which strata $\CF(G, b, \{\mu\})_{P^*}$ have classical points?
\end{question}
The first question is non-empty, as is shown by the Lubin-Tate case, in which only $\CF(G, b, \{\mu\})_{G^*}$ is non-empty. There are examples of strata $\CF(G, b, \{\mu\})_{P^*}$ without classical points: One gets these by looking at cases of weakly accessible, but non-accessible, PD-triples, in which case all strata with $P^*\neq G^*$ have no classical points, but some of them are nonempty.

\smallskip

There is also a HN-decomposition of $\breve\CF(G, \{\mu\})$ in the sense of \cite{DOR}. It does not have the stratification property. Here we have an understanding of  the structure of the individual strata, in terms of period domains of PD-triples of smaller dimension. However, even for these simpler strata, the question of the non-emptiness of strata is only partially solved (by Orlik).  
\begin{question}
What is the relation between the two stratifications? 
\end{question}

\end{document}